\documentclass[1 [leqno,11pt]{amsart}
\usepackage{amssymb, amsmath,amsmath,latexsym,amssymb,amsfonts,amsbsy, amsthm}
\newcommand{\andf}{\quad\hbox{and}\quad}
\newcommand{\with}{\quad\hbox{with}\quad}
\def\Supp{\mathop{\rm Supp}\nolimits\ }
\newcommand{\newcom}{\newcommand}
\def\longformule#1#2{
\displaylines{ \qquad{#1} \hfill\cr \hfill {#2} \qquad\cr } }
\def\inte#1{
\displaystyle\mathop{#1\kern0pt}^\circ }
\newcom{\al}{\alpha}
\newcom{\be}{\beta}
\newcom{\s}{\sigma}
\newcom{\eps}{\epsilon}
\newcom{\ve}{\varepsilon}
\newcom{\ga}{\gamma}
\newcom{\Ga}{\Gamma}
\newcom{\ka}{\kappa}
\newcom{\Lam}{\Lambda}
\newcom{\lam}{\lambda}
\newcom{\vp}{\varphi}
\def\cR{{\mathcal R}}
\newcom{\Om}{\Omega}
\newcom{\om}{\omega}
\newcom{\Sig}{\Sigma}
\newcom{\sig}{\sigma}
\newcom{\tht}{\theta}
\newcom{\tri}{\triangle}
\newcom{\oo}{\infty}
\newcom{\vphi}{\varphi}
\newcom{\cB}{{\mathcal B}}
\newcom{\cC}{{\mathcal C}}
\newcom{\cD}{{\mathcal D}}
\newcom{\cF}{{\mathcal F}}
\newcom{\cL}{{\mathcal L}}
\newcom{\cM}{{\mathcal M}}
\newcom{\cP}{{\mathcal P}}
\newcom{\cS}{{\mathcal S}}
\newcom{\cQ}{{\mathcal Q}}
\newcom{\cT}{{\mathcal T}}
\newcom{\cY}{{\mathcal Y}}
\newcom{\cZ}{{\mathcal Z}}
\newcom{\R}{\Bbb R}
\newcom{\T}{\Bbb T}
\newcom{\N}{\Bbb N}
\newcom{\Z}{\Bbb Z}
\newcom{\C}{\Bbb C}
\newcom{\E}{\Bbb E}
\let\wh=\widehat
\def\dive{\mathop{\rm div}\nolimits}
\let\e=\varepsilon

\newcom{\f}{\frac}
\newcom{\dint}{\displaystyle\int}
\newcom{\dsum}{\displaystyle\sum}
\newcom{\dlim}{\displaystyle\lim}
\newcom{\ov}{\overline}
\newcom{\wt}{\widetilde}
\newcom{\pa}{\partial}
\newcom{\p}{\partial}
\newcom\na{\nabla}
\newcom{\D}{\Delta}
\newcom\rto{\rightarrow}
\newcom\lto{\leftarrow}
\newcom\mto{\mapsto}
\newcom{\disp}{\displaystyle}
\newcom{\non}{\nonumber}
\newcom{\no}{\noindent}
\newcom{\QED}{$\square$}

\def\ef{\hphantom{MM}\hfill\llap{$\square$}\goodbreak}

\def\eqdefa{\buildrel\hbox{\footnotesize def}\over =}

\newcommand{\beq}{\begin{equation}}
\newcommand{\eeq}{\end{equation}}
\newcommand{\ben}{\begin{eqnarray}}
\newcommand{\een}{\end{eqnarray}}
\newcommand{\beno}{\begin{eqnarray*}}
\newcommand{\eeno}{\end{eqnarray*}}

\newtheorem{thm}{Theorem}[section]
\newtheorem{lem}{Lemma}[section]
\newtheorem{rmk}{Remark}[section]
\newtheorem{col}{Corollary}[section]
\newtheorem{prop}{Proposition}[section]
\renewcommand{\theequation}{\thesection.\arabic{equation}}

\newtheorem{theorem}{Theorem}[section]
\newtheorem{definition}[theorem]{Definition}
\newtheorem{lemma}[theorem]{Lemma}

\newtheorem{Theorem}{Theorem}[section]

\newtheorem{Remark}[Theorem]{Remark}

\begin{document}
\title[Global wellposedness of 3-D inhomogeneous NS equations]
{\small Global well-posedness of 3-D inhomogeneous Navier-Stokes
equations with ill-prepared initial data}

\author[P. ZHANG]{Ping Zhang}%
\address[P. ZHANG]
 {Academy of
Mathematics $\&$ Systems Science and  Hua Loo-Keng Key Laboratory of
Mathematics, The Chinese Academy of Sciences\\
Beijing 100190, CHINA } \email{zp@amss.ac.cn}
\author[Z. ZHANG]{Zhifei Zhang}\address[Z. ZHANG]
{School of  Mathematical Science, Peking University, Beijing 100871,
P. R. CHINA} \email{zfzhang@math.pku.edu.cn}

\date{July 18, 2014}%4.14

\maketitle
\begin{abstract} In this paper, we investigate the global well-posedness  of 3-D
incompressible inhomogeneous  Navier-Stokes equations with
ill-prepared large initial data which are slowly varying in  one
space variable, that is, initial data of the form
$\bigl(1+\e^{\be}a_0(x_{\rm h},\e x_3),(\ve^{1-\al} v^{\rm h}_0,
\ve^{-\al}v_0^3)(x_{\rm h},\e x_3)\bigr)$ for  any $\al\in ]0,1/3[,$
 $\be>2\al,$ and $\ve$ being sufficiently small. We remark that initial data of this type do not
satisfy the smallness conditions in \cite{c-p-z,HPZ3} no matter how
small $\e$ is. In particular, this result greatly improves the
global well-posedness result in \cite{PZZ3} with the so-called
well-prepared initial data.
\end{abstract}

\noindent {\sl Keywords:} Inhomogeneous  Navier-Stokes equations,
Littlewood-Payley theory,
well-posedness, ill-prepared data\\

\noindent {\sl AMS Subject Classification (2000):} 35Q30, 76D05 \\

\renewcommand{\theequation}{\thesection.\arabic{equation}}
\setcounter{equation}{0}
%%%%%%%%%%%%%%%%%%%%%%%%%%%%%%%%%%%%%%%%%%%%%%
%%%%%%%%%%%%%%%%%%%%%%%%%%%%%%%%%%%%%%%%%%
\section{Introduction}

In this paper, we consider the global well-posedness of the
following incompressible inhomogeneous Navier-Stokes equations in
$\R^3$
\begin{equation}\label{eq:InhomoNS}
\left\{
\begin{array}{ll}
\p_t\rho+u\cdot\na\rho=0, \qquad (t,x)\in\R^+\times\R^3, \\
\rho(\p_tu+u\cdot\na u)-\Delta u+\na p=0, \\
\textrm{div} u=0,\\
(\rho,u)|_{t=0}=(\rho_{0},u_{0}),
\end{array}
\right.
\end{equation}
where $\rho, u=(u_1,u_2, u_3)$ stand for the density and  velocity
of the fluid respectively,  $p$ is a scalar pressure function.
  Such system describes a fluid which is obtained by
mixing two immiscible fluids that are incompressible and that have
different densities. It may also describe a fluid containing a
melted substance.

\smallskip When the initial density is away from zero, we denote by
$a\eqdefa\frac{1}{\rho}-1,$ and then \eqref{eq:InhomoNS} can be
equivalently reformulated as
\begin{equation}\label{1.3}
 \quad\left\{\begin{array}{l}
\displaystyle \p_t a + u \cdot \na a=0,
\hspace{1cm}(t,x)\in \R^+\times\R^3,\\
\displaystyle \p_t u + u \cdot \na u+ (1+a)(\na p-\Delta u)=0, \\
\displaystyle \dive\, u = 0, \\
\displaystyle (a, u)|_{t=0}=(a_0, u_{0}).
\end{array}\right.
\end{equation}
Notice that just as the classical Navier-Stokes system $(NS)$ (which
corresponds to the case when $a=0$ in \eqref{1.3}), the
inhomogeneous Navier-Stokes system (\ref{1.3}) also has a scaling.
Indeed if $(a, u)$ solves (\ref{1.3}) with initial data $(a_0,
u_0)$, then for $\forall \, \ell>0$,
\begin{equation}\label{1.2}
(a, u)_{\ell} \eqdefa (a(\ell^2\cdot, \ell\cdot), \ell u(\ell^2
\cdot, \ell\cdot))\quad\mbox{and}\quad (a_0,u_0)_\ell\eqdefa
(a_0(\ell\cdot),\ell u_0(\ell\cdot))
\end{equation}
 $(a, u)_{\ell}$ is also a solution of (\ref{1.3}) with initial data $(a_0,u_0)_\ell$.

\medbreak Lady\v zenskaja and Solonnikov \cite{LS} first established
the unique resolvability of (\ref{1.3}) in bounded domain $\Om$ with
homogeneous Dirichlet boundary condition for $u.$ Similar results
were obtained by Danchin \cite{danchin2} in $\R^d$ with initial data
in the almost critical (corresponding to the scaling in \eqref{1.2})
Sobolev spaces. In \cite{danchin}, Danchin studied in general space
dimension $d$ the unique solvability of the system (\ref{1.3}) with
initial data being small in the scaling invariant (or critical)
homogeneous Besov spaces. This result was extended to more general
Besov spaces  by Abidi in \cite{abidi}, and by Abidi, Paicu in
\cite{AP}. The smallness assumption on the initial density was
removed in \cite{AGZ2,AGZ3}.

\smallbreak
 Very recently,  Danchin and Mucha \cite{DM} noticed that
it was possible to establish existence \emph{and} uniqueness of a
solution to \eqref{eq:InhomoNS} in the case of a small discontinuity
for the initial density and  in a critical functional framework.
More precisely, the global existence and uniqueness was established
for any data $(\rho_0,u_0)$ which satisfies
\begin{equation}\label{eq:small}
\|\rho_0-1\|_{\cM( B^{-1+\frac dp}_{p,1}(\R^d))}+\|u_0\|_{
B^{-1+\frac dp}_{p,1}(\R^d)}\leq c.
\end{equation}  for some
$p\in[1,2d[$ and small enough constant $c,$  and where $\cM(
B^{-1+\frac dp}_{p,1}(\R^d))$ denotes the multiplier space of the
Besov space $B^{-1+\frac dp}_{p,1}(\R^d).$ One may check \cite{DM}
for details. Let us remark that the classical Navier-Stokes system
$(NS)$ has a unique global solution provided that the initial data
satisfy $\|u_0\|_{ B^{-1+\frac dp}_{p,\infty}(\R^d)}\leq c$ for any
$p\in ]1,\infty[$ (see \cite{cannonemeyerplanchon}). The restriction
of $p\in[1,2d[$ in \cite{DM} and the relevant references is due to
the appearance of the free transport equation in \eqref{1.3} and
thus need to deal with the product of $a$ with $\na p$ in the
velocity equation.

\smallbreak Whereas inspired  by results concerning the global
well-posedness of 3-D incompressible anisotropic Navier-Stokes
system with the third component of the initial velocity field being
large (see for instance \cite{PZ1}), Paicu and the first author
\cite{PZ2} relaxed the smallness condition in \cite{AP} so that
\eqref{1.3} still has a unique global solution provided that
\beq\label{small2} \Bigl(\|a_0\|_{B_{p,1}^{\f3p}}+\|u_0^{\rm
h}\|_{B^{-1+\frac3p}_{p,1}}\Bigr)\exp\Bigl( C_0
\|u_0^3\|_{B^{-1+\frac3p}_{p,1}}^2\ \Big)\leq c_0 \eeq for some
$c_0$ sufficiently small and $p\in ]1,6[.$ This smallness condition
\eqref{small2} was improved by Huang, Paicu and the first author in
\cite{HPZ3} to \beq\label{small3}
\Bigl(\|a_0\|_{L^\infty}+\|u_0^{\rm
h}\|_{B^{-1+\f{d}p}_{p,r}}\Bigr)\exp\Bigl(C_r \|u_0^d\|_{B^{-1+\f
dp}_{p,r}}^{2r}\Bigr)\leq c_0\eeq for some $p\in ]1,d[,$ $r\in
]1,\infty[$ and  in general $d$ space dimension.  We emphasize that
the proof in \cite{HPZ3,PZ2} used in a fundamental way the
algebraical structure of \eqref{1.3}, namely, $\dive u=0.$ The first
step is to obtain energy estimates on the horizontal components of
the velocity field on the one hand and then on the vertical
component on the other hand. Compared with \cite{PZ1}, the
additional difficulties with this strategy in \cite{ HPZ3, PZ2} are
that: there appears a hyperbolic type equation in \eqref{1.3} and
due to the appearance of $a$ in the momentum equation of
\eqref{1.3}, the pressure term is more difficult to be handled.

\medbreak On the other hand,  Chemin and Gallagher \cite{c-g}
initiated  the global large solutions of 3-D classical Navier-Stokes
system $(NS)$ with data which are slowly varying in one direction,
that is data of the form:
 \beno \bigl(v_0^{\rm h}+\e
u^{\rm h}_0, u_0^3\bigr)(x_h,\e x_3) \with x_{\rm h}=(x_1,x_2)\eeno
for smooth divergence free vector fields $v_0^{\rm h}$ and
$u_0=(u_0^{\rm h},u_0^3).$ The main idea behind the proof in
\cite{c-g} is that the solutions to 3-D Navier-Stokes equations
$(NS)$ slowly varying in one space variable can be well approximated
by solutions of 2-D Navier-Stokse equation. Yet just as the
classical 2-D Navier-Stokes system, 2-D inhomogeneous Navier-Stokes
equations is also globally well-posed with general initial data (see
\cite{danchin2, LS} for instance). This motivates  the authors
\cite{c-p-z} to prove the global well-posedness of \eqref{1.3} with
data of the form:
\begin{equation*}  a_0^{\ve}(x)=\ve^{\be} a_0(x_{\rm h},\ve x_3), \quad u_0^{\ve}(x)=(v_0^{\rm h}(x_{\rm h},\ve
x_3), 0)
\end{equation*}
for any $\be>1/4.$ Paicu and the first author \cite{PZZ3} proved the
global well-posedness of \eqref{1.3} with initial data of the form:
\begin{equation*}  a_0^{\ve}(x)=\ve^{\be} a_0(x_{\rm h},\ve x_3), \quad u_0^{\ve}(x)=(\ve u_0^{\rm h}, u_0^3)(x_{\rm h},\ve
x_3)
\end{equation*}
for any $\be>0.$

 \smallbreak Furthermore, for the classical Navier-Stokes system $(NS)$ with the so-called
ill-prepared data \beq \label{ill} (\ve^{1-\al} u_0^{\rm
h},\ve^{-\al} u_0^3)(x_{\rm h},\ve x_3), \eeq Chemin, Gallagher and
Paicu \cite{cgp} proved the global well-posedness of $(NS)$ in
$\R^2\times\Bbb{T}$  with initial data given by \eqref{ill} for
$\al=0$. Paicu and the second author \cite{mz1} proved the global
well-posedness of $(NS)$ in $\R^3$  with data given by \eqref{ill}
for $\al=\f1 2$. This result was improved lately by the authors in
\cite{mz2} for any $\al\in \bigl]\f1 2,1\bigr[.$ We remark that to
prove results in those relevant references, they may need to use
analytical type initial data and the tool developed by Chemin
\cite{Ch04} which consists in making analytic-type estimates and
controlling the size of the analyticity band simultaneously.

\medbreak Motivated by \cite{cgp, mz1, mz2}, we shall consider the
global solutions of  \eqref{eq:InhomoNS} with ill-prepared initial
data of the form \beq\label{data1}
\rho_{0}(x)=\overline{\rho}+\ve^\be a_0(x_{\rm h},\ve x_3),\quad
u_{0}(x)=\big(\ve^{1-\al}v^{\rm h}_0, \ve^{-\al}v^3_0\big)(x_{\rm
h},\ve x_3), \eeq
 where $\overline{\rho}$ is a positive constant,
 $ v^{\rm h}_0=(v^1_0,v^2_0)$ and $v_0=(v_0^{\rm h}, v_0^3)$ satisfies $\dive
 v_0=0.$ Of course, this type  data do not satisfy
the smallness conditions \eqref{small2} and \eqref{small3} no matter
how small $\e$ is.

Our main result in this paper  states as follows.

\begin{theorem}\label{thm:main}
{\sl Let $\delta>0$, $\al\in \bigl]0,\f13\bigr[, \be>2\al$ and
$\ga\in ]0,\ga_0[$ with $\ga_0\eqdefa \min\bigl(\f{\be-2\al}5,
\f{1-3\al}5\bigr)$. Let $a_0$ and the solenoidal vector field $v_0$
satisfy \beq\label{data2} \|(a_0,v_0)\|_X\eqdefa
\|a_0\|_{X_1}+\|v_0\|_{X_2}+\|v_0\|_{X_3}<\infty, \eeq where
\beq\label{i.1}
\begin{split}
&\|a_0\|_{X_1}\eqdefa\big\|e^{\delta|D|}a_0\big\|_{B^{1-\ga,\f12+\ga}}+\big\|e^{\delta|D|}a_0\big\|_{B^{1+\ga,\f12-\ga}}+\big\|e^{\delta|D|}a_0\big\|_{B^{\ga,\f32-\ga}},\\
&\|v_0\|_{X_2}\eqdefa\big\|e^{\delta|D|}v_0\big\|_{B^{-\f12+\ga,-\ga}}+\big\|e^{\delta|D|}v_0\big\|_{B^{0, -\f12}},\\
&\|v_0\|_{X_3}\eqdefa\big\|e^{\delta|D|}v_0\big\|_{B^{\ga,\f12-\ga}}
+\big\|e^{\delta|D|}v_0\big\|_{B^{-\ga,\f12+\ga}},\end{split} \eeq
Then there exists a small positive constant $\ve_0,$ which depends
on $\|(a_0,v_0)\|_X,$ such that for $\ve\leq \ve_0,$ the
inhomogeneous Navier-Stokes equations (\ref{eq:InhomoNS}) with
initial data given by \eqref{data1} has a unique global smooth
solution. }
\end{theorem}

\begin{rmk} The exact value of $\ve_0$ will be given by
\eqref{epthm}. In fact, we can also deduce from the proof of Theorem
\ref{thm:main} that there exists a positive constant $\eta$ such
that for any $a_0$ and divergence free vector field $v_0$ satisfying
\beno \|(a_0,v_0)\|_X\eqdefa
\|a_0\|_{X_1}+\|v_0\|_{X_2}+\|v_0\|_{X_3}\le \eta, \eeno the
inhomogeneous Navier-Stokes equations (\ref{eq:InhomoNS}) with
initial data given by \eqref{data1} has a unique global smooth
solution for any $\ve>0.$
\end{rmk}

Here the anisotropic Besov space $B^{\s,s}(\R^3)$ and all the other
functional framework will be presented  in the  next section.

\medbreak Let us remark that besides the difficulties caused by
proving global in time Cauchy-Kowalewskya type results  in
\cite{cgp, mz1, mz2} for the classical Navier-Stokes system $(NS)$,
here we shall encounter the following types of new difficulties:
\begin{itemize}
\item
Note that after the scaling transformation, we shall obtain a
inhomogeneous Navier-Stokes system \eqref{eq:InhomoNS-scaled} with
anisotropic dissipation $\D_{\rm h}+\ve^2\p_{3}^2$  and anisotropic
pressure gradient $-\na^\ve q$ for $\na^\ve=(\na_{\rm h},
\ve^2\p_3).$ To capture the subtle dissipation in this new system,
we shall use anisotropic Littlewood-Paley analysis, which has been
used successfully for both homogeneous and inhomogeneous
Navier-Stokes system \cite{bg, c-p-z, CZ5, mz1,mz2} lately. However
due to the appearance of the free transport equation in
\eqref{eq:InhomoNS}, the analyticity assumption only for the
vertical variable in \cite{cgp, mz1, mz2} will not be enough here.
Instead we shall consider the initial data which are analytic in all
the space variables. We emphasize once again that the algebraical
structural of the system \eqref{eq:InhomoNS-scaled} and the tool
developed by Chemin \cite{Ch04} will play also crucial roles in this
paper.

\item Since we can not use commutator's argument to deal with the
propagation of analytic regularity for transport equation, in order
to control the inhomogeneity $a_\Phi(t)$ in the critical anisotropic
Besov space $B^{1,\f12}(\R^3),$ we require the global in time $L^1$
estimate with values in Besov spaces, which are in the scalings of
both the space $B^{2,\f12}(\R^3)$ and in that of
$B^{1,\f12}_{2,1}(\R^3),$ for the convection velocity field.

\item However,
in order to control $\|v_\Phi\|_{L^1_t(B^{1,\f12})},$ we would
require the estimate of  $G(\ve^\be a)\na^\ve q$ in the space
$L^1_t(B^{-1.\f12})$ for  $G(r)\eqdefa\f{r}{1+r},$ which is
impossible due to product laws in two space dimensions. The idea to
overcome this difficulty is to use Lemma \ref{lem:parabolic-inhome}
so that we only need to handle the estimate of
$\bigl\|\bigl[G(\ve^\be a)\na^\ve
q\bigr]_\Phi\|_{L^1_t(B^{-1+\ga.\f12-\ga})}$ for some small
positive constant $\ga.$ This in turn would require the estimates of
$a_\Phi$ in $\wt{L}^\infty_t(B^{1-\ga,\f12+\ga})$ and in
$\wt{L}^\infty_t(B^{1+\ga,\f12-\ga}),$  and $v_0$ satisfying
$\|v_0\|_{X_2}$ being finite. This explains the reason why the data
in Theorem \ref{thm:main} are so much complicated.

\item As  in the proof of the global well-posdness of
inhomogeneous Navier-Stokes system with initial data in the critical
spaces, for instance in \cite{AGZ2,AGZ3,DM, HPZ3}, the pressure is
always  a big difficulty. We point out that the assumption for
$\be>2\al$ in Theorem \ref{thm:main} will only be used to handle the
estimates of $q_{31}$ in \eqref{p.3} and of $q_{41}$ in \eqref{p.4}.
Otherwise, the assumption for $\be>\al$ would be enough in Theorem
\ref{thm:main}.
\end{itemize}

\medbreak Let us end this introduction by the notations we shall use
in this context.\\

 For~$a\lesssim b$, we mean that there is a
uniform constant $C,$ which may be different on different lines but
be independent of $\ve,$ such that $a\leq Cb$. For $X$ a Banach
space and $I$ an interval of $\R,$ we denote by $C(I;\,X)$ the set
of continuous functions on~$I$ with values in $X.$    For $q$
in~$[1,+\infty],$ the notation $L^q(I;\,X)$ stands for the set of
measurable functions on $I$ with values in $X,$ such that
$t\longmapsto\|f(t)\|_{X}$ belongs to $L^q(I).$ We denote by
$L^p_T(L^q_{\rm h}(L^r_{\rm v}))$ the space $L^p([0,T];
L^q(\R_{x_{\rm h}};L^r(\R_{x_3})))$ with $x_{\rm h}=(x_1,x_2),$ and
$\nabla_{\rm h}=(\p_{x_1},\p_{x_2}),$ $\D_{\rm
h}=\p_{x_1}^2+\p_{x_2}^2.$ $\na_\ve=(\na_{\rm h}, \ve\p_3),$
$\D_\ve=\D_{\rm h}+\ve^2\p_3^2,$ and $\na^\ve=(\na_{\rm h},
\ve^2\pa_3)$.  Finally,  we denote by
$\bigl\{d_{k,\ell}\bigr\}_{k,\ell\in\Z}$ and
$\bigl\{d_{k,\ell}(t)\bigr\}_{k,\ell\in\Z}$ (resp. $
\bigl\{d_{k}\bigr\}_{k\in\Z}$ and $\bigl\{d_{k}(t)\bigr\}_{k\in\Z}$
) to be  generic elements in the sphere of $\ell^1(\Z^2)$ (resp.
$\ell^1(\Z)$).

\setcounter{equation}{0}
\section{Structure of the proof}

\subsection{Reduction to a rescaled problem}\label{sect2}
For simplicity, we shall  take $\bar{\rho}=1$ in \eqref{data1} in
what follows. As in \cite{cgp,PZZ3,mz1,mz2}, we shall seek a
solution of \eqref{eq:InhomoNS} with the form \beno
\rho(t,x)=1+\ve^\be a(t,x_{\rm h},\ve x_3),\quad
u(t,x)=\big(\ve^{1-\al}v^{\rm h}, \ve^{-\al}v^3\big)(t,x_{\rm h},\ve
x_3) \andf q(t,x)=p(t,x_{\rm h},\ve x_3). \eeno  This leads to the
following rescaled inhomogeneous Navier-Stokes equations
\begin{equation}\label{eq:InhomoNS-scaled}
\left\{
\begin{array}{ll}
\p_ta+\ve^{1-\al}v\cdot\na a=0,\\
(1+\ve^\be a)(\p_t v+\ve^{1-\al} v\cdot\na v)-\Delta_\ve v+\na^\ve q=0, \\
\textrm{div} v=0,\\
(a,v)|_{t=0}=(a_0,v_0).
\end{array}
\right.
\end{equation}
 Due to
$\textrm{div} v=0$, the rescaled pressure $q$ is determined by the
following elliptic equation \ben\label{eq:pressure-scaled}
-\textrm{div}\Big(\f 1 {1+\ve^{\be}a}\na^\ve q\Big)
=\ve^{1-\al}\textrm{div}(v\cdot\na v)-\textrm{div}\Big(\f 1
{1+\ve^{\be}a}\Delta_\ve v\Big), \een which is degenerate in $x_3$
direction when $\ve$ is small and whence $\na q$ is not uniformly
bounded in the usual isentropic Besov spaces. In order to handle
this problem and also to capture the subtle dissipative mechanism in
\eqref{eq:InhomoNS-scaled}, we need to use the anisotropic
Littlewood-Paley theory.

 As in \cite{bg,CDGG,c-p-z,CZ1,CZ5,Pa02,mz1,mz2}, the definitions of the spaces we are going to work with
require anisotropic dyadic decomposition   of the Fourier variables.
Let us recall from \cite{BCD} that \beq
\begin{split}
&\Delta_k^{\rm h}a=\cF^{-1}(\varphi(2^{-k}|\xi_{\rm
h}|)\widehat{a}),\qquad
\Delta_\ell^{\rm v}a =\cF^{-1}(\varphi(2^{-\ell}|\xi_3|)\widehat{a}),\\
&S^{\rm h}_ka=\cF^{-1}(\chi(2^{-k}|\xi_{\rm h}|)\widehat{a}),
\qquad\ S^{\rm v}_\ell a =
\cF^{-1}(\chi(2^{-\ell}|\xi_3|)\widehat{a})
 \quad\mbox{and}\\
&\Delta_ja=\cF^{-1}(\varphi(2^{-j}|\xi|)\widehat{a}),
 \qquad\ \
S_ja=\cF^{-1}(\chi(2^{-j}|\xi|)\widehat{a}), \end{split}
\label{1.0}\eeq where $\xi_{\rm h}=(\xi_1,\xi_2),$ $\cF a$ and
$\widehat{a}$ denote the Fourier transform of the distribution $a,$
$\chi(\tau)$ and~$\varphi(\tau)$ are smooth functions such that
 \beno
&&\Supp \varphi \subset \Bigl\{\tau \in \R\,/\  \ \frac34 \leq
|\tau| \leq \frac83 \Bigr\}\andf \  \ \forall
 \tau>0\,,\ \sum_{j\in\Z}\varphi(2^{-j}\tau)=1,\\
&&\Supp \chi \subset \Bigl\{\tau \in \R\,/\  \ \ |\tau|  \leq
\frac43 \Bigr\}\quad \ \ \andf \  \ \, \chi(\tau)+ \sum_{j\geq
0}\varphi(2^{-j}\tau)=1.
 \eeno

\begin{definition}\label{def2.1} {\sl Let us define the anisotropic Besov space $B^{s_1,s_2}(\R^3)$   as the space of
distribution $f$ in~$\cS'_h(\R^3),$ which means that $f$ is
in~$\cS'(\R^3)$ and satisfies~$\
\lim_{j\to-\infty}\|S_jf\|_{L^\infty}=0$,  such that
$$
\|f\|_{B^{s_1,s_2}}\eqdefa\sum_{k,\ell\in\Z}2^{ks_1+\ell
s_2}\|\Delta_k^{\rm h}\Delta_\ell^{\rm v} f\|_{L^2}
$$
is finite. }
\end{definition}

We also need to use Chemin-Lerner type spaces
$\widetilde{L}^p_T(B^{s_1,s_2})$ with its norm defined by \beq
\label{cheminl}\|u\|_{\widetilde{L}^p_T(B^{s_1,s_2})}\eqdefa
\sum_{k,\ell\in\Z}2^{ks_1+\ell s_2}\|\Delta_k^{\rm
h}\Delta_\ell^{\rm v}u\|_{L^p_T(L^2)}. \eeq It is easy to observe
that $\widetilde{L}^1_T(B^{s_1,s_2})=L^1_T(B^{s_1,s_2})$ and for any
$p>1$, \ben \|u\|_{{L}^p_T(B^{s_1,s_2})}\le
\|u\|_{\widetilde{L}^p_T(B^{s_1,s_2})}. \een

Theorem \ref{thm:main} can be deduced from the following theorem.

\begin{thm}\label{thm:main-r}
{\sl Under the same assumptions of  Theorem \ref{thm:main}, there
exists a positive constant $\ve_0,$ which depends on
$\|(a_0,v_0)\|_X,$ such that the rescaled inhomogeneous
Navier-Stokes equations (\ref{eq:InhomoNS-scaled}) has a unique
global smooth solution for any $\ve\in ]0,\ve_0[$.}
\end{thm}

\begin{rmk} More detailed information concerning the solution of
(\ref{eq:InhomoNS-scaled}) obtained in Theorem \ref{thm:main-r} will
be presented in Subsection \ref{subsect2.3}. As a matter of fact, we
shall prove that for $\theta(t), \psi(t)$ determined respectively by
\eqref{g.4} and \eqref{psias}, there holds \beno \sup_{t\geq
0}\theta(t)\leq C\ve^\ga\|v_0\|_{X_2}\andf \sup_{t\geq 0}\Psi(t)\leq
C\big(\|a_0\|_{X_1}+\|v_0\|_{X_3}\big). \eeno
\end{rmk}

\subsection{The functional setting}

The proof of Theorem \ref{thm:main-r} relies on the exponential
decay estimate for the Fourier transform of the solution. For this
end, we define \beq\label{k.2} f_\Psi(t)\eqdefa
\cF^{-1}\bigl(e^{\Psi(t,\cdot)}\widehat f(t,\cdot)\bigr). \eeq

We introduce the first key quantity $\theta(t)$ describing the
evolution of the analytic band of the solution, which is defined by
\beq\label{g.4}
\begin{split}
\dot\theta(t)=&\ve^{1-\al}\Bigl(\|v_\Phi^{\rm
h}(t)\|_{B^{1,\f12}}+\|v_\Phi^{\rm h}(t)\|_{B^{1-\ga,\f12+\ga}}
+\|v_\Phi^{\rm h}(t)\|_{B^{1+\ga,\f12-\ga}}+\ve^{1+\ga}\|v_\Phi^{\rm
h}(t)\|_{B^{-\ga,\f32+\ga}}\Bigr)\\
&+\ve^\ga\Bigl(\|v_\Phi^3(t)\|_{B^{1,\f12}}+\|v_\Phi^3(t)\|_{B^{1+\ga,\f12-\ga}}+\ve^{1+\ga}\|v_\Phi^3(t)\|_{B^{-\ga,\f32+\ga}}\Bigr),
\end{split} \eeq
with $\theta(0)=0$, where the phase $\Phi$ is given by
\ben\label{def:phase} \Phi(t,\xi)\eqdefa
(\delta-\lambda\theta(t))|\xi| \een for some $\lambda>0$ that will
be chosen later on. To control the growth of $\theta(t)$, we need to
introduce the second key quantity $\Psi(t)$ defined by
\beq\label{psias} \Psi(t)\eqdefa
\Psi_1(t)+\Psi_2(t)+\Psi_3(t)+\Psi_4(t), \eeq where \beq\label{p.1}
\begin{split}
\Psi_1(t)\eqdefa&\|a_\Phi\|_{\widetilde{L}^\infty_t(B^{1,\f12})}+\|a_\Phi\|_{\widetilde{L}^\infty_t(B^{1+\ga,\f12-\ga})}
+\|a_\Phi\|_{\widetilde{L}^\infty_t(B^{1-\ga,\f12+\ga})}
+\ve^{3\al+3\ga}\|a_\Phi\|_{\widetilde{L}^\infty_t(B^{\ga,\f32-\ga})},\\
\Psi_2(t)\eqdefa&
\|v_\Phi\|_{\widetilde{L}^\infty_t(B^{0,\f12})}+\|v_\Phi\|_{\widetilde{L}^\infty_t(B^{\ga,\f12-\ga})}
+\|v_\Phi\|_{\widetilde{L}^\infty_t(B^{-\ga,\f12+\ga})},\\
\Psi_3(t)\eqdefa&\ve^{2\al+2\ga}\Big(\|v_\Phi^{\rm
h}\|_{{L}^1_t(B^{2,\f12})}+\|v_\Phi^{\rm
h}\|_{{L}^1_t(B^{2+\ga,\f12-\ga})}+\|v_\Phi^{\rm
h}\|_{{L}^1_t(B^{2-\ga,\f12+\ga})}
+\ve^{2}\|v_\Phi^{\rm h}\|_{{L}^1_t(B^{0,\f52})}\\
&+\ve^{2}\|v_\Phi^{\rm
h}\|_{{L}^1_t(B^{\ga,\f52-\ga})}+\ve^{2}\|v_\Phi^{\rm
h}\|_{{L}^1_t(B^{-\ga,\f52+\ga})}\Big)
+\|v_\Phi^3\|_{{L}^1_t(B^{2,\f12})}+\|v_\Phi^3\|_{{L}^1_t(B^{2+\ga,\f12-\ga})}\\
&+\|v_\Phi^3\|_{{L}^1_t(B^{2-\ga,\f12+\ga})}
+\ve^{2}\|v_\Phi^3\|_{{L}^1_t(B^{0,\f52})}+\ve^{2}\|v_\Phi^3\|_{{L}^1_t(B^{\ga,\f52-\ga})}+\ve^{2}\|v_\Phi^3\|_{{L}^1_t(B^{-\ga,\f52+\ga})},\\
\Psi_4(t)\eqdefa &\ve^{\al+\ga}\Big(\|v_\Phi^{\rm
h}\|_{\widetilde{L}^2_t(B^{1,\f12})}+\|v_\Phi^{\rm
h}\|_{\widetilde{L}^2_t(B^{1+\ga,\f12-\ga})}+\|v_\Phi^{\rm
h}\|_{\widetilde{L}^2_t(B^{1-\ga,\f12+\ga})}
+\ve\|v_\Phi^{\rm h}\|_{\widetilde{L}^2_t(B^{0,\f32})}\\
&+\ve\|v_\Phi^{\rm
h}\|_{\widetilde{L}^2_t(B^{\ga,\f32-\ga})}+\ve\|v_\Phi^{\rm
h}\|_{\widetilde{L}^2_t(B^{-\ga,\f32+\ga})}\Big)
+\|v_\Phi^3\|_{\widetilde{L}^2_t(B^{1,\f12})}+\|v_\Phi^3\|_{\widetilde{L}^2_t(B^{1+\ga,\f12-\ga})}\\
&+\|v_\Phi^3\|_{\widetilde{L}^2_t(B^{1-\ga,\f12+\ga})}
+\ve\|v_\Phi^3\|_{\widetilde{L}^2_t(B^{0,\f32})}+\ve\|v_\Phi^3\|_{\widetilde{L}^2_t(B^{\ga,\f32-\ga})}+\ve\|v_\Phi^3\|_{\widetilde{L}^2_t(B^{-\ga,\f32+\ga})}.
\end{split}
\eeq

The proof of Theorem \ref{thm:main-r} will be based on the following
two propositions, whose proofs will be presented in Section
\ref{sect7} and Section \ref{sect8} respectively.  Let us make the
{\it a priori} assumption that \ben\label{assum-a} \Psi_1(T)\le
K,\een which will be determined hereafter.

\begin{prop}\label{prop:tht} {\it Under the assumption that $\al\in
\bigl]0,\f12\big[, \be>\al$ and $0<\ga< \min\bigl(\f  {\be-\al} 2,
\f {1-2\al} 4\bigr),$ there exists a positive constant $C$  such
that, for any positive $\lambda$ and for any $t$ satisfying
$\theta(t)\le \delta/\lambda$, and for $\epsilon$ given by
\eqref{k.50}, $\ve$ is so small that \beq\label{epsilon} \ve\leq
\min\Bigl(\Bigl(\f{\epsilon}{K}\Bigr)^{\f1\be},\Bigl(\f1
{2CK}\Bigr)^{\f1{\be-\ga}}\Bigr).\eeq Then we have \beno \tht(t) \le
C\Bigl(\ve^\ga\|v_0\|_{X_2}+\max\big(\ve^{\be-\al-2\ga},\ve^{\al},\ve^{1-2\al-2\ga}\bigr)\Psi(t)\tht(t)\Bigr).
\eeno }
\end{prop}

\begin{prop}\label{prop:Psi}
{\it Let $\al\in \bigl]0,\f13\big[, \be>2\al,$ $0<\ga\le
\min\bigl(\f{\be-2\al}2,\f{1-3\al}4\bigr),$ and $\ve$ satisfy
\eqref{epsilon}. Then there exists  a positive constant $C$ such
that, for any positive $\lambda$ and for any $t$ satisfying
$\theta(t)\le \delta/\lambda$, we have \beno \Psi(t)\le
C\big(\|a_0\|_{X_1}+\|v_0\|_{X_3}\big)+C\biggl(\f 1
{\lambda}+\max\Bigl(\ve^\ga,\ve^{\be-2\al-2\ga},\ve^{1-3\al-4\ga},
K\ve^{\be-2\al-\ga}\Bigr)\Psi(t)\biggr)\Psi(t). \eeno}
\end{prop}

\subsection{Proof of Theorem \ref{thm:main-r}}\label{subsect2.3}

The proof of Theorem \ref{thm:main-r} essentially follows from the
{\it a priori} estimates for smooth enough solutions of
(\ref{eq:InhomoNS-scaled}) (see \cite{cgp} for instance). For
simplicity, here we only present the global {\it a priori} estimates
for smooth enough solutions of (\ref{eq:InhomoNS-scaled}). Toward
this, for $\theta(t), \Psi(t)$ determined respectively by
\eqref{g.4} and \eqref{psias}, we take $K=K_0\eqdefa
4C\big(\|a_0\|_{X_1}+\|v_0\|_{X_3}\big)$ in \eqref{assum-a} and
define  \beq\label{Tstar} T^\ast \eqdefa\sup\big\{\ T>0: \ \theta(T)
\le 4C\ve^\ga\|v_0\|_{X_2} \andf \Psi(T)\le K_0\ \big\}. \eeq Then
it suffices to prove that $T^*=+\oo$ provided that $\ve$ is
sufficiently small. In order to use Proposition \ref{prop:tht} and
Proposition \ref{prop:Psi}, we need to assume that $\theta(T)\le \f
\delta \lambda,$  which leads to the condition that \beq
\label{epga} 4C\ve^\ga\|v_0\|_{X_2} \le \f \delta \lambda. \eeq Then
under the assumptions of Theorem \ref{thm:main},  we infer from
Proposition \ref{prop:tht} and Proposition \ref{prop:Psi} that for
all $T\in [0,T^\ast[$,
\begin{equation} \label{equ:thtpsi}
\begin{split}
&\theta(T)\le C\bigl(\ve^\ga+4\ve^{2\ga}K_0\bigr)\|v_0\|_{X_2},\andf\\
&\Psi(T)\le \f{K_0}4+C\Bigl(\f 1 {\lambda}+\ve^\ga
K_0+\ve^{2\ga}K_0^2\Bigr)K_0,
 \end{split}
 \end{equation} provided that $\ve$ is so small that $\ga\leq
 \min\bigl(\f{\be-\al}4,\f{\be-2\al}3,\f{1-3\al}5\bigr).$
We then select $\lambda$ so large  that ${\lambda}={4C}$, and then
choose $\ve$ to be so small that there holds \eqref{epsilon},
\eqref{epga} and $8C\ve^\ga K_0\leq 1,$ that is \beq\label{epthm}
\ve\leq
\min\biggl(\Bigl(\f{\epsilon}{K_0}\Bigr)^{\f1\be},\Bigl(\f1{2CK_0}\Bigr)^{\f1{\be-\ga}},
\Bigl(\f1{8CK_0}\Bigr)^{\f1\ga},
\Bigl(\f{\delta}{16C^2\|v_0\|_{X_2}}\Bigr)^{\f1{\ga}}\biggr). \eeq
With this choice of $\ve$, we infer from (\ref{equ:thtpsi})  that
for all $T\in [0,T^\ast[$, \beno \theta(T)\leq
2C\ve^\ga\|v_0\|_{X_2} \andf \Psi(T)\leq \f{3K_0}4, \eeno which
contradicts with \eqref{Tstar} if $T^*<+\oo$. This in turn shows
that $T^*=\infty,$ and whence we conclude the proof of Theorem
\ref{thm:main-r}.\ef

\setcounter{equation}{0}

\section{The action of subadditive phases on products}

For any function $f$, we denote by $f^+$ the inverse Fourier
transform of $|\widehat f|$. Let us notice that the map $f\mapsto
f^+$ preserves the norm of the anisotropic Besov space $B^{s_1,s_2}$
given by Definition \ref{def2.1}. Throughout this section, $\Phi$
denotes a locally bounded function on $\R^+\times \R^3$ which
verifies  \beq\label{subaddi} \Phi(t,\xi)\le
\Phi(t,\xi-\eta)+\Phi(t,\eta). \eeq

\medbreak
 For the
convenience of the readers, we recall the following anisotropic
Bernstein type lemma from \cite{CZ1, Pa02}:

\begin{lem} \label{lem:Berstein}
 {\sl Let $\cB_{\rm h}$ (resp.~$\cB_{\rm v}$) a ball
of~$\R^2_{\rm h}$ (resp.~$\R_{\rm v}$), and~$\cC_{\rm h}$
(resp.~$\cC_{\rm v}$) a ring of~$\R^2_{\rm h}$ (resp.~$\R_{\rm v}$);
let~$1\leq p_2\leq p_1\leq \infty$ and ~$1\leq q_2\leq q_1\leq
\infty.$ Then there holds:

\smallbreak\noindent If the support of~$\wh a$ is included
in~$2^k\cB_{\rm h}$, then
\[
\|\partial_{x_{\rm h}}^\alpha a\|_{L^{p_1}_{\rm h}(L^{q_1}_{\rm v})}
\lesssim 2^{k\left(|\al|+2\left(\f 1 {p_2}-\f 1 {p_1}\right)\right)}
\|a\|_{L^{p_2}_{\rm h}(L^{q_1}_{\rm v})}.
\]
If the support of~$\wh a$ is included in~$2^\ell\cB_{\rm v}$, then
\[
\|\partial_{x_3}^\beta a\|_{L^{p_1}_{\rm h}(L^{q_1}_{\rm v})}
\lesssim 2^{\ell\left(\beta+\left(\f 1{q_2}-\f1{q_1}\right)\right)}
\| a\|_{L^{p_1}_{\rm h}(L^{q_2}_{\rm v})}.
\]
If the support of~$\wh a$ is included in~$2^k\cC_{\rm h}$, then
\[
\|a\|_{L^{p_1}_{\rm h}(L^{q_1}_{\rm v})} \lesssim
2^{-kN}\sup_{|\al|=N} \|\partial_{x_{\rm h}}^\al a\|_{L^{p_1}_{\rm
h}(L^{q_1}_{\rm v})}.
\]
If the support of~$\wh a$ is included in~$2^\ell\cC_{\rm v}$, then
\[
\|a\|_{L^{p_1}_{\rm h}(L^{q_1}_{\rm v})} \lesssim 2^{-\ell N}
\|\partial_{x_3}^N a\|_{L^{p_1}_{\rm h}(L^{q_1}_{\rm v})}.
\]
}
\end{lem}

\begin{lem}\label{lem3.1}
{\sl Let $\s_1<\s<\s_2$ and $s_2<s<s_1$ with
$\s_1+s_1=\s+s=\s_2+s_2.$ Then one has \beq \label{k.1}
\|g_\Phi\|_{B^{\s,s}}\lesssim
\|g_\Phi\|_{B^{\s_1,s_1}}+\|g_\Phi\|_{B^{\s_2,s_2}}. \eeq} \end{lem}
\begin{proof} According to Definition \ref{def2.1} and \eqref{k.2},
one has \beno \|g_\Phi\|_{B^{\s,s}}=\sum_{k<\ell}2^{k\s}2^{\ell
s}\|\D_k^{\rm h}\D_\ell^{\rm v}g_\Phi\|_{L^2}+\sum_{k\geq
\ell}2^{k\s}2^{\ell s}\|\D_k^{\rm h}\D_\ell^{\rm v}g_\Phi\|_{L^2}.
\eeno However, it is easy to observe that  \beno
\begin{split} \sum_{k<\ell}2^{k\s}2^{\ell s}\|\D_k^{\rm h}\D_\ell^{\rm
v}g_\Phi\|_{L^2}\lesssim
&\sum_{k<\ell}d_{k,\ell}2^{k(\s-\s_1)}2^{\ell(s-s_1)}\|g_\Phi\|_{B^{\s_1,s_1}}\\
\lesssim&
\sum_{k<\ell}d_{k,\ell}2^{\ell(\s+s-\s_1-s_1)}\|g_\Phi\|_{B^{\s_1,s_1}}\lesssim
\|g_\Phi\|_{B^{\s_1,s_1}} \end{split} \eeno  where we used the fact
that $\s+s=\s_1+s_1$ and that $\s>\s_1$ so that $2^{k(\s-\s_1)}\leq
2^{\ell(\s-\s_1)}.$

Along the same line, we have \beno \sum_{k\geq \ell}2^{k\s}2^{\ell
s}\|\D_k^{\rm h}\D_\ell^{\rm v}g_\Phi\|_{L^2}\lesssim
\|g_\Phi\|_{B^{\s_2,s_2}}. \eeno This completes the proof of
\eqref{k.1}.
\end{proof}

To study the law of product in the anisotropic Besov spaces, we need
to use Bony's decomposition. We first recall the isotropic
para-differential decomposition from \cite{Bo} that: let $a$ and $b$
be in~$ \cS_h'(\R^3)$, \beq \label{pd}\begin{split} &
ab=T(a,b)+ R(a,b)+\bar{T}(a,b)\andf ab =T(a,b)+\cR(a,b)\with\\
& T(a,b)=\sum_{j\in\Z}S_{j-1}a\Delta_j b, \quad
\bar{T}(a,b)=T(b,a),\quad \cR(a,b)=\sum_{j\in\Z}\D_j aS_{j+2}b \andf\\
&R(a,b)=\sum_{j\in\Z}\Delta_j
a\tilde{\Delta}_{j}b,\quad\hbox{with}\quad
\tilde{\Delta}_{j}b=\sum_{j'=j-1}^{j+1}\D_{j'} b.
\end{split} \eeq Sometimes we shall use Bony's decomposition for
both horizontal and vertical variables simultaneously.

\medbreak As an application of the above basic facts on
Littlewood-paley theory, we now present the following law of product
in the anisotropic Besov spaces:

\begin{lem}\label{lem3.2}
{\sl Let $\s_1,\cdots,\s_8$ and $s_1,\cdots, s_8$ be real numbers so
that \beno
\begin{split}
&\s_1+\s_2=\s_3+\s_4=\s_5+\s_6=\s_7+\s_8>0
 \with \s_1,\s_4,\s_5,\s_8\leq 1 \andf\\
&s_1+s_2=s_3+s_4=s_5+s_6=s_7+s_8>0\ \ \ \with s_1,s_4,s_6, s_7\leq
\f12.
\end{split}
\eeno Then there holds \beq\label{k.3}
\begin{split}
\|[ab]_\Phi\|_{B^{\s_1+\s_2-1,s_1+s_2-\f12}}\lesssim &
\|a_\Phi\|_{B^{\s_1,s_1}}\|b_\Phi\|_{B^{\s_2,s_2}}+\|a_\Phi\|_{B^{\s_3,s_3}}\|b_\Phi\|_{B^{\s_4,s_4}}\\
&+
\|a_\Phi\|_{B^{\s_5,s_5}}\|b_\Phi\|_{B^{\s_6,s_6}}+\|a_\Phi\|_{B^{\s_7,s_7}}\|b_\Phi\|_{B^{\s_8,s_8}}.
\end{split} \eeq } \end{lem}

\begin{proof} We first observe from \eqref{subaddi} that
 \beq \label{k.3ad} \big|\cF\bigl[\Delta_k^{\rm
h}\Delta_\ell^{\rm v}(ab)\bigr]_\Phi(\xi)\big|\le
\cF\bigl[\Delta_k^{\rm h}\Delta_\ell^{\rm v}(a^+_\Phi
b^+_\Phi)\bigr](\xi). \eeq Hence it suffices to prove \eqref{k.3}
for $\Phi=0.$

Indeed we get, by applying Bony's decomposition \eqref{pd} for both
horizontal and vertical variables, that \beno ab=\bigl(T^{\rm
h}+R^{\rm h}+\bar{T}^{\rm h}\bigr)\bigl(T^{\rm v}+R^{\rm
v}+\bar{T}^{\rm v}\bigr)(a,b). \eeno Considering the support to the
Fourier transform of the terms in $R^{\rm h}R^{\rm v}(a,b),$ by
applying Lemma \ref{lem:Berstein}, we obtain \beno
\begin{split}
\|\D_k^{\rm h}\D_\ell^{\rm v}R^{\rm h}R^{\rm v}(a,b)\|_{L^2}\lesssim
&2^k2^{\f{\ell}2}\sum_{\substack{k'\geq k-3\\ \ell'\geq
\ell-3}}\|\D_{k'}^{\rm h}\D_{\ell'}^{\rm
v}a\|_{L^2}\|\wt{\D}_{k'}^{\rm
h}\wt{\D}_{\ell'}^{\rm v}b\|_{L^2}\\
\lesssim &2^k2^{\f{\ell}2}\sum_{\substack{k'\geq k-3\\ \ell'\geq
\ell-3}}d_{k',\ell'}2^{-k'(\s_1+\s_2)}2^{-\ell'(s_1+s_2)}\|a\|_{B^{\s_1,s_1}}\|b\|_{B^{\s_2,s_2}}\\
\lesssim
&d_{k,\ell}2^{-k(\s_1+\s_2-1)}2^{-\ell\bigl(s_1+s_2-\f12\bigr)}\|a\|_{B^{\s_1,s_1}}\|b\|_{B^{\s_2,s_2}}.
\end{split}
\eeno The same estimate holds for $T^{\rm h}T^{\rm v}(a,b),$ $T^{\rm
h}R^{\rm v}(a,b),$ and $R^{\rm h}T^{\rm v}(a,b).$

While since $\s_5\leq 1$ and $s_6\leq \f12,$ it follows from Lemma
\ref{lem:Berstein} that \beno &&\|S_{k'-1}^{\rm h}\D_{\ell'}^{\rm
v}a\|_{L^\infty_{\rm h}(L^2_{\rm v})}\lesssim
d_{\ell'}2^{k'(1-\s_5)}2^{-\ell's_5}\|a\|_{B^{\s_5,s_5}}\andf\\
&&\|\D_{k'}^{\rm h}S_{\ell'-1}^{\rm v}b\|_{L^2_{\rm h}(L^\infty_{\rm
v})}\lesssim
d_{k'}2^{-k'\s_6}2^{\ell'\bigl(\f12-s_6\bigr)}\|b\|_{B^{\s_6,s_6}},
\eeno from which, we infer \beno
\begin{split}
\|\D_k^{\rm h}\D_\ell^{\rm v}T^{\rm h}\bar{T}^{\rm
v}(a,b)\|_{L^2}\lesssim &\sum_{\substack{|k'- k|\leq 4\\
|\ell'-\ell|\leq 4}}\|S_{k'-1}^{\rm h}\D_{\ell'}^{\rm
v}a\|_{L^\infty_{\rm h}(L^2_{\rm v})}\|{\D}_{k'}^{\rm
h}S_{\ell'-1}^{\rm v}b\|_{L^2_{\rm h}(L^\infty_{\rm v})}\\
\lesssim
&d_{k,\ell}2^{-k(\s_5+\s_6-1)}2^{-\ell\bigl(s_5+s_6-\f12\bigr)}\|a\|_{B^{\s_5,s_5}}\|b\|_{B^{\s_6,s_6}}.
\end{split}
\eeno The same estimate holds for $R^{\rm h}\bar{T}^{\rm v}(a,b).$

By the same manner, we obtain \beno
\begin{split}
\|\D_k^{\rm h}\D_\ell^{\rm v}\bar{T}^{\rm h}{T}^{\rm
v}(a,b)\|_{L^2}\lesssim &\sum_{\substack{|k'- k|\leq 4\\
|\ell'-\ell|\leq 4}}\|\D_{k'}^{\rm h}S_{\ell'-1}^{\rm
v}a\|_{L^2_{\rm h}(L^\infty_{\rm v})}\|{S}_{k'-1}^{\rm
h}\D_{\ell'}^{\rm v}b\|_{L^\infty_{\rm h}(L^2_{\rm v})}\\
\lesssim
&d_{k,\ell}2^{-k(\s_7+\s_8-1)}2^{-\ell\bigl(s_7+s_8-\f12\bigr)}\|a\|_{B^{\s_7,s_7}}\|b\|_{B^{\s_8,s_8}}.
\end{split}
\eeno The same estimate holds for $\bar{T}^{\rm h}{R}^{\rm v}(a,b).$

Finally since $\s_4\leq 1, s_4\leq \f12,$ applying Lemma
\ref{lem:Berstein} yields \beno \|S_{k'-1}^{\rm h}S_{\ell'-1}^{\rm
v}b\|_{L^\infty}\lesssim
2^{k'(1-\s_4)}2^{\ell'\bigl(\f12-s_4\bigr)}\|b\|_{B^{\s_4,s_4}},
\eeno which ensures \beno
\begin{split}
\|\D_k^{\rm h}\D_\ell^{\rm v}\bar{T}^{\rm h}\bar{T}^{\rm
v}(a,b)\|_{L^2}\lesssim &\sum_{\substack{|k'- k|\leq 4\\
|\ell'-\ell|\leq 4}}\|\D_{k'}^{\rm h}\D_{\ell'}^{\rm
v}a\|_{L^2}\|{S}_{k'-1}^{\rm
h}S_{\ell'-1}^{\rm v}b\|_{L^\infty}\\
\lesssim
&d_{k,\ell}2^{-k(\s_3+\s_4-1)}2^{-\ell\bigl(s_3+s_4-\f12\bigr)}\|a\|_{B^{\s_3,s_3}}\|b\|_{B^{\s_4,s_4}}.
\end{split}
\eeno This completes the proof of \eqref{k.3}.
\end{proof}

We remark that the law of product of Lemma \ref{lem3.2} works also
for Chemin-Lerner type spaces $\widetilde{L}^p_T(B^{s_1,s_2}).$
Indeed the proof of Lemma \ref{lem3.2} implies the following
corollary:

\begin{col}\label{col3.1}
{\sl Let $p,p_1,p_2, p_3, p_4\in [1,\infty]$ with
$\f1p=\f1{p_1}+\f1{p_2}=\f1{p_3}+\f1{p_4}.$ Then under the
assumptions that if $\sigma_1,\sigma_2,\s_3,\s_4\le 1$ and $s_1,s_2,
s_3,s_4$ satisfy $0<\sigma_1+\sigma_2=\s_3+\s_4,$ $s_1,s_4\le \f12$
and $0<s_1+s_2=s_3+s_4,$  or if $\sigma_1,\sigma_2,\s_3,\s_4$ and
$s_1,s_2, s_3,s_4\leq \f12$ satisfy $\s_1,\s_4\leq 1,$
$0<\sigma_1+\sigma_2=\s_3+\s_4,$ and  $0<s_1+s_2=s_3+s_4,$   one has
\beq\label{k.5} \|[ab]_\Phi\|_{\wt{L}^p_T(B^{\sigma_1+\sigma_2-1,
s_1+s_2-\f12})}\lesssim
\|a_\Phi\|_{\widetilde{L}^{p_1}_T(B^{\sigma_1,s_1})}
\|b_\Phi\|_{\widetilde{L}^{p_2}_T(B^{\sigma_2,s_2})}+\|a_\Phi\|_{\widetilde{L}^{p_3}_T(B^{\sigma_3,s_3})}
\|b_\Phi\|_{\widetilde{L}^{p_4}_T(B^{\sigma_4,s_4})}. \eeq In the
particular case when $\s_1,\s_2\leq 1$ with $\s_1+\s_2>0$ and
$s_1,s_2\leq \f12$ with $s_1+s_2>0,$ one has
  \beq\label{k.6}
\|[ab]_\Phi\|_{\wt{L}^p_T(B^{\sigma_1+\sigma_2-1,
s_1+s_2-\f12})}\lesssim
\|a_\Phi\|_{\widetilde{L}^{p_1}_T(B^{\sigma_1,s_1})}
\|b_\Phi\|_{\widetilde{L}^{p_2}_T(B^{\sigma_2,s_2})}. \eeq }
\end{col}

\begin{Remark}\label{rem:product}
Let us remark that if $\sigma_1,\sigma_2\le 1, \sigma_1+\sigma_2>0$
and $s_1\le \f12$, the proof of Lemma \ref{lem3.2} also implies
\beno \big\|\bigl[\Delta_k^{\rm h}\Delta_\ell^{\rm v}T^{\rm
v}(a,b)\bigr]_\Phi(t)\big\|_{L^2}\le
C(d_{k}(t)d_\ell+d_{k,\ell})2^{k(1-\sigma_1-\sigma_2)}2^{\ell(\f12-s_1-s_2)}\|a_\Phi(t)\|_{B^{\sigma_1,s_1}}
\|b_\Phi\|_{\widetilde{L}^\infty_t({B}^{\sigma_2,s_2})}. \eeno
\end{Remark}

\begin{lem}\label{lem:composition}
{\sl Let $\s, s>0$ and $\s\leq 1$ or $s\leq \f12;$  let $G(r)=\f r
{1+r}.$ Then there exists $\epsilon>0$ such that if \beq\label{k.50}
 \|a_\Phi\|_{\widetilde{L}^\infty_T(B^{1,\f12})}\le
\epsilon, \eeq there holds  \beno
\|[G(a)]_\Phi\|_{\widetilde{L}^\infty_T(B^{\s,s})}\le
2\|a_\Phi\|_{\wt{L}^\infty(B^{\s,s})}. \eeno}
\end{lem}

\begin{proof} Indeed under the assumption of
Lemma \ref{lem:composition}, we deduce from Corollary \ref{col3.1}
that \beq\label{k.9q} \|[ab]_\Phi\|_{\wt{L}^\infty_T(B^{\s, s})}\leq
C\Bigl( \|a_\Phi\|_{\widetilde{L}^{\infty}_T(B^{1,\f12})}
\|b_\Phi\|_{\widetilde{L}^{\infty}_T(B^{\sigma,s})}+\|a_\Phi\|_{\widetilde{L}^{\infty}_T(B^{\sigma,s})}
\|b_\Phi\|_{\widetilde{L}^{\infty}_T(B^{1,\f12})}\Bigr). \eeq Thanks
to \eqref{k.9q}, one can inductively prove that \beno
\|[a^k]_\Phi\|_{\wt{L}^\infty_T(B^{\s, s})}\leq kC^k
\|a_\Phi\|_{\widetilde{L}^{\infty}_T(B^{1,\f12})}^{k-1}
\|a_\Phi\|_{\widetilde{L}^{\infty}_T(B^{\sigma,s})}. \eeno On the
other hand,  Taylor's expansion gives \beno
G(r)=\sum_{k=1}^\infty(-1)^{k-1}r^k\quad \textrm{for} \quad r\in
]-1,1[, \eeno from which and \eqref{k.50}, we infer
\begin{align*}
\|[G(a)]_\Phi\|_{\widetilde{L}^\infty_T(B^{\s,s})}&\le
\sum_{k=0}^\infty\|[a^{k+1}]_\Phi\|_{\wt{L}^\infty_T(B^{\s, s})} \\
&\le\|a_\Phi\|_{\widetilde{L}^{\infty}_T(B^{\sigma,s})}
\sum_{k=0}^\infty(k+1)(\epsilon C)^k\le
2\|a_\Phi\|_{\widetilde{L}^{\infty}_T(B^{\sigma,s})}
\end{align*} whenever $\epsilon$ is so small that $C\epsilon\leq
\delta_0$ for some $\delta_0$ sufficiently small.  This yields the
lemma.\end{proof}

\setcounter{equation}{0}

\section{The action of the phase $\Phi$ on the heat semigroup}

This section is devoted to studying the action of the Fourier
multiplier $e^{\Phi(t,D)}$ on the heat semigroup $e^{t\Delta_\ve}$
for the phase function $\Phi(t,\xi)$ given by (\ref{def:phase}). Let
us first present the classical parabolic smoothing estimates in the
Chemin-Lerner type space.

\begin{lemma}\label{lem:parabolic-home}
{\sl Let $\be\in [0,2], r\in [1,\infty]$ and $ \s,s\in \R$. Let
$v_0=(v_0^{\rm h}, v_0^3)$ be a divergence free vector filed. Then
one has \beq \label{g.1}\begin{split} & \ve^{\f \be
r}\bigl\|\big[e^{t\Delta_\ve}v_0\big]_\Phi\bigr\|_{\widetilde{L}^{r}_T(B^{\s,s})}
\lesssim \|e^{\delta|D|}v_0\|_{B^{\s-\f{2-\be}r,s-\f \be r}}\andf\\
& \ve^{\f \be
r}\bigl\|\big[e^{t\Delta_\ve}v_0^3\big]_\Phi\bigr\|_{\widetilde{L}^{r}_T(B^{\s,s})}
\lesssim \|e^{\delta|D|}v_0^{\rm h}\|_{B^{\s+1-\f{2-\be}r,s-1-\f \be
r}}. \end{split}\eeq }
\end{lemma}

\begin{proof}\,By virtue of \eqref{k.2} and (\ref{def:phase}),
we get \beno \begin{split} \bigl\|\Delta_k^{\rm h}\Delta_\ell^{\rm
v}\big[e^{t\Delta_\ve}v_0\big]_\Phi\bigr\|_{L^2}\lesssim &
e^{-ct\bigl(2^{2k}+\ve^22^{2\ell}\bigr)}\|e^{\delta|D|}\Delta_k^{\rm
h}\Delta_\ell^{\rm v}v_0\|_{L^2}\\
\lesssim &
d_{k,\ell}2^{-k\bigl(\s-\f{2-\be}r\bigr)}2^{-\ell\bigl(s-\f\be
r\bigr)}e^{-ct\bigl(2^{2k}+\ve^22^{2\ell}\bigr)}\|e^{\delta|D|}v_0\|_{B^{\s-\f{2-\be}r,s-\f
\be r}}, \end{split} \eeno from which and  \beno
\bigl\|e^{-ct\bigl(2^{2k}+\ve^22^{2\ell}\bigr)}\bigr\|_{L^r_t}\le
C\min\bigl(2^{-2k},\ve^{-2}2^{-2\ell}\bigr)^\f1r, \eeno we deduce
\beno \ve^{\f \be r}\bigl\|\Delta_k^{\rm h}\Delta_\ell^{\rm
v}\big[e^{t\Delta_\ve}v_0\big]_\Phi\bigr\|_{L^r_t(L^2)}\lesssim
 d_{k,\ell}2^{-k\s}2^{-\ell s}\|e^{\delta|D|}v_0\|_{B^{\s-\f{2-\be}r,s-\f
\be r}},\eeno which leads to the first inequality of \eqref{g.1}.

Exactly by the same manner, since $\dive v_0=0,$ we get, applying
Lemma \ref{lem:Berstein}, that \beno \begin{split}
\bigl\|\Delta_k^{\rm h}\Delta_\ell^{\rm
v}\big[e^{t\Delta_\ve}v_0^3\big]_\Phi\bigr\|_{L^2}\lesssim &
2^{-\ell}e^{-ct\bigl(2^{2k}+\ve^22^{2\ell}\bigr)}\|e^{\delta|D|}\Delta_k^{\rm
h}\Delta_\ell^{\rm v}\dive_{\rm h}v_0^{\rm h}\|_{L^2}\\
\lesssim &
d_{k,\ell}2^{-k\bigl(\s-\f{2-\be}r\bigr)}2^{-\ell\bigl(s-\f\be
r\bigr)}e^{-ct\bigl(2^{2k}+\ve^22^{2j}\bigr)}\|e^{\delta|D|}v_0\|_{B^{\s+1-\f{2-\be}r,s-1-\f
\be r}}, \end{split} \eeno and whence
 \beno
\ve^{\f \be r}\bigl\|\Delta_k^{\rm h}\Delta_\ell^{\rm
v}\big[e^{t\Delta_\ve}v_0^3\big]_\Phi\bigr\|_{L^r_t(L^2)}\lesssim
 d_{k,\ell}2^{-k\s}2^{-\ell s}\|e^{\delta|D|}v_0\|_{B^{\s+1-\f{2-\be}r,s-1-\f
\be r}},\eeno which implies the second inequality of \eqref{g.1}.
This completes the proof of the lemma.
\end{proof}

In what follows, we denote \beq\label{g.2}
 E_\ve
f(t)\eqdefa\int_0^te^{(t-t')\Delta_\ve}f(t')\,dt'. \eeq

\begin{lemma}\label{lem:parabolic-inhome}
{\sl Let $\be\in [0,2], r_1, r_2\in [1,\infty]$ with $ r_2\le r_1,$
and $ \s,s\in \R$. Then there holds  \beq\label{g.3}
\ve^{\f\be{r}}\|[E_\ve f]_\Phi\|_{\widetilde{L}^{r_1}_T(B^{\s,s})}
\lesssim
\|f_\Phi\|_{\widetilde{L}^{r_2}_T(B^{\s-\f{2-\be}r,s-\f\be{r}})},
\eeq with $\f1r=1+\f 1 {r_1}-\f1 {r_2}$.}
\end{lemma}

\begin{proof}\,Notice that \beno \|\Delta_k^{\rm
h}\Delta_\ell^v[E_\ve f]_\Phi\|_{L^2}\lesssim
\int_0^te^{-c(t-t')\bigl(2^{2k}+\ve^22^{2\ell}\bigr)}\|\Delta_k^{\rm
h}\Delta_\ell^{\rm v}f_\Phi(t')\|_{L^2}\,dt', \eeno from which and
Young's inequality, we infer
\begin{align*}
\|\Delta_k^{\rm h}\Delta_\ell^{\rm v}[E_\ve
f]_\Phi\|_{L^{r_1}_T(L^2)}&\lesssim
\|e^{-t(2^{2k}+\ve^22^{2\ell})}\|_{L^r_T}\|\Delta_k^{\rm
h}\Delta_\ell^{\rm v}f_\Phi\|_{L^{r_2}_T(L^2)}\\
&\lesssim
\min\bigl(2^{-2k},\ve^{-2}2^{-2\ell}\bigr)^\f1r\|\Delta_k^{\rm
h}\Delta_\ell^{\rm v}f_\Phi\|_{L^{r_2}_T(L^2)}
\end{align*}
with $\f1 r=1+\f1 {r_1}-\f1 {r_2}$, which implies
\eqref{g.3}.\end{proof}

The following lemma concerns the regularizing effect due to the analyticity.

\begin{lemma}\label{lem:analytic}
{\sl Let $\s,s\in \R$, and $p(D)$ be a Fourier multiplier with
symbol $p(\xi)$ satisfying $|p(\xi)|\le C|\xi_3|$. Assume that $f$
verifies \ben\label{ass:f} \|\Delta_k^{\rm h}\Delta_\ell^{\rm
v}f_\Phi(t)\|_{L^2}\lesssim
(d_{k}(t)d_\ell+d_{k,\ell})2^{-k\s}2^{-\ell
s}\dot\theta(t)\|g_\Phi\|_{\widetilde{L}^\infty_t({B}^{\s,s})} \een
for $\dot\theta(t)$ given by \eqref{g.4}.  Then there holds
\beq\label{g.5} \|[E_\ve
p(D)f]_\Phi\|_{\widetilde{L}^{\infty}_T({B}^{\s,s})} \le \f C
\lambda\|g_\Phi\|_{\widetilde{L}^{\infty}_T({B}^{\s,s})}. \eeq }
\end{lemma}

\begin{proof} In view of \eqref{def:phase}, we write  \ben\label{eq:phi-t-s}
\Phi(t,D)-\Phi(t',D)=-\lambda\int_{t'}^t\dot\theta(\tau)\,d\tau|D|,
\een from which and   (\ref{ass:f}), we infer
\begin{align*}
\|\Delta_k^{\rm h}\Delta_\ell^{\rm v}[E_\ve
p(D)f]_\Phi\|_{L^\infty_t(L^2)}\lesssim&
2^\ell\int_0^te^{-c\lambda\int_{t'}^t\dot\theta(\tau)\,d\tau 2^\ell}\|\Delta_k^{\rm h}\Delta_\ell^{\rm v}f_\Phi(t')\|_{L^2}\,dt'\\
\lesssim& 2^{-k\s}2^{\ell(1-s)}\Bigl(d_\ell\int_0^te^{-c\lambda\int_{t'}^t\dot\theta(\tau)\,d\tau 2^\ell}d_k(t')\dot\theta(t')\,dt'\\
&\quad\qquad\qquad\
+d_{k,\ell}\int_0^te^{-c\lambda\int_{t'}^t\dot\theta(\tau)\,d\tau
2^\ell}\dot\theta(t')\,dt'\Bigr)\|g_\Phi\|_{\widetilde{L}^\infty_t({B}^{\s,s})},
\end{align*}
which implies
\begin{align*}
\|[E_\ve p(D)f]_\Phi\|_{\widetilde{L}^{\infty}_T({B}^{\s,s})}=&
\sum_{k,\ell\in \Z}2^{k\s}2^{\ell s}\|\Delta_k^{\rm h}\Delta_\ell^{\rm v}[E_\ve a(D)f]_\Phi\|_{L^\infty_t(L^2)}\\
\leq &
C\sum_{\ell\in\Z}2^\ell\Bigl(d_\ell\int_0^te^{-c\lambda\int_{t'}^t\dot\theta(\tau)\,d\tau
2^\ell}\dot\theta(t')\,dt'\\
&\qquad\qquad+\sum_{k\in \Z}d_{k,\ell}\int_0^te^{-c\lambda\int_{t'}^t\dot\theta(\tau)\,d\tau 2^\ell}\dot\theta(t')\,dt'\Bigr)
\|g_\Phi\|_{\widetilde{L}^\infty_t({B}^{\s,s})}\\
\leq & \f C \lambda\|g_\Phi\|_{\widetilde{L}^\infty_t({B}^{\s,s})}.
\end{align*}
This proves \eqref{g.5}. \end{proof}

\setcounter{equation}{0}

\section{Propagation of analytic regularity for the transport equation}

In this section, we investigate the propagation of analytic
regularity for the following transport equation:
\ben\label{eq:trans} \p_ta+\ve^{1-\al}v\cdot\na a=f,\quad
a(0,x)=a_0(x). \een

\begin{prop}\label{prop:transport}
{\sl Let $\s\in ]-1,1],$  $s\in \bigl]-\f12,\f12\bigr]$ and $v$ be a
solenoidal vector field. Let $\theta(T)\leq \f{\delta}{\lambda}$ and
$\Phi$ be the phase function given by \eqref{def:phase}. Assume that $e^{\delta |D|}a_0\in B^{\s,s}, f_\Phi\in
L^1_T(B^{\s,s})$, and $v_\Phi\in L^1_T(B^{1,\f12})\cap
L^1_T(B^{2,\f12})$. Then \eqref{eq:trans} has a unique solution $a$
on $[0,T]$ so that for any $t\in [0,T]$, there holds \beq
\label{o.9}
\begin{split}
\|a_\Phi\|_{\widetilde{L}^\infty_t(B^{\s,s})}\le
 \|e^{\delta|D|}a_0\|_{B^{\s,s}}+\|f_\Phi\|_{L^1_t(B^{\s,s})}&+C\Bigl(\f 1\lambda+\ve^{1-\al}\|v_\Phi\|_{L^1_t(B^{2,\f12})}\Bigr)\|a_\Phi\|_{\widetilde{L}^\infty_t(B^{\s,s})}.
\end{split} \eeq}
\end{prop}

\begin{proof}\,Since both the existence and uniqueness parts of Proposition \ref{prop:transport}
basically follow from the Estimate \eqref{o.9}. For simplicity, we
just present the detailed derivation of the {\it a priori} estimate
\eqref{o.9} for smooth enough solutions of \eqref{eq:trans}. Indeed
by virtue of (\ref{eq:phi-t-s}),  we first integrate
(\ref{eq:trans}) with respect to $t$ and then apply the operator
$e^{\Phi(t,D)}$ to the resulting equation to get
 \beq\label{o.4}
a_\Phi(t)=e^{\Phi(t,D)}a_0-\ve^{1-\al}\int_0^te^{-\lambda\int_{t'}^t\dot\theta(\tau)\,d\tau|D|}\big[v\cdot\na
a\big]_\Phi(t')\,dt'+\int_0^te^{-\lambda\int_{t'}^t\dot\theta(\tau)\,d\tau|D|}f_\Phi(t')\,dt'.
\eeq  We claim that \beq\label{o.10}
\begin{split}
\|\Delta_k^{\rm h}\Delta_\ell^{\rm v}\p_3[v^3 a]_\Phi(t)\|_{L^2}\leq
C&2^{-k\s}2^{-\ell s}\Bigl(
d_{k,\ell}2^\ell\|v_\Phi^3(t)\|_{B^{1,\f12}}
+d_{k}d_\ell(t)2^k\|v_\Phi^{\rm
h}(t)\|_{B^{1,\f12}}\\
&+d_{k}(t)d_\ell 2^\ell\|v_\Phi^3(t)\|_{B^{1,\f12}}
+d_{k,\ell}(t)\|v_\Phi^h(t)\|_{B^{2,\f12}}\Bigr)\|a_\Phi\|_{\widetilde{L}^\infty_t(B^{\s,s})}.
\end{split} \eeq Along the same line to the proof of Lemma \ref{lem3.2}, since the phase function $\Psi$ given by \eqref{def:phase}
verifies \eqref{subaddi} whenever $\theta(T)\leq
\f{\delta}{\lambda},$ it suffices to prove \eqref{o.10} for
$\Phi=0.$ As a matter of fact, by using Bony's decomposition for
both horizontal and vertical variables to $v^3a$, we write \beno
v^3a(t)=\bigl(T^{\rm h}+R^{\rm h}+\bar{T}^{\rm h}\bigr)\bigl(T^{\rm
v}+R^{\rm v}+\bar{T}^{\rm v}\bigr)(v^3,a)(t). \eeno Considering the
support to the Fourier transform of the terms in $R^{\rm h}R^{\rm
v}(v^3,a)(t),$ we get, by applying Lemma \ref{lem:Berstein}, that
\beno
\begin{split}
\|\D_k^{\rm h}\D_\ell^{\rm v}R^{\rm h}R^{\rm
v}(v^3,a)(t)\|_{L^2}\lesssim& 2^k2^{\f{\ell}2}\sum_{\substack{k'\geq
k-3\\\ell'\geq \ell-3}}\|\D_{k'}^{\rm h}\D_{\ell'}^{\rm
v}v^3(t)\|_{L^2}\|\wt{\D}_{k'}^{\rm h}\wt{\D}_{\ell'}^{\rm
v}a\|_{L^\infty_t(L^2)}\\
\lesssim& 2^k2^{\f{\ell}2}\sum_{\substack{k'\geq k-3\\\ell'\geq
\ell-3}}d_{k',\ell'}2^{-k'(1+\s)}2^{-\ell'\bigl(\f12+s\bigr)}\|v^3(t)\|_{B^{1,\f12}}\|a\|_{\wt{L}^\infty_t(B^{\s,s})}\\
\lesssim&
 d_{k,\ell}2^{-k\s}2^{-\ell s}\|v^3(t)\|_{B^{1,\f12}}\|a\|_{\wt{L}^\infty_t(B^{\s,s})}.
 \end{split}
 \eeno
 The same estimate holds for $T^{\rm h}T^{\rm v}(v^3,a)(t),$ $T^{\rm h}R^{\rm v}(v^3,a)(t)$ and $R^{\rm h}T^{\rm v}(v^3,a)(t).$

By the same manner, we have \beno
\begin{split}
\|\D_k^{\rm h}\D_\ell^{\rm v}\bar{T}^{\rm h}R^{\rm
v}(v^3,a)(t)\|_{L^2}\lesssim& 2^{\f{\ell}2}\sum_{\substack{|k'-
k|\leq 4\\\ell'\geq \ell-3}}\|\D_{k'}^{\rm h}\D_{\ell'}^{\rm
v}v^3(t)\|_{L^2}\|S_{k'-1}^{\rm h}\wt{\D}_{\ell'}^{\rm
v}a\|_{L^\infty_t(L^\infty_{\rm h}(L^2_{\rm v}))}\\
\lesssim &d_k(t)d_\ell 2^{-k\s}2^{-\ell
s}\|v^3(t)\|_{B^{1,\f12}}\|a\|_{\wt{L}^\infty_t(B^{\s,s})}.
 \end{split}
 \eeno
The same estimate holds for $\bar{T}^{\rm h}T^{\rm v}(v^3,a)(t).$

Whereas using the fact that $\dive v(t)=0$ and Lemma
\ref{lem:Berstein}, one has \beno \begin{split} \|S_{k'-1}^{\rm
h}\D_{\ell'}^{\rm v}v^3(t)\|_{L^\infty_{\rm h}(L^2_{\rm v})}\lesssim
&2^{-\ell'}\|S_{k'-1}^{\rm h}\D_{\ell'}^{\rm
v}\p_3v^3(t)\|_{L^\infty_{\rm
h}(L^2_{\rm v})}\\
\lesssim &2^{-\ell'}\|S_{k'-1}^{\rm h}\D_{\ell'}^{\rm v}\dive_{\rm
h}v^{\rm h}(t)\|_{L^\infty_{\rm h}(L^2_{\rm v})}\lesssim
d_{\ell'}(t)2^{k'}2^{-\f{3\ell'}2}\|v^{\rm h}(t)\|_{B^{1,\f12}},
\end{split} \eeno from which, we infer \beno
\begin{split}
\|\D_k^{\rm h}\D_\ell^{\rm v}T^{\rm h}\bar{T}^{\rm
v}(v^3,a)(t)\|_{L^2}\lesssim& \sum_{\substack{|k'-k|\leq
4\\|\ell'-\ell|\leq 4}}\|S_{k'-1}^{\rm h}\D_{\ell'}^{\rm
v}v^3(t)\|_{L^\infty_{\rm h}(L^2_{\rm v})}\|\D_{k'}^{\rm
h}S_{\ell'-1}^{\rm v}a\|_{L^\infty_t(L^2_{\rm h}(L^\infty_{\rm v}))}\\
\lesssim&
 d_{k}d_\ell(t)2^{k(1-\s)}2^{-\ell(1+s)}\|v^3(t)\|_{B^{1,\f12}}\|a\|_{\wt{L}^\infty_t(B^{\s,s})}.
 \end{split}
 \eeno

Finally using once again that $\dive v(t)=0$ and Lemma
\ref{lem:Berstein}, we obtain \beno
\begin{split}
\|\D_k^{\rm h}\D_\ell^{\rm v}R^{\rm h}\bar{T}^{\rm
v}(v^3,a)(t)\|_{L^2}\lesssim& 2^{k}\sum_{\substack{k'\geq k-3
\\|\ell'-\ell|\leq 4}}\|\D_{k'}^{\rm h}\D_{\ell'}^{\rm
v}v^3(t)\|_{L^2}\|\wt{\D}_{k'}^{\rm
h}S_{\ell'-1}^{\rm v}a\|_{L^\infty_t(L^2_{\rm h}(L^\infty_{\rm v}))}\\
\lesssim& 2^{k}\sum_{\substack{k'\geq k-3
\\|\ell'-\ell|\leq 4}}2^{-\ell'}\|\D_{k'}^{\rm h}\D_{\ell'}^{\rm
v}\dive_{\rm h}v^{\rm h}(t)\|_{L^2}\|\wt{\D}_{k'}^{\rm
h}S_{\ell'-1}^{\rm v}a\|_{L^\infty_t(L^2_{\rm h}(L^\infty_{\rm v}))}\\
\lesssim &d_{k,\ell}(t)2^{-k\s}2^{-\ell(1+s)}\|v^{\rm
h}(t)\|_{B^{2,\f12}}\|a\|_{\wt{L}^\infty_t(B^{\s,s})}.
 \end{split}
 \eeno
The same estimate holds for $\bar{T}^{\rm h}\bar{T}^{\rm
v}(v^3,a)(t).$ This completes the proof of \eqref{o.10} for
$\Phi=0.$

Exactly by the same manner to the proof of \eqref{o.10}, we can also
get \beq\label{o.11}
\begin{split}
\|\Delta_k^{\rm h}\Delta_\ell^{\rm v}\dive_{\rm h}[v^{\rm h}
a]_\Phi(t)\|_{L^2}\leq C2^{-k\s}2^{-\ell s}\Bigl(
d_{k,\ell}2^k&\|v_\Phi^{\rm h}(t)\|_{B^{1,\f12}}
+d_{k}d_\ell(t)2^k\|v_\Phi^{\rm
h}(t)\|_{B^{1,\f12}}\\
&\quad
+d_{k,\ell}(t)\|v_\Phi^h(t)\|_{B^{2,\f12}}\Bigr)\|a_\Phi\|_{\widetilde{L}^\infty_t(B^{\s,s})}.
\end{split} \eeq

By summing up \eqref{o.4}, \eqref{o.10} and \eqref{o.11}, we  write
\begin{align*}
\|\Delta_k^{\rm h}\Delta_\ell^{\rm v} a_\Phi\|_{L^\infty_t(L^2)}\le
& \|e^{\delta |D|}\Delta_k^{\rm h}\Delta_\ell^{\rm v} a_0\|_{L^2}
+\int_0^t\|\Delta_k^{\rm h}\Delta_\ell^{\rm v} f_\Phi(t')\|_{L^2}\,dt' \\
&+C2^{-k\s}2^{-\ell
s}\Bigl(d_{k,\ell}(2^k+2^\ell)\int_0^te^{-c\lambda\int_{t'}^t\dot\theta(\tau)\,d\tau(2^k+2^\ell)}\ve^{1-\al}
\|v_\Phi(t')\|_{B^{1,\f12}}\,dt'\\
&+d_{\ell}2^\ell\int_0^te^{-c\lambda\int_{t'}^t\dot\theta(\tau)\,d\tau 2^\ell}d_k(t')\ve^{1-\al}\|v_\Phi^3(t')\|_{B^{1,\f12}}\,dt'\\
&+d_{k}2^k\int_0^te^{-c\lambda\int_{t'}^t\dot\theta(\tau)\,d\tau 2^k}d_\ell(t')\ve^{1-\al}\|v_\Phi^{\rm h}(t')\|_{B^{1,\f12}}\,dt'\\
&+\int_0^t d_{k,\ell}(t')\ve^{1-\al}\|v_\Phi^{\rm
h}(t')\|_{B^{2,\f12}}\,dt'\Bigr)\|a_\Phi\|_{\widetilde{L}^\infty_t(B^{\s,s})}.
\end{align*}
Then \eqref{o.9} follows  by  Definition \ref{def2.1}, \eqref{g.4}
and
\begin{align*}
\int_0^te^{-c\lambda\int_{t'}^t\dot\theta(\tau)\,d\tau(2^k+2^\ell)}\ve^{1-\al}\|v_\Phi(t')\|_{B^{1,\f12}}\,dt'
\le& \int_0^te^{-c\lambda\int_{t'}^t\dot\theta(\tau)\,d\tau(2^k+2^\ell)}\dot\theta(t')\,dt'\\
\le& \f C \lambda\bigl(2^k+2^\ell\bigr)^{-1},
\end{align*}
and
\begin{align*}
&\sum_{k\in\Z}\int_0^te^{-c\lambda\int_{t'}^t\dot\theta(\tau)\,d\tau2^\ell}d_k(t')\ve^{1-\al}\|v_\Phi^3(t')\|_{B^{1,\f12}}\,dt'
\le
\int_0^te^{-c\lambda\int_{t'}^t\dot\theta(\tau)\,d\tau2^\ell}\dot\theta(t')\,dt'\le
\f C \lambda 2^{-\ell},\\
&\sum_{\ell\in\Z}\int_0^te^{-c\lambda\int_{t'}^t\dot\theta(\tau)\,d\tau2^k}d_\ell(t')\ve^{1-\al}\|v_\Phi^{\rm
h}(t')\|_{B^{1,\f12}}\,dt' \le
\int_0^te^{-c\lambda\int_{t'}^t\dot\theta(\tau)\,d\tau2^k}\dot\theta(t')\,dt'\le
\f C \lambda 2^{-k}.
\end{align*} This completes the proof of Proposition \ref{prop:transport}. \end{proof}

\begin{rmk}\label{rmk5.1}
We mention here that we can not prove the uniform estimate of
$a_\Phi$ in the isentropic Besov space
$\wt{L}^\infty_t(B^{\f32}_{2,1})$ as that in \cite{c-p-z, PZZ3}. The
main reason is that we can not use commutator's argument to prove
the propagation of analytic regularity for the transport equation.
\end{rmk}

\begin{lem}\label{lem5.1}
{\sl Let $v(t)$ be a smooth solenoidal vector field and $\ga\in
]0,1[.$ Let $\theta(T)\leq \f{\delta}{\lambda}$ and $\Phi$ be the
phase function given by \eqref{def:phase}. Then one has
\beq\label{o.15}
\begin{split}
\|\Delta_k^{\rm h}\Delta_\ell^{\rm v}[v\cdot\na
a]_\Phi&(t)\|_{L^2}\le
C2^{-k(1-\ga)}2^{-\ell\bigl(\f12+\ga\bigr)}\Bigl(
d_{k,\ell}(2^k+2^\ell)\|v_\Phi(t)\|_{B^{1,\f12}}\|a_\Phi\|_{\widetilde{L}^\infty_t(B^{1-\ga,\f12+\ga})}\\
&+\bigl(d_{k}d_\ell(t)2^k\|v_\Phi^{\rm h}(t)\|_{B^{1-\ga,\f12+\ga}}
+d_{k,\ell}(t)\|v_\Phi^{\rm
h}(t)\|_{B^{2-\ga,\f12+\ga}}\bigr)\|a_\Phi\|_{\widetilde{L}^\infty_t(B^{1,\f12})}\Bigr),
\end{split} \eeq
and \beq\label{o.15ad}
\begin{split}
\|\Delta_k^{\rm h}\Delta_\ell^{\rm v}[v\cdot\na
a]_\Phi&(t)\|_{L^2}\le
C2^{-k(1+\ga)}2^{-\ell\bigl(\f12-\ga\bigr)}\Bigl(
d_{k,\ell}(2^k+2^\ell)\|v_\Phi(t)\|_{B^{1,\f12}}\|a_\Phi\|_{\widetilde{L}^\infty_t(B^{1+\ga,\f12-\ga})}\\
&+\bigl(d_{k}(t)d_\ell2^\ell\|v_\Phi^{3}(t)\|_{B^{1+\ga,\f12-\ga}}
+d_{k,\ell}(t)\|v_\Phi^{\rm
h}(t)\|_{B^{2+\ga,\f12-\ga}}\bigr)\|a_\Phi\|_{\widetilde{L}^\infty_t(B^{1,\f12})}\Bigr),
\end{split} \eeq and \beq\label{o.16}
\begin{split}
\|\Delta_k^{\rm h}\Delta_\ell^{\rm v}[v\cdot&\na
a]_\Phi(t)\|_{L^2}\leq C2^{-k\ga}2^{-\ell\bigl(\f32-\ga\bigr)}\Bigl(
d_{k,\ell}(2^k+2^\ell)\|v_\Phi(t)\|_{B^{1,\f12}}\|a_\Phi\|_{\widetilde{L}^\infty_t(B^{\ga,\f32-\ga})}\\
&+d_{k,\ell}(t)\bigl(\|v^{\rm
h}_\Phi(t)\|_{B^{1,\f32}}\|a_\Phi\|_{\wt{L}^\infty_t(B^{1+\ga,\f12-\ga})}
+\|v_\Phi^{\rm
h}(t)\|_{B^{1+\ga,\f32-\ga}}\|a_\Phi\|_{\widetilde{L}^\infty_t(B^{1,\f12})}\bigr)\Bigr).
\end{split}
\eeq }
\end{lem}

\begin{proof} Once again similar to the proof of Lemma \ref{lem3.2}, it suffices to prove (\ref{o.15}-\ref{o.16}) for $\Phi=0.$ Indeed we first get, by using Bony's
decomposition for both horizontal and vertical variables, that \beno
v^3a=\bigl(T^{\rm h}T^{\rm v}+T^{\rm h}\cR^{\rm v}+\cR^{\rm h}T^{\rm
v}+\cR^{\rm h}\cR^{\rm v}\bigr)(v^3,a)(t). \eeno Note that \beno
\|S_{k'+2}^{\rm h}\D_{\ell'}^{\rm v}a\|_{L^\infty_t(L^\infty_{\rm
h}(L^2_{\rm v}))}\lesssim
d_{k,\ell'}2^{k'\ga}2^{-\ell'\bigl(\f12+\ga\bigr)}\|a\|_{\widetilde{L}^\infty_t(B^{1-\ga,\f12+\ga})},
\eeno from which, we deduce \beno
\begin{split}
\|\Delta_k^{\rm h}\Delta_\ell^{\rm v}\cR^{\rm h}T^{\rm
v}(v^3,a)(t)\|_{L^2}\lesssim &\sum_{\substack{k'\geq
k-N_0\\|\ell'-\ell|\leq 4}}\|\Delta_{k'}^{\rm h}S_{\ell'-1}^{\rm
v}v^3(t)\|_{L^2_{\rm h}(L^\infty_{\rm v})}\|S_{k'+2}^{\rm
h}\D_{\ell'}^{\rm v}a\|_{L^\infty_t(L^\infty_{\rm h}(L^2_{\rm v}))}\\
\lesssim
&d_{k,\ell}2^{-k(1-\ga)}2^{-\ell\bigl(\f12+\ga\bigr)}\|v^3(t)\|_{B^{1,\f12}}\|a\|_{\widetilde{L}^\infty_t(B^{1-\ga,\f12+\ga})}.
\end{split}
\eeno The same estimate holds for $T^{\rm h}T^{\rm v}(v^3,a)(t).$

While due to $\dive v=0,$ one has \beno \|S_{k'-1}^{\rm
h}\D_{\ell'}^{\rm v}v^3(t)\|_{L^\infty_{\rm h}(L^2_{\rm v})}\lesssim
2^{-\ell'} \|S_{k'-1}^{\rm h}\D_{\ell'}^{\rm v}\dive_{\rm h}v^{\rm
h}(t)\|_{L^\infty_{\rm h}(L^2_{\rm v})}\lesssim
d_{k,\ell}(t)2^{k'\ga}2^{-\ell'\bigl(\f32+\ga\bigr)}\|v^{\rm
h}(t)\|_{B^{2-\ga,\f12+\ga}}, \eeno from which, we infer \beno
\begin{split}
\|\Delta_k^{\rm h}\Delta_\ell^{\rm v}T^{\rm h}\cR^{\rm
v}(v^3,a)(t)\|_{L^2}\lesssim &\sum_{\substack{|k'- k|\leq
4\\\ell'\geq \ell-N_0}}\|S_{k'-1}^{\rm h}\D_{\ell'}^{\rm
v}v^3(t)\|_{L^\infty_{\rm h}(L^2_{\rm v})}\|\D_{k'}^{\rm
h}S_{\ell'+2}^{\rm v}a\|_{L^\infty_t(L^2_{\rm h}(L^\infty_{\rm v}))}\\
\lesssim &\sum_{\substack{|k'- k|\leq 4\\\ell'\geq
\ell-N_0}}d_{k',\ell'}(t)2^{-k'(1-\ga)}2^{-\ell'\bigl(\f32+\ga\bigr)}\|v^h(t)\|_{B^{2-\ga,\f12+\ga}}\|a\|_{\widetilde{L}^\infty_t(B^{1,\f12})}\\
\lesssim
&d_{k,\ell}(t)2^{-k(1-\ga)}2^{-\ell\bigl(\f32+\ga\bigr)}\|v^h(t)\|_{B^{2-\ga,\f12+\ga}}\|a\|_{\widetilde{L}^\infty_t(B^{1,\f12})}.
\end{split}
\eeno
The same estimate holds for $\cR^{\rm h}\cR^{\rm
v}(v^3,a)(t).$ Hence in view of Lemma \ref{lem:Berstein}, we obtain
\beq\label{o.12}
\begin{split}
\|\Delta_k^{\rm h}\Delta_\ell^{\rm v}\p_3(v^3 a)(t)\|_{L^2}\leq
C2^{-k(1-\ga)}2^{-\ell\bigl(\f12+\ga\bigr)}\Bigl(&
d_{k,\ell}2^\ell\|v^3(t)\|_{B^{1,\f12}}\|a\|_{\widetilde{L}^\infty_t(B^{1-\ga,\f12+\ga})}
\\
& +d_{k,\ell}(t)\|v^{\rm
h}(t)\|_{B^{2-\ga,\f12+\ga}}\|a\|_{\widetilde{L}^\infty_t(B^{1,\f12})}\Bigr).
\end{split} \eeq
Exactly following the same strategy, we can also prove
\beq\label{o.13}
\begin{split}
\|\Delta_k^{\rm h}\Delta_\ell^{\rm v}\dive_h(v^{\rm h}&
a)(t)\|_{L^2}\leq C2^{-k(1-\ga)}2^{-\ell\bigl(\f12+\ga\bigr)}\Bigl(
d_{k,\ell}2^k\|v^{\rm
h}(t)\|_{B^{1,\f12}}\|a\|_{\widetilde{L}^\infty_t(B^{1-\ga,\f12+\ga})}
\\
& +\bigl(2^kd_kd_\ell(t)\|v^{\rm
h}(t)\|_{B^{1-\ga,\f12+\ga}}+d_{k,\ell}(t)\|v^{\rm
h}(t)\|_{B^{2-\ga,\f12+\ga}}\bigr)
\|a\|_{\widetilde{L}^\infty_t(B^{1,\f12})}\Bigr),
\end{split} \eeq
and \beq\label{o.12ad}
\begin{split}
\|\Delta_k^{\rm h}\Delta_\ell^{\rm v}\p_3(v^3 a)&(t)\|_{L^2}\leq
C2^{-k(1+\ga)}2^{-\ell\bigl(\f12-\ga\bigr)}\Bigl(
d_{k,\ell}2^\ell\|v^3(t)\|_{B^{1,\f12}}\|a\|_{\widetilde{L}^\infty_t(B^{1+\ga,\f12-\ga})}
\\
&+\bigl(2^\ell d_k(t)d_\ell\|v^{3}(t)\|_{B^{1+\ga,\f12-\ga}}
+d_{k,\ell}(t)\|v^{\rm
h}(t)\|_{B^{2+\ga,\f12-\ga}}\bigr)\|a\|_{\widetilde{L}^\infty_t(B^{1,\f12})}\Bigr),
\end{split} \eeq
and \beq\label{o.13ad}
\begin{split}
\|\Delta_k^{\rm h}\Delta_\ell^{\rm v}\dive_h(v^{\rm h}
a)(t)\|_{L^2}\leq C2^{-k(1+\ga)}2^{-\ell\bigl(\f12-\ga\bigr)}\Bigl(&
d_{k,\ell}2^k\|v^{\rm
h}(t)\|_{B^{1,\f12}}\|a\|_{\widetilde{L}^\infty_t(B^{1+\ga,\f12-\ga})}
\\
&+d_{k,\ell}(t)\|v^{\rm h}(t)\|_{B^{2+\ga,\f12-\ga}}
\|a\|_{\widetilde{L}^\infty_t(B^{1,\f12})}\Bigr).
\end{split} \eeq

Combining \eqref{o.12} with \eqref{o.13}, we conclude the proof of
\eqref{o.15} for $\Phi=0.$ Whereas by summing up \eqref{o.12ad} and
\eqref{o.13ad}, we achieve \eqref{o.15ad}.

On the other hand, since $\ga\in ]0,1[,$ one has \beno
\|S_{k'+2}^{\rm h}\D_{\ell'}^{\rm v}a\|_{L^\infty_t(L^\infty_{\rm
h}(L^2_{\rm v}))}\lesssim
d_{k,\ell}2^{k'(1-\ga)}2^{-\ell'\bigl(\f32-\ga\bigr)}\|a\|_{\wt{L}^\infty_t(B^{\ga,\f32-\ga})},
\eeno which ensures \beno
\begin{split}
\|\Delta_k^{\rm h}\Delta_\ell^{\rm v}\cR^{\rm h}T^{\rm
v}(v^3,a)(t)\|_{L^2}\lesssim &\sum_{\substack{k'\geq
k-N_0\\|\ell'-\ell|\leq 4}}\|\Delta_{k'}^{\rm h}S_{\ell'-1}^{\rm
v}v^3(t)\|_{L^2_{\rm h}(L^\infty_{\rm v})}\|S_{k'+2}^{\rm
h}\D_{\ell'}^{\rm v}a\|_{L^\infty_t(L^\infty_{\rm h}(L^2_{\rm v}))}\\
\lesssim
&d_{k,\ell}2^{-k\ga}2^{-\ell\bigl(\f32-\ga\bigr)}\|v^3(t)\|_{B^{1,\f12}}\|a\|_{\widetilde{L}^\infty_t(B^{\ga,\f32-\ga})}.
\end{split}
\eeno The same estimate holds for $T^{\rm h}T^{\rm v}(v^3,a)(t).$

While again due to $\dive v=0,$ one has \beno \begin{split}
\|S_{k'-1}^{\rm h}\D_{\ell'}^{\rm v}v^3(t)\|_{L^\infty_{\rm
h}(L^2_{\rm v})}\lesssim & 2^{-\ell'} \|S_{k'-1}^{\rm
h}\D_{\ell'}^{\rm v}\dive_{\rm
h}v^{\rm h}(t)\|_{L^\infty_{\rm h}(L^2_{\rm v})}\\
\lesssim
&d_{k,\ell}(t)2^{k'(1-\ga)}2^{-\ell'\bigl(\f52-\ga\bigr)}\|v^{\rm
h}(t)\|_{B^{1+\ga,\f32-\ga}},\end{split} \eeno from which, we deduce
\beno
\begin{split}
\|\Delta_k^{\rm h}\Delta_\ell^{\rm v}T^{\rm h}\cR^{\rm
v}(v^3,a)(t)\|_{L^2}\lesssim &\sum_{\substack{|k'- k|\leq
4\\\ell'\geq \ell-N_0}}\|S_{k'-1}^{\rm h}\D_{\ell'}^{\rm
v}v^3(t)\|_{L^\infty_{\rm h}(L^2_{\rm v})}\|\D_{k'}^{\rm
h}S_{\ell'+2}^{\rm v}a\|_{L^\infty_t(L^2_{\rm h}(L^\infty_{\rm v}))}\\
\lesssim
&d_{k,\ell}(t)2^{-k\ga}2^{-\ell\bigl(\f52-\ga\bigr)}\|v^3\|_{B^{1+\ga,\f32-\ga}}\|a\|_{\widetilde{L}^\infty_t(B^{1,\f12})}.
\end{split}
\eeno
 The same estimate holds for $\cR^{\rm h}\cR^{\rm
v}(v^3,a)(t).$ We thus obtain \beq\label{o.17}
\begin{split}
\|\Delta_k^{\rm h}\Delta_\ell^{\rm v}\p_3(v^3 a)(t)\|_{L^2}\leq
C2^{-k\ga}2^{-\ell\bigl(\f32-\ga\bigr)}\Bigl(&
d_{k,\ell}2^\ell\|v^3(t)\|_{B^{1,\f12}}\|a\|_{\widetilde{L}^\infty_t(B^{\ga,\f32-\ga})}
\\
& +d_{k,\ell}(t)\|v^{\rm
h}(t)\|_{B^{1+\ga,\f32-\ga}}\|a\|_{\widetilde{L}^\infty_t(B^{1,\f12})}\Bigr).
\end{split} \eeq
The same argument assures that \beq\label{o.18}
\begin{split}
\|\Delta_k^{\rm h}&\Delta_\ell^{\rm v}\dive_h(v^{\rm h}
a)(t)\|_{L^2}\leq C2^{-k\ga}2^{-\ell\bigl(\f32-\ga\bigr)}\Bigl(
d_{k,\ell}2^k\|v^{\rm
h}\|_{B^{1,\f12}}\|a\|_{\widetilde{L}^\infty_t(B^{\ga,\f32-\ga})}
\\
& +d_{k,\ell}(t)\bigl(\|v^{\rm
h}(t)\|_{B^{1,\f32}}\|a_\Phi\|_{\wt{L}^\infty_t(B^{1+\ga,\f12-\ga})}+\|v^{\rm
h}(t)\|_{B^{1+\ga,\f32-\ga}}
\|a\|_{\widetilde{L}^\infty_t(B^{1,\f12})}\bigr)\Bigr).
\end{split} \eeq
By summing up \eqref{o.17} and \eqref{o.18}, we complete the proof
of \eqref{o.16}, and also the lemma.
\end{proof}

With Lemma \ref{lem5.1}, we deduce from the proof of Proposition
\ref{prop:transport} that
\begin{prop}\label{prop5.2}
{\sl Let $a$ be a smooth enough solution of \eqref{eq:trans} on
$[0,T].$ Then under the assumptions of Lemma \ref{lem5.1}, for any $
t\in ]0,T[,$ we have \beq \label{o.19}
\begin{split}
\|a_\Phi\|_{\widetilde{L}^\infty_t(B^{1-\ga,\f12+\ga})}\le&
\|e^{\delta|D|}a_0\|_{B^{1-\ga,\f12+\ga}}
+\|f_\Phi\|_{L^1_t(B^{1-\ga,\f12+\ga})}+\f
C\lambda\bigl(\|a_\Phi\|_{\widetilde{L}^\infty_t(B^{1-\ga,\f12+\ga})}
\\
&+\|a_\Phi\|_{\widetilde{L}^\infty_t(B^{1,\f12})}\bigr)+C\ve^{1-\al}\|v_\Phi^{\rm
h}\|_{L^1_t(B^{2-\ga,\f12+\ga})}\|a_\Phi\|_{\widetilde{L}^\infty_t(B^{1,\f12})},
\end{split} \eeq
and
 \beq \label{o.19ad}
\begin{split}
\|a_\Phi\|_{\widetilde{L}^\infty_t(B^{1+\ga,\f12-\ga})}\le&
\|e^{\delta|D|}a_0\|_{B^{1+\ga,\f12-\ga}}
+\|f_\Phi\|_{L^1_t(B^{1+\ga,\f12-\ga})}+\f
C\lambda\bigl(\|a_\Phi\|_{\widetilde{L}^\infty_t(B^{1+\ga,\f12-\ga})}
\\
&+\|a_\Phi\|_{\widetilde{L}^\infty_t(B^{1,\f12})}\bigr)+C\ve^{1-\al}\|v_\Phi^{\rm
h}\|_{L^1_t(B^{2+\ga,\f12-\ga})}\|a_\Phi\|_{\widetilde{L}^\infty_t(B^{1,\f12})},
\end{split} \eeq and \beq \label{o.20}
\begin{split}
\|a_\Phi&\|_{\widetilde{L}^\infty_t(B^{\ga,\f32-\ga})}\le
\|e^{\delta|D|}a_0\|_{B^{\ga,\f32-\ga}}+\|f_\Phi\|_{L^1_t(B^{\ga,\f32-\ga})}
+\f C\lambda\|a_\Phi\|_{\widetilde{L}^\infty_t(B^{\ga,\f32-\ga})}\\
&\qquad+C\ve^{1-\al}\bigl(\|v^{\rm
h}_\Phi\|_{L^1_t(B^{1,\f32})}\|a_\Phi\|_{\wt{L}^\infty_t(B^{1+\ga,\f12-\ga})}+\|v_\Phi\|_{L^1_t(B^{1+\ga,\f32-\ga})}\|a_\Phi\|_{\widetilde{L}^\infty_t(B^{1,\f12})}\bigr).
\end{split} \eeq}
\end{prop}

\setcounter{equation}{0}
\section{Elliptic estimates in the analytical class}
In this section, we present the estimates of the pressure function
in the analytical class. Recall that the re-scaled pressure function
$q$ satisfies \beq\label{l.14} -\textrm{div}\Big(\f 1
{1+\ve^{\be}a}\na^\ve q\Big) =\ve^{1-\al}\textrm{div}(v\cdot\na
v)-\textrm{div}\Big(\f 1 {1+\ve^{\be}a}\Delta_\ve v\Big). \eeq In
the sequel, we always denote $G(r)\eqdefa\f {r} {1+r},$ and
$\theta(t), \Phi(t), \Psi(t)$ to be given by \eqref{g.4},
\eqref{def:phase} and \eqref{psias} respectively. Moreover, we
always assume that $\theta(T)\leq\f{\delta}{\lambda}. $
\begin{prop}\label{prop6.1}
{\sl Let $\al\in \bigl]0,\f12\bigr[, \be>\al$ and $0<\ga\le
\f12\min\bigl(\be-\al,1-2\al\bigr)$. Then there exists a positive
constant $C_0$ such that for $\epsilon$ given by \eqref{k.50}, if
$a$ satisfies \beq\label{l.14ad}
\|a_\Phi\|_{\widetilde{L}^\infty_t(B^{1,\f12})}\le K\andf \ve\leq
\min\Bigl(\Bigl(\f1{2C_0K}\Bigr)^{\f 1{\be-\ga}}, \Bigl(\f{\epsilon}{K}\Bigr)^{\f1{\be}}\Bigr),
\eeq we have \beq\label{l.15}
\begin{split} \ve^{1-\al}\|q\|_{Y_t}\le &
C\max\bigl(\ve^{\be-\al-2\ga},
\ve^{1-2\al-2\ga}\bigr)\theta(t)\Psi(t)\with\\
 \|q\|_{Y_t}\eqdefa &\|\na_{\rm h}
q_\Phi\|_{L^1_t(B^{-1,\f12})}+\|\na_\ve
q_\Phi\|_{L^1_t(B^{-1+\ga,\f12-\ga})}. \end{split} \eeq }
\end{prop}

\begin{proof}\, In view  of \eqref{l.14} and $\dive v=0,$ we write \beq\label{l.2}
\begin{split} q=&-(-\Delta_\ve)^{-1}\na_\ve\cdot\bigl(G(\ve^{\be}a)\na_\ve q\bigr)+\ve^{1-\al}(-\Delta_\ve)^{-1}
\textrm{div}_{\rm h}\textrm{div}_{\rm h}(v^{\rm h}\otimes v^{\rm h})\\
&+2\ve^{1-\al}(-\Delta_\ve)^{-1}\pa_3\textrm{div}_{\rm h}(v^3v^{\rm h})-2\ve^{1-\al}(-\Delta_\ve)^{-1}\pa_3(v^3\textrm{div}_hv^{\rm h})\\
&+(-\Delta_\ve)^{-1}\textrm{div}\bigl(G(\ve^{\be}a)\Delta_\ve
v\bigr)\eqdefa q_1+\cdots+q_5.
\end{split} \eeq

To avoid the difficulty of product laws in the Bessov space
$B^{-1}_{2,1}(\R^2),$ we write \beq \label{o.7}
\begin{split}
\|\na_{\rm
h}&[q_1]_\Phi\|_{L^1_t(B^{-1,\f12})}\\
=&\ve^{-\ga}\bigl\||D_{\rm
h}|^{-\ga}|\ve D_3|^{\ga}\na_{\rm h}(-\D_\ve)^{-1}\na_\ve\cdot|D_{\rm
h}|^{\ga}| D_3|^{-\ga}\bigl[G(\ve^\be a)\na_\ve
q\bigr]_\Phi\bigr\|_{L^1_t(B^{-1,\f12})}\\
\leq & C\ve^{-\ga}\bigl\||D_{\rm
h}|^{\ga}|D_3|^{-\ga}\bigl[G(\ve^\be a)\na_\ve
q\bigr]_\Phi\bigr\|_{L^1_t(B^{-1,\f12})}\\
\leq & C\ve^{-\ga}\big\|\bigl[G(\ve^\be a)\na_\ve
q\bigr]_\Phi\big\|_{L^1_t(B^{-1+\ga,\f12-\ga})},
\end{split}
\eeq where $|D_{\rm h}|$ and $|D_3|$ denote the Fourier multipliers
with symbols $|\xi_{\rm h}|=\sqrt{\xi_1^2+\xi_2^2}$ and $|\xi_3|$
respectively. In what follows, we shall frequently use this kind of
tricks to deal with the estimate of the pressure function.

In view of \eqref{o.7}, if $\ve$ is so small that $\ve^\be K\leq
\epsilon,$ we get, by applying Corollary \ref{col3.1} and Lemma
\ref{lem:composition}, that
\begin{align*}
\|q_1\|_{Y_t}&\le C\ve^{-\ga}\bigl\|\big[G(\ve^\be a)\na_\ve q\big]_\Phi\bigr\|_{L^1_t(B^{-1+\ga,\f12-\ga})}\\
&\le
C\ve^{\be-\ga}\|a_\Phi\|_{\widetilde{L}^\infty_t(B^{1,\f12})}\|\na_\ve
q_\Phi\|_{{L}^1_t(B^{-1+\ga,\f12-\ga})}.
\end{align*}
Applying the law of product of  Corollary \ref{col3.1} gives
\begin{align*}
\|q_2\|_{Y_t} &\le C\ve^{1-\al}\Bigl(\big\|\big[v^{\rm h}v^{\rm
h}\big]_\Phi\big\|_{L^1_t(B^{0,\f12})}
+\big\|\big[v^{\rm h}v^{\rm h}\big]_\Phi\bigr\|_{L^1_t(B^{\ga,\f12-\ga})}\Bigr)\\
&\le C\ve^{1-\al}\|v^{\rm h}_\Phi\|_{L^1_t(B^{1,\f12})}\big(\|v^{\rm
h}_\Phi\|_{\widetilde{L}^\infty_t(B^{\ga,\f12-\ga})}+\|v^{\rm
h}_\Phi\|_{\widetilde{L}^\infty_t(B^{0,\f12})}\big),
\end{align*}
and
\begin{align*}
\|q_3\|_{Y_t} &\le C\ve^{-\al}\Bigl(\|\big[v^3v^{\rm h}\big]_\Phi\big\|_{L^1_t(B^{0,\f12})}+\big\|\big[v^3v^{\rm h}\big]_\Phi\big\|_{L^1_t(B^{\ga,\f12-\ga})}\Bigr)\\
&\le C\ve^{-\al}\|v^3_\Phi\|_{L^1_t(B^{1,\f12})}\big(\|v^{\rm
h}_\Phi\|_{\widetilde{L}^\infty_t(B^{\ga,\f12-\ga})}+\|v^{\rm
h}_\Phi\|_{\widetilde{L}^\infty_t(B^{0,\f12})}\big).
\end{align*}
Whereas we get, by  applying first the similar trick as that in
\eqref{o.7} and then Corollary \ref{col3.1}, that
\begin{align*}
\|q_4\|_{Y_t}\le C\ve^{-\al-\ga}\big\|\big[v^3\textrm{div}_hv^{\rm
h}\big]_\Phi\|_{L^1_t(B^{-1+\ga,\f12-\ga})}\le
C\ve^{-\al-\ga}\|v^3_\Phi\|_{L^1_t(B^{1,\f12})}\|v^{\rm
h}_\Phi\|_{\widetilde{L}^\infty_t(B^{\ga,\f12-\ga})}.
\end{align*}

To handle $q_5$ in \eqref{l.2}, we  split it further  as
\beq\label{eq:q5}\begin{split}
q_5=&(-\Delta_\ve)^{-1}\textrm{div}_{\rm
h}\bigl(G(\ve^{\be}a)\Delta_{\rm h}v^{\rm h}\bigr)
+(-\Delta_\ve)^{-1}\textrm{div}_{\rm h}\bigl(G(\ve^{\be}a)\ve^2\pa_3^2v^{\rm h}\bigr)\\
&+(-\Delta_\ve)^{-1}\pa_3\bigl(G(\ve^{\be}a)\Delta_{\rm
h}v^3\bigr)+(-\Delta_\ve)^{-1}\pa_3\bigl(G(\ve^{\be}a)\ve^2\pa_3^2v^3\bigr)\\
\eqdefa &q_{5,1}+\cdots+q_{5,4}.
\end{split}\eeq
Similar to the estimate of $q_1,$ one has \beno
&&\|q_{5,1}\|_{Y_t}\le
C\ve^{\be-\ga}\|a_\Phi\|_{\widetilde{L}^\infty_t(B^{1,\f12})}\|v^{\rm
h}_\Phi\|_{L^1_t(B^{1+\ga,\f12-\ga})},\\
&&\|q_{5,3}\|_{Y_t}\le
C\ve^{-1+\be-\ga}\|a_\Phi\|_{\widetilde{L}^\infty_t(B^{1,\f12})}\|v^3_\Phi\|_{L^1_t(B^{1+\ga,\f12-\ga})}.
\eeno While note that
$$\longformule{ (-\Delta_\ve)^{-1}\textrm{div}_{\rm h}\big(G(\ve^\be
a)\ve^2\pa_3^2v^{\rm h}\big)}{{}=\ve^{-1+\delta}|D_{\rm
h}|^{-1+\delta}|\ve
D_3|^{1-\delta}(-\Delta_\ve)^{-1}\textrm{div}_{\rm h}|D_{\rm
h}|^{1-\delta}|D_3|^{-1+\delta}\big(G(\ve^\be a)\ve^2\pa_3^2v^{\rm
h}\big), } $$
for $\delta$ equals $\ga$ and $2\ga,$  we infer \beno
\begin{split}
\|[q_{5,2}]_\Phi\|_{Y_t}\le & \ve^{-1+\ga}\big\|\bigl[G(\ve^\be
a)\ve^2\p_3^2v^{\rm h}\bigr]_\Phi\big\|_{L^1_t(B^{-\ga,-\f12+\ga})}\\
\le &
C\ve^{1+\be+\ga}\|a_\Phi\|_{\widetilde{L}^\infty_t(B^{1,\f12})}\|v^{\rm
h}_\Phi\|_{L^1_t(B^{-\ga,\f32+\ga})}.
\end{split}
\eeno By a similar manner and using $\dive v=0,$ one has
\begin{align*}
\|q_{54}\|_{Y_t}\le &C\ve^{1-\ga}\big\|\bigl[G(\ve^\be
a)\p_3\dive_{\rm h}v^{\rm h}\bigr]_\Phi\big\|_{L^1_t(B^{-1+\ga,\f12-\ga})}\\
\leq &
C\ve^{1+\be-\ga}\|a_\Phi\|_{\widetilde{L}^\infty_t(B^{1,\f12})}\|v^{\rm
h}_\Phi\|_{L^1_t(B^{\ga,\f32-\ga})}.
\end{align*}

By summing up the above estimates,  we arrive at
\begin{align*}
\ve^{1-\al}\|q\|_{Y_t}\le &
C\ve^{\be-\al}\|a_\Phi\|_{\widetilde{L}^\infty_t(B^{1,\f12})}\Bigl(\ve^{1-\ga}\bigl(\|q\|_{Y_t}+\|v^{\rm
h}_\Phi\|_{L^1_t(B^{1+\ga,\f12-\ga})}\bigr)\\
&+\ve^{-\ga}\|v^{3}_\Phi\|_{L^1_t(B^{1+\ga,\f12-\ga})}+\ve^{2-\ga}\|v^{\rm
h}_\Phi\|_{L^1_t(B^{\ga,\f32-\ga})}+\ve^{2}\|v^{\rm
h}_\Phi\|_{L^1_t(B^{-\ga,\f32+\ga})}\Bigr)\\
&+C\ve^{1-2\al}\Bigl(\ve\|v^{\rm
h}_\Phi\|_{L^1_t(B^{1,\f12})}+\ve^{-\ga}\|v^{3}_\Phi\|_{L^1_t(B^{1,\f12})}\Bigr)\\
&\qquad\qquad\qquad\qquad\quad\qquad\qquad\times\Bigl(\|v^{\rm
h}_\Phi\|_{\wt{L}^\infty_t(B^{\ga,\f12-\ga})}+\|v^{\rm
h}_\Phi\|_{\wt{L}^\infty_t(B^{0,\f12})}\Bigr).
\end{align*}
While we get, by applying Lemma \ref{lem3.1}, that \beno
\ve^{2+\be-\al-\ga}\|v^{\rm h}_\Phi\|_{L^1_t(B^{\ga,\f32-\ga})}\leq
C\ve^{\be-2\ga} \ve^{2-\al +\ga}\bigl(\|v^{\rm
h}_\Phi\|_{L^1_t(B^{-\ga,\f32+\ga})}+\|v^{\rm
h}_\Phi\|_{L^1_t(B^{1+\ga,\f12-\ga})}\bigr). \eeno Then  due to the
assumptions  of $\al,\be,\ga$ in the proposition, \eqref{l.15}
follows by choosing $\epsilon$ suitably small in \eqref{l.14ad}.
\end{proof}

\begin{prop}\label{prop6.2}
{\sl Let $\al\in ]0,1[, \be>\al$ and $0<\ga\le \min\bigl(\f
{\be-\al} 4, \f {1-\al} 3\bigr)$. Then there exists some positive
constant $C_0$ such that for $\epsilon$ given by \eqref{k.50}, if
$a$ satisfies \beq\label{l.1before}
\|a_\Phi\|_{\widetilde{L}^\infty_t(B^{1,\f12})}+\|a_\Phi\|_{\widetilde{L}^\infty_t(B^{1-\ga,\f12+\ga})}\le
K \andf \ve^\be\leq \min\Bigl(\f1{2C_0K},\f{\epsilon}K\Bigr), \eeq
there holds \beq\label{l.1} \ve^{2\al+\ga}\|q\|_{Z_t} \le
C\Psi^2(t)\with \|q\|_{Z_t}\eqdefa\|\na_\ve
q_\Phi\|_{L^1_t(B^{\ga,\f12-\ga})} + \|\na_\ve
q_\Phi\|_{L^1_t(B^{-\ga,\f12+\ga})}. \eeq}
\end{prop}

\begin{proof}\, Following the same line to the  proof of Proposition \ref{prop6.1}, we
shall split the proof of \eqref{l.1} into the following steps:

$\bullet$ \underline{Estimate of $\na_\ve q_1$}

By virtue of \eqref{l.2}, we get, by applying Corollary
\ref{col3.1}, that \beno \begin{split}\|\na_\ve
[q_1]_\Phi\|_{L^1_t(B^{\ga,\f12-\ga})}\lesssim \|[G(\ve^\beta
a)\na_\ve q]_\Phi\| _{L^1_t(B^{\ga,\f12-\ga})}\lesssim &
\|[G(\ve^\beta a)]_\Phi\|_{\wt{L}^\infty_t(B^{1,\f12})}\|\na_\ve
q_\Phi\| _{L^1_t(B^{\ga,\f12-\ga})}, \end{split} \eeno and \beno
\begin{split} \|\na_\ve &
[q_1]_\Phi\|_{L^1_t(B^{-\ga,\f12+\ga})}\lesssim \|[G(\ve^\beta
a)\na_\ve q]_\Phi\| _{L^1_t(B^{-\ga,\f12+\ga})}\\
&\lesssim  \|[G(\ve^\beta
a)]_\Phi\|_{\wt{L}^\infty_t(B^{1,\f12})}\|\na_\ve q_\Phi\|
_{L^1_t(B^{-\ga,\f12+\ga})}+ \|[G(\ve^\beta
a)]_\Phi\|_{\wt{L}^\infty_t(B^{1-\ga,\f12+\ga})}\|\na_\ve q_\Phi\|
_{L^1_t(B^{0,\f12})}. \end{split} \eeno While it follows from Lemma
\ref{lem3.1} that \beno \|\na_\ve q_\Phi\|
_{L^1_t(B^{0,\f12})}\lesssim \|\na_\ve q_\Phi\|
_{L^1_t(B^{\ga,\f12-\ga})}+ \|\na_\ve q_\Phi\|
_{L^1_t(B^{-\ga,\f12+\ga})}. \eeno Therefore, if $\ve$ is so small
that $\ve^\be K\leq \epsilon,$ by applying Lemma
\ref{lem:composition}, we obtain \beq\label{l.3} \|q_1\|_{Z_t}\le
C\ve^\be\big(\|a_\Phi\|_{\widetilde{L}^\infty_t(B^{1,\f12})}+\|a_\Phi\|_{\widetilde{L}^\infty_t(B^{1-\ga,\f12+\ga})}\big)\|q\|_{Z_t}.
\eeq

$\bullet$ \underline{Estimate of $\na_\ve q_2$}

Applying the law of product of Corollary \ref{col3.1} and Lemma
\ref{lem3.2} yields that \beno
\begin{split}
\|\na_\ve[q_2]_\Phi\| _{L^1_t(B^{\ga,\f12-\ga})}\lesssim
&\ve^{1-\al}\|[v^{\rm h}\otimes v^{\rm
h}]_\Phi\|_{L^1_t(B^{1+\ga,\f12-\ga})}\\
\lesssim &\ve^{1-\al}\|v^{\rm
h}_\Phi\|_{\widetilde{L}^\infty_t(B^{0,\f12})} \|v^{\rm
h}_\Phi\|_{L^1_t(B^{2+\ga,\f12-\ga})},
\end{split}
\eeno and   \beno
\begin{split}
\|\na_\ve[q_2]_\Phi\| _{L^1_t(B^{-\ga,\f12+\ga})}\lesssim
&\ve^{1-\al}\|[v^{\rm h}\otimes v^{\rm
h}]_\Phi\|_{L^1_t(B^{1-\ga,\f12+\ga})}\\
\lesssim &\ve^{1-\al}\bigl(\|v^{\rm
h}_\Phi\|_{\widetilde{L}^\infty_t(B^{0,\f12})} \|v^{\rm
h}_\Phi\|_{L^1_t(B^{2-\ga,\f12+\ga})}+\|v^{\rm
h}_\Phi\|_{\widetilde{L}^\infty_t(B^{-\ga,\f12+\ga})} \|v^{\rm
h}_\Phi\|_{L^1_t(B^{2,\f12})}\bigr).
\end{split}
\eeno This gives rise to \beq\label{l.4}
\begin{split}
\|q_2\|_{Z_t}\le C\ve^{1-\al}\Big(\|v^{\rm
h}_\Phi\|_{\widetilde{L}^\infty_t(B^{0,\f12})}\bigl(\|v^{\rm
h}_\Phi\|_{L^1_t(B^{2+\ga,\f12-\ga})}&+\|v^{\rm
h}_\Phi\|_{L^1_t(B^{2-\ga,\f12+\ga})}\bigr)\\
& +\|v^{\rm
h}_\Phi\|_{\widetilde{L}^\infty_t(B^{-\ga,\f12+\ga})}\|v^{\rm
h}_\Phi\|_{L^1_t(B^{2,\f12})}\Bigr).
\end{split} \eeq

$\bullet$ \underline{Estimate of $\na_\ve q_3$}

To deal with $\na_\ve q_3$ given by \eqref{l.2}, we first use Bony's
decomposition \eqref{pd} for the vertical variable to split it as
\beq\label{p.3}
q_3=\ve^{1-\al}(-\Delta_\ve)^{-1}\pa_3\textrm{div}_{\rm h}(T^{\rm
v}({v^3},v^{\rm
h}))+\ve^{1-\al}(-\Delta_\ve)^{-1}\pa_3\textrm{div}_{\rm h}(\cR^{\rm
v}(v^3,{v^{\rm h}}))\eqdefa q_{31}+q_{32}. \eeq

Applying Lemma \ref{lem:Berstein} and $\dive v=0$ yields \beno
\begin{split}
\|\D_{k}^{\rm h}\D_\ell^{\rm v}\cR^{\rm h}\cR^{\rm
v}(v^3,v^h)\|_{L^1_t(L^2)}\lesssim &\sum_{\substack{k'\geq
k-N_0\\\ell'\geq \ell-N_0}}\|\D_{k'}^{\rm h}\D_{\ell'}^{\rm
v}v^3\|_{L^1_t(L^2)}\|S_{k'+2}^{\rm h}S_{\ell'+2}^{\rm v}v^{\rm
h}\|_{L^\infty_t(L^\infty)}\\
\lesssim &\sum_{\substack{k'\geq k-N_0\\\ell'\geq
\ell-N_0}}2^{-\ell'}\|\D_{k'}^{\rm h}\D_{\ell'}^{\rm v}\dive_{\rm h}
v^{\rm h}\|_{L^1_t(L^2)}\|S_{k'+2}^{\rm h}S_{\ell'+2}^{\rm v}v^{\rm
h}\|_{L^\infty_t(L^\infty)}\\
\lesssim& d_{k,\ell}2^{-k\ga}2^{-\ell\bigl(\f32-\ga\bigr)}\|v^{\rm
h}\|_{L^1_t(B^{2,\f12})}\|v^{\rm
h}\|_{\wt{L}^\infty_t(B^{\ga,\f12-\ga})}.
\end{split}
\eeno The same estimate holds for $T^{\rm h}\cR^{\rm v}(v^3,v^h).$
This gives \beno \|\cR^{\rm
v}(v^3,v^h)\|_{L^1_t(B^{\ga,\f32-\ga})}\lesssim \|v^{\rm
h}\|_{L^1_t(B^{2,\f12})}\|v^{\rm
h}\|_{\wt{L}^\infty_t(B^{\ga,\f12-\ga})}. \eeno In view of
\eqref{k.3ad}, similar estimate holds for $[\cR^{\rm
v}(v^3,v^h)]_\Phi,$ which ensures \beq\label{l.5}
\begin{split}
\|\na_\ve[q_{32}]_\Phi\|_{L^1_t(B^{\ga,\f12-\ga})} \lesssim
&\ve^{1-\al}  \|[\cR^{\rm
v}(v^3,v^h)]_\Phi\|_{L^1_t(B^{\ga,\f32-\ga})}\\
\lesssim & \ve^{1-\al} \|v^{\rm
h}_\Phi\|_{L^1_t(B^{2,\f12})}\|v^{\rm
h}_\Phi\|_{\wt{L}^\infty_t(B^{\ga,\f12-\ga})}. \end{split} \eeq

Again due to $\dive v=0,$ we have \beno \begin{split}
\|S_{k'-1}^{\rm h}\D_{\ell'}^{\rm v}v^3\|_{L^1_t(L^\infty_{\rm
h}(L^2_{\rm v}))}\lesssim& 2^{-\ell'}\|S_{k'-1}^{\rm
h}\D_{\ell'}^{\rm
v}\dive_{\rm h}v^{\rm h}\|_{L^1_t(L^\infty_{\rm h}(L^2_{\rm v}))}\\
\lesssim
&d_{k',\ell'}2^{k'\ga}2^{-\ell'\bigl(\f32+\ga\bigr)}\|v^{\rm
h}\|_{L^1_t(B^{2-\ga,\f12+\ga})}, \end{split} \eeno which implies
\beno
\begin{split}
\|\D_{k}^{\rm h}\D_\ell^{\rm v}T^{\rm h}\cR^{\rm
v}(v^3,v^h)\|_{L^1_t(L^2)}\lesssim &\sum_{\substack{|k'- k|\leq
4\\\ell'\geq \ell-N_0}}\|S_{k'-1}^{\rm h}\D_{\ell'}^{\rm
v}v^3\|_{L^1_t(L^\infty_{\rm h}(L^2_{\rm v}))}\|\D_{k'}^{\rm
h}S_{\ell'+2}^{\rm v}v^{\rm
h}\|_{L^\infty_t(L^2_{\rm h}(L^\infty_{\rm v}))}\\
\lesssim& d_{k,\ell}2^{k\ga}2^{-\ell\bigl(\f32+\ga\bigr)}\|v^{\rm
h}\|_{L^1_t(B^{2-\ga,\f12+\ga})}\|v^{\rm
h}\|_{\wt{L}^\infty_t(B^{0,\f12})}.
\end{split}
\eeno The same estimate holds for $R^{\rm h}\cR^{\rm v}(v^3,v^h)$
and $\bar{T}^{\rm h}\cR^{\rm v}(v^3,v^h).$ This leads to \beno
\|\cR^{\rm v}(v^3,v^{\rm h})\|_{L^1_t(B^{-\ga,\f32+\ga})}\lesssim
\|v^{\rm h}\|_{L^1_t(B^{2-\ga,\f12+\ga})}\|v^{\rm
h}\|_{\wt{L}^\infty_t(B^{0,\f12})}. \eeno Similar estimate holds for
$[\cR^{\rm v}(v^3,v^{\rm h})]_\Phi,$ which implies
 \beq\label{l.6}
\begin{split}
\|\na_\ve[q_{32}]_\Phi\|_{L^1_t(B^{-\ga,\f12+\ga})} \lesssim
&\ve^{1-\al}  \|[\cR^{\rm
v}(v^3,v^h)]_\Phi\|_{L^1_t(B^{-\ga,\f32+\ga})}\\
\lesssim & \ve^{1-\al}  \|v^{\rm
h}_\Phi\|_{L^1_t(B^{2-\ga,\f12+\ga})}\|v^{\rm
h}_\Phi\|_{\wt{L}^\infty_t(B^{0,\f12})}. \end{split} \eeq

Combining \eqref{l.5} with \eqref{l.6}, we obtain \beq\label{l.7}
\|q_{32}\|_{Z_t}\lesssim \ve^{1-\al}\Bigl(\|v^{\rm
h}_\Phi\|_{L^1_t(B^{2,\f12})}\|v^{\rm
h}_\Phi\|_{\widetilde{L}^\infty_t(B^{\ga,\f12-\ga})} +\|v^{\rm
h}_\Phi\|_{L^1_t(B^{2-\ga,\f12+\ga})}\|v^{\rm
h}_\Phi\|_{\widetilde{L}^\infty_t(B^{0,\f12})}\Bigr). \eeq

While using Bony's decomposition \eqref{pd} to $T^{\rm v}(v^3,
v^{\rm h})$ for the horizontal variables, one has \beno T^{\rm
v}(v^3, v^{\rm h})=\bigl(T^{\rm h}+R^{\rm h}+\bar{T}^{\rm
h}\bigr)T^{\rm v}(v^3, v^{\rm h}), \eeno from which, we deduce by a
similar proof of Lemma \ref{lem3.2} that \beno \|[T^{\rm v}(v^3,
v^{\rm h})]_\Phi\|_{L^1_t(B^{1+\ga,\f12-\ga})}\lesssim
\|v^3_\Phi\|_{\wt{L}^2_t(B^{1,\f12})}\|v^{\rm
h}_\Phi\|_{\wt{L}^2_t(B^{1+\ga,\f12-\ga})}+\|v^3_\Phi\|_{L^1_t(B^{2,\f12})}\|v^{\rm
h}_\Phi\|_{\wt{L}^\infty_t(B^{\ga,\f12-\ga})}, \eeno and \beno
\|[T^{\rm v}(v^3, v^{\rm
h})]_\Phi\|_{L^1_t(B^{1-\ga,\f12+\ga})}\lesssim
\|v^3_\Phi\|_{\wt{L}^2_t(B^{1,\f12})}\|v^{\rm
h}_\Phi\|_{\wt{L}^2_t(B^{1-\ga,\f12+\ga})}, \eeno so that there
holds \beq \label{l.8} \begin{split}
 \|q_{31}\|_{Z_t}\lesssim &\ve^{-\al}\Bigl( \|[T^{\rm v}(v^3,
v^{\rm h})]_\Phi\|_{L^1_t(B^{1+\ga,\f12-\ga})}+\|[T^{\rm v}(v^3,
v^{\rm h})]_\Phi\|_{L^1_t(B^{1-\ga,\f12+\ga})}\Bigr)\\
\lesssim &
\ve^{-\al}\Bigl(\|v^3_\Phi\|_{\wt{L}^2_t(B^{1,\f12})}\bigl(\|v^{\rm
h}_\Phi\|_{\wt{L}^2_t(B^{1+\ga,\f12-\ga})}+\|v^{\rm
h}_\Phi\|_{\wt{L}^2_t(B^{1-\ga,\f12+\ga})}\bigr)\\
&\qquad\qquad\qquad\qquad\qquad\qquad\qquad\qquad+\|v^3_\Phi\|_{L^1_t(B^{2,\f12})}\|v^{\rm
h}_\Phi\|_{\wt{L}^\infty_t(B^{\ga,\f12-\ga})}\Bigr). \end{split}\eeq

$\bullet$ \underline{Estimate of $\na_\ve q_4$}

Along the same line to the manipulation of $\na_\ve q_3$, we first
split $q_4$ as \beq\label{p.4}
q_4=\ve^{1-\al}(-\Delta_\ve)^{-1}\pa_3T^{\rm
v}({v^3},\textrm{div}_{\rm h}v^{\rm
h})+\ve^{1-\al}(-\Delta_\ve)^{-1}\pa_3(\cR^{\rm
v}(v^3,{\textrm{div}_{\rm h}v^{\rm h}}))\eqdefa q_{41}+q_{42}. \eeq
Similar to \eqref{o.7}, we have \beno \|q_{42}\|_{Z_t}\lesssim
\e^{1-\al-2\ga}\|[\cR^{\rm v}(v^3,\dive_{\rm h}v^{\rm
h})]_\Phi\|_{L^1_t(B^{-1+\ga,\f32-\ga})}, \eeno from which and a
similar proof of \eqref{l.7}, we infer \beq\label{l.9}
\|q_{42}\|_{Z_t}\lesssim \ve^{1-\al-2\ga}\|v^{\rm
h}_\Phi\|_{L^1_t(B^{2,\f12})}\|v^{\rm
h}_\Phi\|_{\widetilde{L}^\infty_t(B^{\ga,\f12-\ga})}. \eeq While a
similar proof of Lemma \ref{lem3.2} gives rise to  \beno \|[T^{\rm
v}({v^3},\textrm{div}_{\rm h}v^{\rm
h})]_\Phi\|_{L^1_t(B^{\ga,\f12-\ga})}\lesssim
\|v^3_\Phi\|_{\wt{L}^2_t(B^{1,\f12})}\|v^{\rm
h}_\Phi\|_{\wt{L}^2_t(B^{1+\ga,\f12-\ga})}, \eeno and \beno
\|[T^{\rm v}({v^3},\textrm{div}_{\rm h}v^{\rm
h})]_\Phi\|_{L^1_t(B^{-\ga,\f12+\ga})}\lesssim
\|v^3_\Phi\|_{\wt{L}^2_t(B^{1,\f12})}\|v^{\rm
h}_\Phi\|_{\wt{L}^2_t(B^{1-\ga,\f12+\ga})}. \eeno We thus obtain
\beq\label{l.10} \begin{split} \|q_{41}\|_{Z_t}\lesssim
&\ve^{-\al}\Bigl(\|[T^{\rm v}({v^3},\textrm{div}_{\rm h}v^{\rm
h})]_\Phi\|_{L^1_t(B^{\ga,\f12-\ga})}+\|[T^{\rm
v}({v^3},\textrm{div}_{\rm h}v^{\rm
h})]_\Phi\|_{L^1_t(B^{-\ga,\f12+\ga})}\Bigr)\\
\lesssim &
\ve^{-\al}\|v^3_\Phi\|_{\widetilde{L}^2_t(B^{1,\f12})}\bigl(\|v^{\rm
h}_\Phi\|_{\widetilde{L}^2_t(B^{1+\ga,\f12-\ga})} +\|v^{\rm
h}_\Phi\|_{\widetilde{L}^2_t(B^{1-\ga,\f12+\ga})}\bigr).
\end{split} \eeq

$\bullet$ \underline{Estimate of $\na_\ve q_5$}

We shall use the decomposition (\ref{eq:q5}) to deal with $q_5$.
Applying Corollary \ref{col3.1} gives \beno
\begin{split}
\|q_{51}\|_{Z_t}\lesssim &\|[G(\ve^\be a)\D_{\rm h}v^{\rm
h}]_\Phi\|_{L^1_t(B^{\ga,\f12-\ga})}+\|[G(\ve^\be a)\D_{\rm h}v^{\rm
h}]_\Phi\|_{L^1_t(B^{-\ga,\f12+\ga})}\\
\lesssim &\|[G(\ve^\be
a)]_\Phi\|_{\wt{L}^\infty_t(B^{1,\f12})}\bigl(\|\D_{\rm h}v^{\rm
h}_\Phi\|_{L^1_t(B^{\ga,\f12-\ga})}+\|\D_{\rm h}v^{\rm
h}_\Phi\|_{L^1_t(B^{-\ga,\f12+\ga})}\bigr)\\
&\qquad\qquad\qquad\qquad\qquad+ \|[G(\ve^\be
a)]_\Phi\|_{\wt{L}^\infty_t(B^{1-\ga,\f12+\ga})}\|\D_{\rm h}v^{\rm
h}_\Phi\|_{L^1_t(B^{0,\f12})},
\end{split} \eeno
from which, $\ve^\be K\leq\epsilon,$  and Lemma
\ref{lem:composition}, we conclude \beq\label{l.11}
\begin{split}
\|q_{51}\|_{Z_t}\lesssim &
\ve^\be\|a_\Phi\|_{\widetilde{L}^\infty_t(B^{1,\f12})}
\bigl(\|v_\Phi^{\rm h}\|_{L^1_t(B^{2+\ga,\f12-\ga})}+\|v_\Phi^{\rm h}\|_{L^1_t(B^{2-\ga,\f12+\ga})}\big)\\
&\qquad\qquad\qquad\qquad\qquad+\ve^\be\|a_\Phi\|_{\widetilde{L}^\infty_t(B^{1-\ga,\f12+\ga})}\|v_\Phi^{\rm
h}\|_{L^1_t(B^{2,\f12})}.
\end{split}
\eeq The same argument yields \beq\label{l.12}
\begin{split}
\|q_{52}\|_{Z_t}\lesssim
&\ve^{2+\be}\|a_\Phi\|_{\widetilde{L}^\infty_t(B^{1,\f12})}\bigl(\
\|v_\Phi^{\rm h}\|_{L^1_t(B^{\ga,\f52-\ga})}+\|v_\Phi^{\rm h}\|_{L^1_t(B^{-\ga,\f52+\ga})}\bigr)\\
&\qquad\qquad\qquad\qquad\qquad+
\ve^{2+\be}\|a_\Phi\|_{\widetilde{L}^\infty_t(B^{1-\ga,\f12+\ga})}\|v_\Phi^{\rm
h}\|_{L^1_t(B^{0,\f52})}.
\end{split}
\eeq Note that \beno
\begin{split}
\|q_{53}\|_{Z_t}\lesssim & \ve^{-2\ga}\|[G(\ve^\be a)\D_{\rm
h}v^3]_\Phi\|_{L^1_t(B^{-1+\ga,\f32-\ga})}\\
\lesssim &\ve^{-2\ga}\Bigl(\|[G(\ve^\be
a)]_\Phi\|_{\wt{L}^\infty_t(B^{1,\f12})}\|\D_{\rm
h}v^3_\Phi\|_{L^1_t(B^{-1+\ga,\f32-\ga})}\\
&\qquad\qquad\qquad+\|[G(\ve^\be
a)]_\Phi\|_{\wt{L}^\infty_t(B^{\ga,\f32-\ga})}\|\D_{\rm
h}v^3_\Phi\|_{L^1_t(B^{0,\f12})}\Bigr), \end{split} \eeno which
together with Lemma \ref{lem:composition} and $\dive v=0$ ensures
that \beq\label{l.13}
\begin{split}
\|q_{53}\|_{Z_t}\lesssim
\ve^{\be-2\ga}\Bigl(\|a_\Phi\|_{\widetilde{L}^\infty_t(B^{1,\f12})}\|v_\Phi^{\rm
h}\|_{L^1_t(B^{2+\ga,\f12-\ga})}
+\|a_\Phi\|_{\widetilde{L}^\infty_t(B^{\ga,\f32-\ga})}\|v_\Phi^3\|_{L^1_t(B^{2,\f12})}\Bigr).
\end{split}
\eeq

Similarly due to $\dive v=0,$ we have \beno
\|\na_\ve[q_{54}]_\Phi\|_{L^1_t(B^{\ga,\f12-\ga})}\lesssim
\ve^\be\|a_\Phi\|_{\widetilde{L}^\infty_t(B^{1,\f12})}\ve\|v_\Phi^{\rm
h}\|_{L^1_t(B^{1+\ga,\f32-\ga})},\eeno and \beno
\begin{split}
 \|\na_\ve[q_{54}]_\Phi\|_{L^1_t(B^{-\ga,\f12+\ga})}\lesssim
\ve^{1+\be}\Bigl(\|a_\Phi\|_{\widetilde{L}^\infty_t(B^{1,\f12})}\|v_\Phi^{\rm
h}\|_{L^1_t(B^{1-\ga,\f32+\ga})}+\|a_\Phi\|_{\widetilde{L}^\infty_t(B^{1-\ga,\f12+\ga})}\|v_\Phi^{\rm
h}\|_{L^1_t(B^{1,\f32})}\Bigr). \end{split} \eeno This gives rise to
 \beq\label{l.13}
\begin{split}
\|q_{54}\|_{Z_t}\lesssim &
\ve^{1+\be}\|a_\Phi\|_{\widetilde{L}^\infty_t(B^{1,\f12})}\bigl(\|v_\Phi^{\rm
h}\|_{L^1_t(B^{1+\ga,\f32-\ga})}
+\|v_\Phi^{\rm h}\|_{L^1_t(B^{1-\ga,\f32+\ga})}\bigr)\\
&\qquad\qquad\qquad\qquad\qquad+\ve^{1+\be}\|a_\Phi\|_{\widetilde{L}^\infty_t(B^{1-\ga,\f12+\ga})}\|v_\Phi^{\rm
h}\|_{L^1_t(B^{1,\f32})}.
\end{split} \eeq

By summing up the above estimates, we conclude that \beq\label{l.20}
\begin{split}
\|q\|_{Z_t}\lesssim &
\ve^{\be}\bigl(\|a_\Phi\|_{\wt{L}^\infty_t(B^{1,\f12})}+\|a_\Phi\|_{\wt{L}^\infty_t(B^{1-\ga,\f12+\ga})}\bigr)\|q\|_{Z_t}\\
&+\ve^{\be}\|a_\Phi\|_{\wt{L}^\infty_t(B^{1,\f12})}\Bigl(\|\D_\ve
v^{\rm h}_\Phi\|_{{L}^1_t(B^{\ga,\f12-\ga})}+\|\D_\ve v^{\rm
h}_\Phi\|_{{L}^1_t(B^{-\ga,\f12+\ga})}\\
&+\ve^{-2\ga}\|v^{\rm
h}_\Phi\|_{{L}^1_t(B^{2+\ga,\f12-\ga})}+\ve\|v^{\rm
h}_\Phi\|_{{L}^1_t(B^{1+\ga,\f32-\ga})}+\ve\|v^{\rm
h}_\Phi\|_{{L}^1_t(B^{1-\ga,\f32+\ga})}\Bigr)\\
&+\ve^{\be}\|a_\Phi\|_{\wt{L}^\infty_t(B^{1-\ga,\f12+\ga})}\Bigl(\|
v^{\rm h}_\Phi\|_{{L}^1_t(B^{2,\f12})}+\ve^2\| v^{\rm
h}_\Phi\|_{{L}^1_t(B^{0,\f52})}+\ve\|v^{\rm
h}_\Phi\|_{{L}^1_t(B^{1,\f32})}\Bigr)\\
 &+\ve^{1-\al}\Bigl(\|v^{\rm
h}_\Phi\|_{\wt{L}^\infty_t(B^{0,\f12})}\bigl(\|v^{\rm
h}_\Phi\|_{L^1_t(B^{2+\ga,\f12-\ga})}+\|v^{\rm
h}_\Phi\|_{L^1_t(B^{2-\ga,\f12+\ga})}\bigr)\\
&\qquad\qquad\qquad\qquad+\|v^{\rm
h}_\Phi\|_{L^1_t(B^{2,\f12})}\bigl(\|v^{\rm
h}_\Phi\|_{\widetilde{L}^\infty_t(B^{-\ga,\f12+\ga})}+\|v^{\rm
h}_\Phi\|_{\widetilde{L}^\infty_t(B^{\ga,\f12-\ga})}\bigr)\Bigr)\\
&+\ve^{1-\al-2\ga}\|v^{\rm h}_\Phi\|_{L^1_t(B^{2,\f12})}\|v^{\rm
h}_\Phi\|_{\wt{L}^\infty_t(B^{\ga,\f12-\ga})}+\ve^{\be-2\ga}\|a_\Phi\|_{\wt{L}^\infty_t(B^{\ga,\f32-\ga})}\|
v^{3}_\Phi\|_{{L}^1_t(B^{2,\f12})}\\
&+\ve^{-\al}\Bigl(\|v^{3}_\Phi\|_{L^1_t(B^{2,\f12})}\|v^{\rm
h}_\Phi\|_{\wt{L}^\infty_t(B^{\ga,\f12-\ga})}\\
&\qquad\qquad\qquad\qquad+\|v^{3}_\Phi\|_{\wt{L}^2_t(B^{1,\f12})}\bigl(\|v^{\rm
h}_\Phi\|_{\wt{L}^2_t(B^{1+\ga,\f12-\ga})}+\|v^{\rm
h}_\Phi\|_{\wt{L}^2_t(B^{1-\ga,\f12+\ga})}\bigr)\Bigr),
\end{split} \eeq
While it follows from Definition \ref{def2.1} that \beq\label{l.21}
\begin{split}
\ve\|v^{\rm
h}_\Phi\|_{{L}^1_t(B^{1+\ga,\f32-\ga})}=&\ve\sum_{k,\ell\in\Z}2^{k(1+\ga)}2^{\ell\bigl(\f32-\ga\bigr)}\|\D_k^{\rm
h}\D_{\ell}^{\rm v}v^{\rm h}_\Phi\|_{L^1_t(L^2)}\\
\leq&\f12\sum_{k,\ell\in\Z}\bigl(2^{k(2+\ga)}2^{\ell\bigl(\f12-\ga\bigr)}+\ve^22^{k\ga}2^{\ell\bigl(\f52-\ga\bigr)}\bigr)\|\D_k^{\rm
h}\D_{\ell}^{\rm v}v^{\rm h}_\Phi\|_{L^1_t(L^2)}\\
=&\f12\bigl(\|v^{\rm
h}_\Phi\|_{L^1_t(B^{2+\ga,\f12-\ga})}+\ve^2\|v^{\rm
h}_\Phi\|_{L^1_t(B^{\ga,\f52-\ga})}\bigr). \end{split} \eeq The same
argument gives \beq \label{l.22}
\begin{split} &\ve\|v_\Phi\|_{L^1_t(B^{1,\f32})}\leq \f12\bigl(\|v^{\rm
h}_\Phi\|_{L^1_t(B^{2,\f12})}+\ve^2\|v^{\rm
h}_\Phi\|_{L^1_t(B^{0,\f52})}\bigr),\\
&\ve\|v_\Phi\|_{L^1_t(B^{1-\ga,\f32+\ga})}\leq \f12\bigl(\|v^{\rm
h}_\Phi\|_{L^1_t(B^{2-\ga,\f12+\ga})}+\ve^2\|v^{\rm
h}_\Phi\|_{L^1_t(B^{-\ga,\f52+\ga})}\bigr). \end{split} \eeq
Therefore, in view of \eqref{p.1},  we deduce from \eqref{l.1before}
and \eqref{l.20} that \beno \|q\|_{Z_t}\le C\Bigl(K
\ve^\be\|q\|_{Z_t}+\ve^{\be-3\al-5\ga}\Psi_1(t)\Psi_3(t)+\bigl(\ve^{-\al}+\ve^{1-3\al-4\ga}\bigr)\Psi_2(t)\Psi_3(t)+\ve^{-2\al-\ga}\Psi_4^2(t)\Bigr),
\eeno which together with the assumptions on $\al,\be$ and $\ga$
leads to \eqref{l.1}, and we completes the proof of the proposition.
\end{proof}

\begin{Remark}\label{rem:pressure} It is easy to observe from the proof of Proposition \ref{prop6.2} that
if $\be>2\al,$ $\ga\le \min\bigl(\f {1-3\al} 4, \f {\be-2\al}
2\bigr)$ and $\ve$ is so small that $\ve^\be\leq
\min\Bigl(\f1{2C_0K},\f{\epsilon}K\Bigr),$ then there holds
$$\longformule{
\|q_1\|_{Z_t}+\|q_2\|_{Z_t}+\|q_{32}\|_{Z_t}+\|q_{42}\|_{Z_t}+\|q_{51}\|_{Z_t}+\|q_{52}\|_{Z_t}+\|q_{54}\|_{Z_t}}{{}\le
C\max\Bigl(\ve^{\be-2\al-2\ga}, \ve^{1-3\al-4\ga},
K\ve^{\be-2\al-\ga}\Bigr)\Psi^2(t).}
$$
\end{Remark}

\setcounter{equation}{0}

\section{Classical parabolic type estimates}\label{sect7}

This section is devoted to the estimate of the analytic band
$\theta$, i.e, the proof of Proposition \ref{prop:tht}. To achieve
this, we first  rewrite the momentum equation of
\eqref{eq:InhomoNS-scaled} as follows \beq\label{eq:velocitya}
\begin{split} &\p_t v-\Delta_\ve v=F_1+F_2+F_3\with\\
 F_1\eqdefa -\ve^{1-\al}& v\cdot\na v,\quad
F_2\eqdefa-\f {\ve^\be a} {1+\ve^\be a}\Delta_\ve v,\quad
F_3\eqdefa-\f {1} {1+\ve^\be a}\na^\ve q. \end{split} \eeq  For
$E_\ve$ given by \eqref{g.2}, applying the Duhamel formula to
\eqref{eq:velocitya} gives
\begin{align}\label{eq:velocity}
v(t)=e^{t\Delta_\ve}v_0+E_\ve(F_1+F_2+F_3).
\end{align}
In what follows, we denote \beq\label{q.1} \|f\|_{H_t}\eqdefa
\big\|f_\Phi\big\|_{L^1_t(B^{1,\f12})}+\big\|f_\Phi\big\|_{L^1_t(B^{1+\ga,\f12-\ga})}
+\ve^{1+\ga}\big\|f_\Phi\big\|_{L^1_t(B^{-\ga,\f32+\ga})}.  \eeq
First of all, it follows from Lemma \ref{lem:parabolic-home} that
\beno \ve^{1-\al}\bigl\|[e^{t\D_\ve}v_0^{\rm
h}]_\Phi\bigr\|_{L^1_t(B^{1,\f12})}\lesssim \ve^\ga \bigl\|e^{\delta
|D|}v_0^{\rm h}\bigr\|_{B^{-\al-\ga,-\f12+\al+\ga}}. \eeno However
since $0<\ga<\f{1-2\al}4,$ we have $-\f12+\ga<-\al-\ga<0$ and
$-\f12<-\f12+\al+\ga<-\ga,$ so that applying Lemma \ref{lem3.1}
yields \beno \ve^{1-\al}\bigl\|[e^{t\D_\ve}v_0^{\rm
h}]_\Phi\bigr\|_{L^1_t(B^{1,\f12})}\lesssim
\ve^\ga\Bigl(\bigl\|e^{\delta |D|}v_0^{\rm
h}\bigr\|_{B^{-\f12+\ga,-\ga}}+\bigl\|e^{\delta |D|}v_0^{\rm
h}\bigr\|_{B^{0,-\f12}}\Bigr)\le \ve^\ga\|v^{\rm h}_0\|_{X_2} \eeno
for the norm $\|\cdot\|_{X_2}$ given by \eqref{i.1}.

Along the same line, one has $$\longformule{
 \ve^{1-\al}\bigl\|[e^{t\D_\ve}v_0^{\rm
h}]_\Phi\bigr\|_{L^1_t(B^{1+\ga,\f12-\ga})}+\ve^{2-\al+\ga}\bigl\|[e^{t\D_\ve}v_0^{\rm
h}]_\Phi\bigr\|_{L^1_t(B^{-\ga,\f32+\ga})}}{{}\lesssim
\ve^\ga\Bigl(\bigl\|e^{\delta |D|}v_0^{\rm
h}\bigr\|_{B^{-\al,-\f12+\al}}+\bigl\|e^{\delta |D|}v_0^{\rm
h}\bigr\|_{B^{-\al-\ga,-\f12+\al+\ga}}\Bigr)\lesssim
\ve^\ga\|v_0^{\rm h}\|_{X_2}.}
$$

While it follows form  the second inequality of \eqref{g.1} that
\beno \big\|[e^{t\Delta_\ve}v_0^3]_\Phi\big\|_{L^1_t(B^{1,\f12})}
+\ve^{1+\ga}\big\|[e^{t\Delta_\ve}v_0^3]_\Phi\big\|_{L^1_t(B^{-\ga,\f32+\ga})}
\lesssim \bigl\|e^{\delta |D|}v_0^{\rm h}\bigr\|_{B^{0,-\f12}}.
\eeno While it follows from the proof of Lemma
\ref{lem:parabolic-home} and $\dive v_0=0$ that \beno
\begin{split}
\bigl\|\D_k^{\rm h}\D_{\ell}^{\rm
v}[e^{t\Delta_\ve}v_0^3]_\Phi\big\|_{L^1_t(L^2)}\lesssim
&2^{-2k}\bigl\|e^{\delta |D|}\D_k^{\rm h}\D_\ell^{\rm
v}v_0^3\bigr\|_{L^2}^{\f{2\ga}{1+\ga}}\Bigl(2^{-\ell}\bigl\|e^{\delta
|D|}\D_k^{\rm h}\D_\ell^{\rm v}\dive_{\rm
h}v_0^h\bigr\|_{L^2}\Bigr)^{\f{1-\ga}{1+\ga}}\\
\lesssim &
d_{k,\ell}2^{-k(1+\ga)}2^{-\ell\bigl(\f12-\ga\bigr)}\|e^{\delta
|D|}v_0^3\|_{B^{0,-\f12}}^{\f{2\ga}{1+\ga}}\|e^{\delta |D|}v_0^{\rm
h}\|_{B^{-\ga,-\f12+\ga}}^{\f{1-\ga}{1+\ga}},
\end{split}
\eeno which gives \beno
\big\|[e^{t\Delta_\ve}v_0^3]_\Phi\big\|_{L^1_t(B^{1+\ga,\f12-\ga})}\lesssim
\|v_0\|_{X_2}. \eeno As a consequence, we obtain
\begin{align}\label{eq:v-chom}
\ve^{1-\al}\|e^{t\Delta_\ve}v_0^{\rm
h}\|_{H_t}+\ve^\ga\|e^{t\Delta_\ve}v_0^3\|_{H_t}\le C\ve^\ga\|v_0\|_{X_2}.
\end{align}

{\bf Step 1.} Estimate of the horizontal velocity\vspace{0.1cm}

$\bullet$ \underline{Estimate of $E_\ve(F_1^{\rm h})$}

Since $\dive v=0,$ $v\cdot\na v^{\rm h}=\na_{\rm h}\cdot(v^{\rm
h}\otimes v^{\rm h})+\pa_3(v^3v^{\rm h})$, we get, by applying Lemma
\ref{lem:parabolic-inhome} and the law of product of Corollary
\ref{col3.1}, that
\begin{align*}
\ve^{1-\al}\big\|[E_\ve(F_1^{\rm h})]_\Phi\big\|_{L^1_t(B^{1,\f12})}
\lesssim & \ve^{2(1-\al)}\bigl\|\dive_{\rm h}[v^{\rm h}\otimes
v^{\rm h}]_\Phi\|_{L^1_t(B^{-1,\f12})}+\ve^{1-2\al}\bigl\|\p_3[
v^{3}v^{\rm h}]_\Phi\bigr\|_{L^1_t(B^{0,-\f12})}\\
\lesssim & \ve^{2(1-\al)}\bigl\|[v^hv^{\rm h}]_\Phi\bigr\|_{L^1_t(B^{0,\f12})}+\ve^{1-2\al}\bigl\|[v^3v^{\rm h}]_\Phi\bigr\|_{L^1_t(B^{0,\f12})}\\
\lesssim & \ve^{1-2\al}\bigl(\ve\|v^{\rm
h}_\Phi\|_{L^1_t(B^{1,\f12})}
+\|v^3_\Phi\|_{L^1_t(B^{1,\f12})}\bigr)\|v^{\rm
h}_\Phi\|_{\widetilde{L}^\infty_t(B^{0,\f12})},
\end{align*}
Along the same line, we have
\begin{align*}
\ve^{1-\al}\big\|[E_\ve(F_1^{\rm h})]_\Phi\big\|_{L^1_t(B^{1+\ga,\f12-\ga})}&\lesssim
\ve^{2(1-\al)}\bigl\|[v^{\rm h}\otimes v^{\rm h}]_\Phi\bigr\|_{L^1_t(B^{\ga,\f12-\ga})}+\ve^{1-2\al}\bigl\|[v^3v^{\rm h}]_\Phi\bigr\|_{L^1_t(B^{\ga,\f12-\ga})}\\
&\lesssim \ve^{1-2\al}\bigl(\ve\|v^{\rm
h}_\Phi\|_{L^1_t(B^{1,\f12})}
+\|v^3_\Phi\|_{L^1_t(B^{1,\f12})}\bigr)\|v^{\rm
h}_\Phi\|_{\widetilde{L}^\infty_t(B^{\ga,\f12-\ga})}.
\end{align*}
and if $\al\leq \f12,$
\begin{align*}
\ve^{2-\al+\ga}\big\|[E_\ve(F_1^{\rm h})]_\Phi\big\|_{L^1_t(B^{-\ga,\f32+\ga})}\lesssim &
\ve^{2(1-\al)}\|[v^{\rm h}\otimes v^{\rm h}]_\Phi\|_{L^1_t(B^{0,\f12})}+\ve^{1-2\al+\ga}\|\pa_3[v^3v^{\rm h}]_\Phi\|_{L^1_t(B^{-\ga,-\f12+\ga})}\\
\lesssim & \ve^{1-2\al}\Bigl(\ve\|v^{\rm
h}_\Phi\|_{L^1_t(B^{1,\f12})}\|v^{\rm
h}_\Phi\|_{\widetilde{L}^\infty_t(B^{0,\f12})}
+\|v^3_\Phi\|_{L^1_t(B^{1,\f12})}\|v^{\rm h}_\Phi\|_{\widetilde{L}^\infty_t(B^{-\ga,\f12+\ga})}\\
&\qquad\qquad\qquad\qquad\qquad\qquad+\e^\ga\|v^3_\Phi\|_{L^1_t(B^{1-\ga,\f12+\ga})}\|v^{\rm
h}_\Phi\|_{\widetilde{L}^\infty_t(B^{0,\f12})}\Bigr).
\end{align*}
We thus obtain \beno
\begin{split}
 \ve^{1-\al}\|E_\ve(F_1^{\rm h})\|_{H_t}\lesssim & \ve^{1-2\al}\Bigl(\bigl(\ve\|v^{\rm h}_\Phi\|_{L^1_t(B^{1,\f12})}
+\|v^3_\Phi\|_{L^1_t(B^{1,\f12})}\bigr)\bigl(\|v^{\rm
h}_\Phi\|_{\widetilde{L}^\infty_t(B^{0,\f12})}+ \|v^{\rm
h}_\Phi\|_{\widetilde{L}^\infty_t(B^{\ga,\f12-\ga})}\bigr)\\
&\qquad+\|v^3_\Phi\|_{L^1_t(B^{1,\f12})}\|v^{\rm
h}_\Phi\|_{\widetilde{L}^\infty_t(B^{-\ga,\f12+\ga})}+\e^\ga\|v^3_\Phi\|_{L^1_t(B^{1-\ga,\f12+\ga})}\|v^{\rm
h}_\Phi\|_{\widetilde{L}^\infty_t(B^{0,\f12})}\Bigr).
\end{split}\eeno
However, note that \beno
\begin{split}
\e^\ga\|v^3_\Phi\|_{L^1_t(B^{1-\ga,\f12+\ga})}=&\sum_{k,\ell\in\Z}2^{k(1-\ga)}2^{\ell\bigl(\f12+\ga\bigr)}\Bigl(\e^{1+\ga}\|\D_k^{\rm
h}\D_\ell^{\rm
v}v_\Phi^3\|_{L^1_t(L^2)}\Bigr)^{\f{\ga}{1+\ga}}\|\D_k^{\rm
h}\D_\ell^{\rm v}v_\Phi^3\|_{L^1_t(L^2)}^{\f{1}{1+\ga}}\\
\lesssim
&\Bigl(\e^{1+\ga}\|v^3_\Phi\|_{L^1_t(B^{-\ga,\f32+\ga})}\Bigr)^{\f{\ga}{1+\ga}}\|v^3_\Phi\|_{L^1_t(B^{1,\f12})}^{\f{1}{1+\ga}},
\end{split}
\eeno we infer
 \beq\label{eq:vh-cF1}
 \ve^{1-\al}\|E_\ve(F_1^{\rm h})\|_{H_t}\lesssim \ve^{1-2\al-\ga}\theta(t)\Psi_2(t).
 \eeq

$\bullet$ \underline{Estimate of $E_\ve(F_2^{\rm h})$}

Similar to the estimate of $E_\ve(F_1^{\rm h}),$ since $\ve^\be
K\leq \epsilon,$ we get, by applying Lemma
\ref{lem:parabolic-inhome}, the law of product of Corollary
\ref{col3.1}  and Lemma \ref{lem:composition}, that \beno
\begin{split}
\|E_\ve(G(\ve^\be a)\D_{\rm h}v^{\rm h})\|_{H_t}\lesssim
&\ve^{-\ga}\|[G(\ve^\be a)\D_{\rm h}v^{\rm
h}]_{\Phi}\|_{L^1_t(B^{-1+\ga,\f12-\ga})}\\
\lesssim
&\ve^{\be-\ga}\|a_\Phi\|_{\widetilde{L}^\infty_t(B^{1,\f12})}\|\D_{\rm
h}v^{\rm h}_{\Phi}\|_{L^1_t(B^{-1+\ga,\f12-\ga})},
\end{split}
\eeno and \beno
\begin{split}
\ve^2\|E_\ve(G(\ve^\be a)\p_3^2v^{\rm h})\|_{H_t}\lesssim
&\ve^{1+\ga}\|[G(\ve^\be a)\pa_3^2v^{\rm
h}]_\Phi\|_{L^1_t(B^{-\ga,-\f12+\ga})}\\
\lesssim
&\ve^{1+\be-\ga}\|a_\Phi\|_{\widetilde{L}^\infty_t(B^{1,\f12})}\|\pa_3^2v^{\rm
h}_\Phi\|_{L^1_t(B^{-\ga,-\f12+\ga})}.
\end{split}
\eeno Therefore, we obtain \beq\label{eq:vh-cF2}\begin{split}
\ve^{1-\al}\|E_\ve(F_2^{\rm h})\|_{H_t} \lesssim &
\ve^{1-\al+\be-\ga}
\|a_\Phi\|_{\widetilde{L}^\infty_t(B^{1,\f12})}\bigl(\|v^{\rm
h}_\Phi\|_{L^1_t(B^{1+\ga,\f12-\ga})} +\ve\|v^{\rm
h}_\Phi\|_{{L}^1_t(B^{-\ga,\f32+\ga})}\bigr)\\
\lesssim & \ve^{\be-\ga}\theta(t)\Psi_1(t).
\end{split} \eeq

$\bullet$ \underline{Estimate of $E_\ve(F_3^{\rm h})$}

Due to \eqref{epsilon}, it follows from  Lemma
\ref{lem:parabolic-inhome} and Lemma \ref{lem:composition} that
\beno
\begin{split}
\|E_\ve(F_3^{\rm h})\|_{H_t}\lesssim
&\bigl\|q\bigr\|_{Y_t}+\e^{-\ga}\bigl\|[G(\ve^\be a)\na_{\rm
h}q]_\Phi\bigr\|_{L^1_t(B^{-1+\ga,\f12-\ga})}\\
\lesssim
& \bigl\|q\bigr\|_{Y_t}+\ve^{\be-\ga}\|a_\Phi\|_{\widetilde{L}^\infty_t(B^{1,\f12})}\|\na_{\rm
h}q_\Phi\|_{L^1_t(B^{-1+\ga,\f12-\ga})},
\end{split}
\eeno from which, the assumption that $\ve^{\be-\ga}K\leq 1$ and
Proposition \ref{prop6.1}, we infer
\begin{align}\label{eq:vh-cF3}
e^{1-\al}\|E_\ve(F_3^{\rm h})\|_{H_t}\le C\ve^{1-\al}\|q\|_{Y_t}\le
C\min\big(\ve^{\be-\al-2\ga},\ve^{1-2\al-2\ga}\bigr)\theta(t)\Psi(t).
\end{align}

By summing up (\ref{eq:v-chom})--(\ref{eq:vh-cF3}), we conclude that
\ben\label{eq:vh-class} \ve^{1-\al}\|v^{\rm h}\|_{H_t}\le
C\bigl(\ve^\ga\|v_0^{\rm
h}\|_{X_2}+\max\big(\ve^{\be-\al-2\ga},\ve^{1-2\al-2\ga}\bigr)\theta(t)\Psi(t)\bigr).
\een

{\bf Step 2.} Estimate of the vertical velocity\vspace{0.1cm}

$\bullet$ \underline{Estimate of $E_\ve(F_1^3)$}

Again since $\dive v=0,$ we write
 $$
 v\cdot\na v^3=\na_{\rm
h}\cdot(v^{\rm h}v^3)-2(v^3\textrm{div}_{\rm h}v^{\rm h}),$$ from
which, Lemma \ref{lem:parabolic-inhome} and the law of product of
Corollary \ref{col3.1}, we deduce that
\begin{align*}
\big\|[E_\ve(F_1^3)]_\Phi\big\|_{L^1_t(B^{1,\f12})}
&\lesssim \ve^{1-\al}\Bigl(\|[v^{\rm h}v^3]_\Phi\|_{L^1_t(B^{0,\f12})}+\ve^{-\ga}\|[v^3\textrm{div}_{\rm h}v^{\rm h}]_\Phi\|_{L^1_t(B^{-1+\ga,\f12-\ga})}\Bigr)\\
&\lesssim \ve^{1-\al}\|v^3_\Phi\|_{L^1_t(B^{1,\f12})}\bigl(\|v^{\rm
h}_\Phi\|_{\widetilde{L}^\infty_t(B^{0,\f12})} +\ve^{-\ga}\|v^{\rm
h}_\Phi\|_{\widetilde{L}^\infty_t(B^{\ga,\f12-\ga})}\bigr),
\end{align*}
and
\begin{align*}
\big\|[E_\ve(F_1^3)]_\Phi\big\|_{L^1_t(B^{1+\ga,\f12-\ga})}&\lesssim
\ve^{1-\al}\Bigl(
\|[v^{\rm h}v^3]_\Phi\|_{L^1_t(B^{\ga,\f12-\ga})}+\|[v^3\textrm{div}_{\rm h}v^{\rm h}]_\Phi\|_{L^1_t(B^{-1+\ga,\f12-\ga})}\Bigr)\\
&\lesssim \ve^{1-\al}\|v^3_\Phi\|_{L^1_t(B^{1,\f12})}\|v^{\rm
h}_\Phi\|_{\widetilde{L}^\infty_t(B^{\ga,\f12-\ga})},
\end{align*}
and
\begin{align*}
\ve^{1+\ga}\big\|[E_\ve(F_1^3)]_\Phi\big\|_{L^1_t(B^{-\ga,\f32+\ga})}\lesssim &
\ve^{1-\al}\bigl(\|[v^{\rm h}v^3]_\Phi\|_{L^1_t(B^{0,\f12})}+\ve^{-\ga}\|[v^3\textrm{div}_{\rm h}v^{\rm h}]_\Phi\|_{L^1_t(B^{-1+\ga,\f12-\ga})}\bigr)\\
\lesssim & \ve^{1-\al}\|v^3_\Phi\|_{L^1_t(B^{1,\f12})}\bigl(\|v^{\rm
h}_\Phi\|_{\widetilde{L}^\infty_t(B^{0,\f12})} +\ve^{-\ga}\|v^{\rm
h}_\Phi\|_{\widetilde{L}^\infty_t(B^{\ga,\f12-\ga})}\bigr).
\end{align*}
Therefore, if $\ga\leq 1-\al,$ we obtain
 \beq\label{eq:v3-cF1} \begin{split}
 \|E_\ve(F_1^3)\|_{H_t}\leq &
C\ve^{1-\al}\|v^3_\Phi\|_{L^1_t(B^{1,\f12})}\bigl(\|v^{\rm
h}_\Phi\|_{\widetilde{L}^\infty_t(B^{0,\f12})} +\ve^{-\ga}\|v^{\rm
h}_\Phi\|_{\widetilde{L}^\infty_t(B^{\ga,\f12-\ga})}\bigr)\\
\leq & C\ve^{1-\al-2\ga}\theta(t)\Psi_2(t). \end{split}  \eeq

$\bullet$ \underline{Estimate of $E_\ve(F_2^3)$}

Similar to the estimate of \eqref{eq:vh-cF2}, we have \beno
\begin{split}
\big\|E_\ve(F_2^3)\big\|_{H_t} &\lesssim \ve^{-\ga}\|[G(\ve^\be
a)\Delta_hv^3]_\Phi\|_{L^1_t(B^{-1+\ga,\f12-\ga})}+
\ve^{1+\ga}\|[G(\ve^\be a)\pa_3^2v^3]_\Phi\|_{L^1_t(B^{-\ga,-\f12+\ga})}\\
&\lesssim
\|a_\Phi\|_{\widetilde{L}^\infty_t(B^{1,\f12})}\bigl(\ve^{\be-\ga}
\|v^3_\Phi\|_{L^1_t(B^{1+\ga,\f12-\ga})}
+\ve^{1+\be+\ga}\|v^3_\Phi\|_{{L}^1_t(B^{-\ga,\f32+\ga})}\bigr),
\end{split} \eeno
so that  we get \beq \label{eq:v3-cF2}
\big\|E_\ve(F_2^3)\big\|_{H_t} \leq
C\ve^{\be-2\ga}\Psi_1(t)\theta(t). \eeq

$\bullet$ \underline{Estimate of $E_\ve(F_3^3)$}

It follows by a similar derivation of \eqref{eq:vh-cF3} that for
$\ga\leq \al,$
\begin{align}
\|E_\ve(F_3^3)\|_{H_t}\le & C\ve^{1-\ga}\|[(1-G(\ve^\be a))\ve\pa_3
q]_\Phi\|_{L^1_t(B^{-1+\ga,\f12-\ga})}\\
 \le &
C\ve^{\al-\ga}\ve^{1-\al}\|q\|_{Y_t}\le
C\ve^{\al-\ga}\theta(t)\Psi(t).\label{eq:v3-cF3}
\end{align}

Since $\ga<\f{\be-\al}2,$ we have $\be-2\ga>\al-\ga,$ by summing up
(\ref{eq:v-chom}) and (\ref{eq:v3-cF1})--(\ref{eq:v3-cF3}), we
arrive at \beq \label{eq:v3-class} \|v^3\|_{H_t}\le
C\bigl(\|v_0\|_{X_2}+\max\bigl(\ve^{1-\al-2\ga},\ve^{\al-\ga}\bigr)\theta(t)\Psi(t)\bigr).
\eeq

Proposition \ref{prop:tht} follows by combining (\ref{eq:vh-class})
with (\ref{eq:v3-class}).\ef

\setcounter{equation}{0}

\section{Regularizing effect of the analyticity}\label{sect8}

The goal of this  section is to present the proof of Proposition
\ref{prop:Psi}. Here we need to use the regularizing effect of the
heat semigroup. As a convention throughout this section, we always
assume that there holds (\ref{assum-a}). \vspace{0.1cm}

\no{\bf Step 1.} Estimate of the density

In view of \eqref{p.1}, we get, by applying \eqref{o.9} and
(\ref{o.19}-\ref{o.20}), that \beno
\begin{split}
\Psi_1(t)\leq
&\bigl\|e^{\delta|D|}a_0\bigr\|_{B^{1,\f12}}+\bigl\|e^{\delta|D|}a_0\bigr\|_{B^{1+\ga,\f12-\ga}}
+\bigl\|e^{\delta|D|}a_0\bigr\|_{B^{1-\ga,\f12+\ga}}+\e^{3\al+3\ga}\bigl\|e^{\delta|D|}a_0\bigr\|_{B^{\ga,\f32-\ga}}\\
&+C\Bigl(\f{1}\lam\Psi_1(t)+\ve^{1-\al}\bigl(\|v_\Phi\|_{L^1_t(B^{2-\ga,\f12+\ga})}+
\|v_\Phi\|_{L^1_t(B^{2,\f12})}+\|v_\Phi\|_{L^1_t(B^{2+\ga,\f12-\ga})}\bigr)\|a_\Phi\|_{\wt{L}^\infty_t(B^{1,\f12})}\\
&+\ve^{1+2\al+3\ga}\bigl(\|v^{\rm
h}_\Phi\|_{L^1_t(B^{1,\f32})}\|a_\Phi\|_{\wt{L}^\infty_t(B^{1+\ga,\f12-\ga})}
+\|v_\Phi\|_{L^1_t(B^{1+\ga,\f32-\ga})}\|a_\Phi\|_{\wt{L}^\infty_t(B^{1,\f12})}\bigr)\Bigr).
\end{split}
\eeno However it is easy to observe from Lemma \ref{lem3.1} that
\beno \bigl\|e^{\delta|D|}a_0\bigr\|_{B^{1,\f12}}\lesssim
\bigl\|e^{\delta|D|}a_0\bigr\|_{B^{1+\ga,\f12-\ga}}+\bigl\|e^{\delta|D|}a_0\bigr\|_{B^{1-\ga,\f12+\ga}},
\eeno and it follows from \eqref{l.21} and \eqref{l.22} that
$$\longformule{ \ve^{1+2\al+2\ga}\bigl(\|v^{\rm
h}_\Phi\|_{L^1_t(B^{1,\f32})} +\|v^{\rm
h}_\Phi\|_{L^1_t(B^{1+\ga,\f32-\ga})}\bigr)}{{}\lesssim
\ve^{2\al+2\ga}\Bigl(\|v^{\rm
h}_\Phi\|_{L^1_t(B^{2,\f12})}+\ve^2\|v^{\rm
h}_\Phi\|_{L^1_t(B^{0,\f52})}
+\|v_\Phi\|_{L^1_t(B^{2+\ga,\f12-\ga})}+\ve^2\|v_\Phi^{\rm
h}\|_{L^1_t(B^{\ga,\f52-\ga})}\Bigr). }$$ Therefore since $0<\ga\leq
\f{1-3\al} 3$, we  obtain
\begin{align}\label{eq:Psi-1}
\Psi_1(t)\le C\|a_0\|_{X_1}+C\Big(\f1
\lambda+\ve^\ga\Psi_3(t)\Big)\Psi_1(t).
\end{align}

\no{\bf Step 2.} Estimate of $\Psi_2(t)$

In the remaining of this section,  we denote \beno
\|f\|_{K_t}\eqdefa
\|f_\Phi\|_{\widetilde{L}^\infty_t(B^{\ga,\f12-\ga})}+\|f_\Phi\|_{\widetilde{L}^\infty_t(B^{-\ga,\f12+\ga})}.
\eeno Then it follows from Lemma \ref{lem:parabolic-home} that
\begin{align}\label{eq:vhom-a}
\|e^{t\Delta_\ve}v_0\|_{K_t}\le C\|v_0\|_{X_3}.
\end{align}

\no{\bf Step 2.1}\ The estimate of the horizontal velocity.

In order to estimate $\|v^{\rm h}\|_{K_t},$ we still need to deal
with the source term in \eqref{eq:velocity}.

$\bullet$ \underline{Estimate of $E_\ve(F_1^{\rm h})$}

In view of  \eqref{eq:velocity}, by using Bony's decomposition
\eqref{pd} in the horizontal variable for $v^3v^{\rm h},$ we write
$F_1^{\rm h}$ as \beno F_1^{\rm h}=-\ve^{1-\al}\na_{\rm
h}\cdot(v^{\rm h}\otimes v^{\rm h})-\ve^{1-\al}\pa_3\cR^{\rm v}(v^3,
v^{\rm h})-\ve^{1-\al}\pa_3T^{\rm v}({v^3},v^{\rm h})\eqdefa
F_{11}^{\rm h}+F_{12}^{\rm h}+F_{13}^{\rm h}. \eeno Applying Lemma
\ref{lem:parabolic-inhome} and the law of product of Corollary
\ref{col3.1} yields
\begin{align*}
\|E_\ve(F_{11}^{\rm h})\|_{K_t}&\lesssim \ve^{1-\al}
\big(\|[v^{\rm h}\na_hv^{\rm h}]_\Phi\|_{L^1_t(B^{\ga,\f12-\ga})}+\|[v^{\rm h}\na_hv^{\rm h}]_\Phi\|_{L^1_t(B^{-\ga,\f12+\ga})}\big)\\
&\lesssim \ve^{1-\al}\Bigl(\|v^{\rm
h}_\Phi\|_{L^1_t(B^{2,\f12})}\big(\|v^{\rm
h}_\Phi\|_{\widetilde{L}^\infty_t(B^{\ga,\f12-\ga})}
+\|v^{\rm h}_\Phi\|_{\widetilde{L}^\infty_t(B^{-\ga,\f12+\ga})}\big)\\
&\qquad\qquad\qquad\qquad\qquad\qquad+\|v^{\rm
h}_\Phi\|_{L^1_t(B^{2-\ga,\f12+\ga})}\|v^{\rm
h}_\Phi\|_{\widetilde{L}^\infty_t(B^{0,\f12})}\Bigr).
\end{align*}
Note that for $\bar{\varphi}$ in $C_c^\infty(\R^+\setminus \{0\})$
with $\bar{\varphi}$ equals $1$ on the support of $\varphi$ in
\eqref{1.0}, let $\wt{\varphi}(\xi_3)\eqdefa
\frac{\bar{\varphi}(|\xi_3|)}{i\xi_3},$ we may write \beno
\D_\ell^{\rm v}v^3=2^{-\ell}\wt{\varphi}(2^{-\ell}|D_3|)\D_\ell^{\rm
v}\p_3v^3,\eeno and due to $\p_3v^3=-\dive_{\rm h}v^{\rm h},$ we
have \beno \cR^{\rm v}(v^3,v^{\rm h})=-\sum_{\ell\in
\Z}2^{-\ell}\wt{\varphi}(2^{-\ell}|D_3|)\D_\ell^{\rm v}\dive_{\rm
h}v^{\rm h}S_{\ell+2}^{\rm v}v^{\rm h},\eeno from which,
 by
using Bony's decomposition in the horizontal variables for $\cR^{\rm
v}(v^3, v^{\rm h}),$ one may deduce, by a similar derivation of
Lemma \ref{lem3.2}, that   $E_\ve(F_{12}^{\rm h})$ shares the same
estimate as $E_\ve(F_{11}^{\rm h})$.

Whereas it follows form  Remark \ref{rem:product} that  \beno
\|\Delta_k^{\rm h}\Delta_\ell^{\rm v}[F_{13}^{\rm
h}]_\Phi(t)\|_{L^2}\lesssim
\bigl(d_{k}(t)d_\ell+d_{k,\ell}\bigr)2^{-k\s}2^{-\ell
s}\|v^3_\Phi(t)\|_{B^{1,\f12}}\|v^{\rm
h}_\Phi\|_{\widetilde{L}^\infty_t({B}^{\s,s})} \eeno for any $\s\in
]-1,1], s\in \R$,  from which, and  Lemma \ref{lem:analytic}, we
infer
\begin{align*}
&\|E_\ve(F_{13}^{\rm h})\|_{K_t}\le \f C \lambda\big(\|v^{\rm
h}_\Phi\|_{\widetilde{L}^\infty_t(B^{\ga,\f12-\ga})} +\|v^{\rm
h}_\Phi\|_{\widetilde{L}^\infty_t(B^{-\ga,\f12+\ga})}\big).
\end{align*}
Hence we obtain \ben\label{eq:vh-F1a} \|E_\ve(F_{1}^{\rm
h})\|_{K_t}\le C\Bigl(\f1
\lambda+\ve^{1-3\al-2\ga}\Psi_3(t)\Bigr)\Psi_2(t). \een

$\bullet$ \underline{Estimate of $E_\ve(F_2^{\rm h})$}

Again due to Lemma \ref{lem:parabolic-inhome}, one has
\begin{align*}
\|E_\ve(F_{2}^{\rm h})\|_{K_t}&\lesssim \|[G(\ve^\be a)\Delta_\ve
v^{\rm h}]_\Phi\|_{L^1_t(B^{\ga,\f12-\ga})} +\|[G(\ve^\be
a)\Delta_\ve v^{\rm h}]_\Phi\|_{L^1_t(B^{-\ga,\f12+\ga})},
\end{align*} which together with
Corollary \ref{col3.1} and Lemma \ref{lem:composition} ensures that
$$\longformule{
\|E_\ve(F_{2}^{\rm h})\|_{K_t}
 \lesssim
\ve^\be\|a_\Phi\|_{\widetilde{L}^\infty_t(B^{1,\f12})}\big(\|v^{\rm
h}_\Phi\|_{L^1_t(B^{2+\ga,\f12-\ga})} +\ve^2\|v^{\rm
h}_\Phi\|_{{L}^1_t(B^{\ga,\f52-\ga})}+\|v^{\rm
h}_\Phi\|_{L^1_t(B^{2-\ga,\f12+\ga})}}{{}+ \ve^2\|v^{\rm
h}_\Phi\|_{{L}^1_t(B^{-\ga,\f52+\ga})}\big)+\ve^{\be}\|a_\Phi\|_{\widetilde{L}^\infty_t(B^{1-\ga,\f12+\ga})}\big(\|v^{\rm
h}_\Phi\|_{{L}^1_t(B^{2,\f12})}+\ve^2\|v^{\rm
h}_\Phi\|_{{L}^1_t(B^{0,\f52})}\big).} $$ Whenever $\ve$ is so small
that $\ve^\be K\leq \epsilon$ for $\epsilon$ determined by
\eqref{k.50}. This gives rise to \ben\label{eq:vh-F2a}
\|E_\ve(F_{2}^{\rm h})\|_{K_t}\le
C\ve^{\be-2\al-2\ga}\Psi_1(t)\Psi_3(t). \een

$\bullet$ \underline{Estimate of $E_\ve(F_3^{\rm h})$}

In view of \eqref{p.3}, we get, by a similar proof of \eqref{l.8},
that \beno
\begin{split}
\|\D_k^{\rm h}\D_\ell^{\rm v}\na_{\rm
h}[q_{31}]_\Phi(t)\|_{L^2}\lesssim &\ve^{1-\al}2^\ell
\bigl\|\D_k^{\rm h}\D_\ell^{\rm v}[T^{\rm v}(v^3,v^{\rm
h})]_\Phi(t)\bigr\|_{L^2}\\
\lesssim &
d_{k\ell}2^\ell2^{-k\ga}2^{-\ell\bigl(\f12-\ga\bigr)}\ve^{1-\al}\|v^3_\Phi(t)\|_{B^{1,\f12}}\|v^{\rm
h}_\Phi\|_{\wt{L}^\infty_t(B^{\ga,\f12-\ga})},
\end{split}
\eeno and \beno \|\D_k^{\rm h}\D_\ell^{\rm v}\na_{\rm
h}[q_{31}]_\Phi(t)\|_{L^2}\lesssim d_{k\ell}2^\ell
2^{k\ga}2^{-\ell\bigl(\f12+\ga\bigr)}\ve^{1-\al}\|v^3_\Phi(t)\|_{B^{1,\f12}}\|v^{\rm
h}_\Phi\|_{\wt{L}^\infty_t(B^{-\ga,\f12+\ga})}, \eeno so that
applying Lemma \ref{lem:analytic} yields \beq\label{p.5}
\|E_\ve(\na_{\rm h}q_{31})\|_{K_t}\leq \f{C}\lam \Psi_2(t). \eeq

Similarly according to \eqref{p.4}, one gets, by using a similar
derivation of \eqref{l.10}, that \beno
\begin{split}
\|\D_k^{\rm h}\D_\ell^{\rm v}\na_{\rm
h}[q_{41}]_\Phi(t)\|_{L^2}\lesssim &\ve^{1-\al}2^{-k}2^\ell
\bigl\|\D_k^{\rm h}\D_\ell^{\rm v}[T^{\rm v}(v^3,\dive_{\rm h}v^{\rm
h})]_\Phi(t)\bigr\|_{L^2}\\
\lesssim & d_{k\ell}2^\ell
2^{-k\ga}2^{-\ell\bigl(\f12-\ga\bigr)}\ve^{1-\al}\|v^3_\Phi(t)\|_{B^{1,\f12}}\|v^{\rm
h}_\Phi\|_{\wt{L}^\infty_t(B^{\ga,\f12-\ga})},
\end{split}
\eeno and \beno \begin{split} \|\D_k^{\rm h}\D_\ell^{\rm v}\na_{\rm
h}[q_{41}]_\Phi(t)\|_{L^2}\lesssim
&\ve^{1-\al-2\ga}2^{-k(1-2\ga)}2^{\ell(1-2\ga)} \bigl\|\D_k^{\rm
h}\D_\ell^{\rm v}[T^{\rm v}(v^3,\dive_{\rm h}v^{\rm
h})]_\Phi(t)\bigr\|_{L^2}\\
\lesssim & d_{k\ell}2^\ell
2^{k\ga}2^{-\ell\bigl(\f12+\ga\bigr)}\ve^{1-\al-2\ga}\|v^3_\Phi(t)\|_{B^{1,\f12}}\|v^{\rm
h}_\Phi\|_{\wt{L}^\infty_t(B^{\ga,\f12-\ga})}, \end{split} \eeno so
that applying Lemma \ref{lem:analytic}  and using $1-\al\ge 3\ga,$
we get \beq\label{p.6} \|E_\ve(\na_{\rm h}q_{41})\|_{K_t}\leq
\f{C}\lam \Psi_2(t). \eeq

Let us examine $q_{53}.$ In order to do it, by using Bony's
decomposition \eqref{pd} for $\D_{\rm h}v^3G(\ve^\be a)$ in the
vertical variable, we write \beno
q_{53}=(-\Delta_\ve)^{-1}\pa_3T^{\rm v}({\Delta_{\rm
h}v^3},G(\ve^\be a))+(-\Delta_\ve)^{-1}\pa_3\cR^{\rm v}({\Delta_{\rm
h}v^3},G(\ve^\be a)). \eeno Note that Remark \ref{rem:product} and
Lemma \ref{lem:composition} ensures  \beno \|\Delta_k^{\rm
h}\Delta_\ell^{\rm v}[T^{\rm v}({\Delta_{\rm h}v^3},G(\ve^\be
a))]_\Phi(t)\|_{L^2} \lesssim
\ve^\be(d_{k}(t)d_\ell+d_{k\ell})2^{k(1-\ga)}2^{-\ell\bigl(\f12-\ga\bigr)}\|v^3_\Phi(t)\|_{B^{1+\ga,\f12-\ga}}\|a_\Phi\|_{\widetilde{L}^\infty_t({B}^{1,\f12})},
\eeno from which and a similar derivation of \eqref{p.5} and
\eqref{p.6}, we infer
\begin{align*}
\|E_\ve(\na_{\rm h}(-\Delta_\ve)^{-1}\pa_3T^{\rm v}({\Delta_{\rm
h}v^3},G(\ve^\be a))\|_{K_t}\le \f {C\ve^{\be-\ga}}
\lambda\Psi_1(t).
\end{align*}
Whereas by using Bony's decomposition \eqref{pd} for $\cR^{\rm
v}({\Delta_{\rm h}v^3},G(\ve^\be a))$ for the horizontal variables
and using $\dive v=0,$ one has \beno \bigl\|\p_3[\cR^{\rm
v}({\Delta_{\rm h}v^3},G(\ve^\be
a))]_\Phi\bigr\|_{L^1_t(B^{-1+\ga,\f12-\ga})}\lesssim
\ve^{\be}\|a_\Phi\|_{\widetilde{L}^\infty_t({B}^{1,\f12})}\|v^{\rm
h}_\Phi\|_{L^1_t(B^{2+\ga,\f12-\ga})}. \eeno Then applying
 Lemma \ref{lem:parabolic-inhome} gives
\begin{align*}
\|E_\ve(\na_{\rm h}(-\Delta_\ve)^{-1}\p_3\cR^{\rm v}({\Delta_{\rm
h}v^3},G(\ve^\be a))\|_{K_t}\lesssim &
\ve^{-2\ga}\bigl\|\p_3[\cR^{\rm v}({\Delta_{\rm h}v^3},G(\ve^\be
a))]_\Phi\bigr\|_{L^1_t(B^{-1+\ga,\f12-\ga})}\\
\lesssim &
\ve^{\be-2\ga}\|a_\Phi\|_{\widetilde{L}^\infty_t({B}^{1,\f12})}\|v^{\rm
h}_\Phi\|_{L^1_t(B^{2+\ga,\f12-\ga})}.
\end{align*}
Hence, thanks to Remark \ref{rem:pressure},
 for
$\ga\le \min\bigl(\f {1-3\al} 4, \f {\be-2\al} 2\bigr)$ and under
the assumption of \eqref{epsilon}, we deduce that
\begin{align}
\|E_\ve(F_3^{\rm h})\|_{K_t}\le C\Bigl(\f 1
\lambda+\max\Bigl(\ve^{\be-2\al-2\ga}, \ve^{1-3\al-4\ga},
K\ve^{\be-2\al-\ga}\Bigr) \Psi(t)\Bigr)\Psi(t).\label{eq:vh-F3a}
\end{align}

In view of \eqref{eq:velocity}, by summing up
(\ref{eq:vhom-a})--(\ref{eq:vh-F3a}), we arrive at
\ben\label{eq:Psi2-vh} \|v^{\rm h}\|_{K_t}\le
C\|v_0\|_{X_3}+C\biggl(\f 1
\lambda+\max\Bigl(\ve^{\be-2\al-2\ga},\ve^{1-3\al-4\ga},
K\ve^{\be-2\al-\ga}\Bigr)\Psi(t)\biggr)\Psi(t). \een

\no{\bf Step 2.2}\ The estimate of the vertical velocity.

Since $\ve$ satisfies $\ve^\be K\leq \epsilon,$ applying Lemma
\ref{lem:parabolic-inhome} gives \beno\begin{split}
\|E_\ve(F_{2}^3)\|_{K_t}\lesssim &
\|[F_2^3]_\Phi\|_{L^1_t(B^{\ga,\f12-\ga})}+\|[F_2^3]_\Phi\|_{L^1_t(B^{-\ga,\f12+\ga})}\\
\lesssim &
\ve^\be\Bigl(\|a_\Phi\|_{\wt{L}^\infty_t(B^{1,\f12})}\bigl(\|\D_\ve
v^3_\Phi\|_{L^1_t(B^{\ga,\f12-\ga})}+\|\D_\ve
v^3_\Phi\|_{L^1_t(B^{-\ga,\f12+\ga})}\bigr)\\
&\qquad\qquad\qquad\qquad\qquad\qquad+\|a_\Phi\|_{\wt{L}^\infty_t(B^{1-\ga,\f12+\ga})}\|\D_\ve
v^3_\Phi\|_{L^1_t(B^{0,\f12})}\Bigr),
\end{split}
\eeno which gives \beq \label{eq:v3-F2a}
\|E_\ve(F_{2}^3)\|_{K_t}\leq C\ve^{\be}\Psi_1(t)\Psi_3(t). \eeq

While again as $\ve^\be K\leq \epsilon,$ $2\al+2\ga<1,$ it follows
from Lemma \ref{lem:parabolic-inhome} and
 Proposition \ref{prop6.2} that
\begin{align}\label{eq:v3-F3a}
\|E_\ve(F_3^3)\|_{K_t}\le C\ve\|q\|_{Z_t}\le
C\ve^{1-2\al-\ga}\Psi(t)^2.
\end{align}
Finally note that $F_1^3$  \beno F_1^3=-\ve^{1-\al}(v^{\rm
h}\cdot\na_{\rm h}v^3)+\ve^{1-\al}(v^3\textrm{div}_{\rm h}v^{\rm
h}). \eeno Then we get, by using Lemma \ref{lem:parabolic-inhome}
and the law of product Corollary \ref{col3.1}, that
\begin{align*}
\ve^{1-\al}\|E_\ve(v^{\rm h}\cdot\na_{\rm h}v^3)\|_{K_t}&\lesssim
\ve^{1-\al}\Bigl(\|v^{\rm
h}_\Phi\|_{\widetilde{L}^\infty_t(B^{0,\f12})}
\big(\|v^3_\Phi\|_{{L}^1_t(B^{2+\ga,\f12-\ga})}
+\|v^3_\Phi\|_{{L}^1_t(B^{2-\ga,\f12+\ga})}\big)\\
&\qquad\qquad\quad+\bigl(\|v^{\rm
h}_\Phi\|_{\widetilde{L}^\infty_t(B^{\ga,\f12-\ga})}+\|v^{\rm
h}_\Phi\|_{\widetilde{L}^\infty_t(B^{-\ga,\f12+\ga})}\bigr)\|v^3_\Phi\|_{{L}^1_t(B^{2,\f12})}\Bigr),
\end{align*}
and
\begin{align*}
\ve^{1-\al}\|E_\ve(v^3\textrm{div}_{\rm h}v^{\rm
h})\|_{K_t}&\lesssim
\ve^{1-\al}\Bigl(\|v^3_\Phi\|_{\widetilde{L}^\infty_t(B^{0,\f12})}
\big(\|v^{\rm h}_\Phi\|_{{L}^1_t(B^{2+\ga,\f12-\ga})}
+\|v^{\rm h}_\Phi\|_{{L}^1_t(B^{2-\ga,\f12+\ga})}\big)\\
&\qquad\qquad\quad+\bigl(\|v^3_\Phi\|_{\widetilde{L}^\infty_t(B^{\ga,\f12-\ga})}+\|v^3_\Phi\|_{\widetilde{L}^\infty_t(B^{-\ga,\f12+\ga})}\bigr)\|v^{\rm
h}_\Phi\|_{{L}^1_t(B^{2,\f12})}\Bigr),
\end{align*}
which ensures \beno \|E_\ve(F_{1}^3)\|_{K_t}\le
C\ve^{1-3\al-2\ga}\Psi_2(t)\Psi_3(t), \eeno from which and
(\ref{eq:v3-F2a}),(\ref{eq:v3-F3a}), we achieve
\ben\label{eq:Psi2-v3} \|v^3\|_{K_t}\le
C\bigl(\|v_0\|_{X_3}+\max\bigl(\ve^{\be},\ve^{1-3\al-2\ga}\bigr)\Psi^2(t)\bigr).
\een

Therefore since Lemma \ref{lem3.1} implies \beno
\|f_\Phi\|_{\widetilde{L}^\infty_t(B^{0,\f12})}\le \|f\|_{K_t},
\eeno by combining (\ref{eq:Psi2-vh}) with (\ref{eq:Psi2-v3}), we
conclude that \beq \label{8.12} \Psi_2(t)\le
C\biggl(\|v_0\|_{X_3}+\f 1 \lambda
\Psi(t)+\max\Bigl(\ve^{\be-2\al-2\ga},\ve^{1-3\al-4\ga},K\ve^{\be-2\al-\ga}\Bigr)\Psi^2(t)\biggr).
\eeq

\no{\bf Step 3.} Estimate of $\Psi_3(t)$

Let
\begin{align*}
\|f\|_{L_t}\eqdefa&
\|f_\Phi\|_{{L}^1_t(B^{2+\ga,\f12-\ga})}+\|f_\Phi\|_{{L}^1_t(B^{2-\ga,\f12+\ga})}
+\ve^{2}\|f_\Phi\|_{{L}^1_t(B^{\ga,\f52-\ga})}+\ve^{2}\|f_\Phi\|_{{L}^1_t(B^{-\ga,\f52+\ga})}.
\end{align*}
Then we deduce from Lemma \ref{lem:parabolic-home} that
\begin{align}\label{eq:vhhom-c}
\|e^{t\Delta_\ve}v_0\|_{L_t}\le C\|v_0\|_{X_3}.
\end{align}
Whereas applying Lemma \ref{lem:parabolic-inhome} gives \beno
\begin{split}
\ve^{2\al+2\ga}\|E_\ve(\ve^{1-\al}\na_{\rm h}\cdot(v^{\rm h}\otimes
v^{\rm h}))\|_{L_t}\lesssim \ve^{1+\al+2\ga}\Bigl(\|[v^{\rm
h}\otimes& v^{\rm h}]_\Phi\|_{L^1_t(B^{1+\ga,\f12-\ga})}\\
&+\|[v^{\rm h}\times v^{\rm
h}]_\Phi\|_{L^1_t(B^{1-\ga,\f12+\ga})}\Bigr),
\end{split}
\eeno and \beno
\begin{split}
\ve^{2\al+2\ga}\|E_\ve(\ve^{1-\al}\p_3(v^{\rm h}
v^{3}))\|_{L_t}\lesssim & \ve^{\al+2\ga}\Bigl(\|[v^{\rm
h}v^3]_\Phi\|_{L^1_t(B^{1+\ga,\f12-\ga})}
+\|[v^{\rm h}v^3]_\Phi\|_{L^1_t(B^{1-\ga,\f12+\ga})}\Bigr)\\
&+\ve^{1+\al+2\ga}\Bigl(\|[v^{\rm
h}v^3]_\Phi\|_{L^1_t(B^{\ga,\f32-\ga})}+\|[v^{\rm
h}v^3]_\Phi\|_{L^1_t(B^{-\ga,\f32+\ga})}\Bigr),
\end{split}
\eeno so that we get, by applying the law of product of Corollary
\ref{col3.1}, that \beno
\begin{split}
\ve^{2\al+2\ga}\|E_\ve(F_1^{\rm h})\|_{L_t}
 \lesssim & \ve^{1+\al+2\ga}\biggl(\bigl(\|v^{\rm
h}_\Phi\|_{L^1_t(B^{2+\ga,\f12-\ga})}
+\|v^{\rm h}_\Phi\|_{L^1_t(B^{2-\ga,\f12+\ga})}\bigr)\|v^{\rm h}_\Phi\|_{\widetilde{L}^\infty_t(B^{0,\f12})}\\
&\qquad\qquad\
+\|v^3_\Phi\|_{\widetilde{L}^2_t(B^{1,\f12})}\bigl(\|v^{\rm
h}_\Phi\|_{\widetilde{L}^2_t(B^{\ga,\f32-\ga})}+\|v^{\rm
h}_\Phi\|_{\widetilde{L}^2_t(B^{-\ga,\f32+\ga})}\bigr)\\
& \qquad\qquad\ +\|v^{\rm
h}_\Phi\|_{\widetilde{L}^2_t(B^{1,\f12})}\bigl(\|v^3_\Phi\|_{\widetilde{L}^2_t(B^{\ga,\f32-\ga})}
+\|v^3_\Phi\|_{\widetilde{L}^2_t(B^{-\ga,\f32+\ga})}\bigr)\biggr)\\
 &
+\ve^{\al+2\ga}\biggl(\|v^3_\Phi\|_{\widetilde{L}^2_t(B^{1,\f12})}\bigl(\|v^{\rm
h}_\Phi\|_{\widetilde{L}^2_t(B^{1+\ga,\f12-\ga})}+\|v^{\rm
h}_\Phi\|_{\widetilde{L}^2_t(B^{1-\ga,\f12+\ga})}\bigr)\\
&\qquad\qquad\ +\|v^{\rm
h}_\Phi\|_{\widetilde{L}^2_t(B^{1,\f12})}\bigl(\|v^3_\Phi\|_{\widetilde{L}^2_t(B^{1+\ga,\f12-\ga})}
+ \|v^3_\Phi\|_{\widetilde{L}^2_t(B^{1-\ga,\f12+\ga})}\bigr)\biggr).
\end{split} \eeno
Due to \eqref{p.1}, we arrive at \beq \label{eq:vh-cF1-3}
\ve^{2\al+2\ga}\|E_\ve(F_1^{\rm h})\|_{L_t}\leq
C\bigl(\ve^{1-\al}\Psi_2(t)\Psi_3(t)+\ve^\ga\Psi_4^2(t)\bigr). \eeq

By the same manner, we have
\begin{align*}
\|E_\ve(F_1^3)\|_{L_t}
\lesssim & \ve^{1-\al}\Bigl(\|[v^hv^3]_\Phi\|_{L^1_t(B^{1+\ga,\f12-\ga})}+\|[v^hv^3]_\Phi\|_{L^1_t(B^{1-\ga,\f12+\ga})}\\
&\qquad\qquad+\|[v^3\textrm{div}_{\rm h}v^{\rm
h}]_\Phi\|_{L^1_t(B^{\ga,\f12-\ga})} +\|[v^{3}\textrm{div}_{\rm
h}v^{\rm h}]_\Phi\|_{L^1_t(B^{-\ga,\f12+\ga})}\Bigr).
\end{align*}
Then applying the law of product of Corollary \ref{col3.1} yields
\beno
\begin{split}
\|E_\ve(F_1^3)\|_{L_t} \lesssim \ve^{1-\al}\Bigl(&\|v^{\rm
h}_\Phi\|_{\widetilde{L}^2_t(B^{1,\f12})}\bigl(\|v^3_\Phi\|_{\widetilde{L}^2_t(B^{1+\ga,\f12-\ga})}
+\|v^3_\Phi\|_{\widetilde{L}^2_t(B^{1-\ga,\f12+\ga})}\bigr)\\
&+\|v^3_\Phi\|_{\widetilde{L}^2_t(B^{1,\f12})}\bigl(\|v^{\rm
h}_\Phi\|_{\widetilde{L}^2_t(B^{1+\ga,\f12-\ga})}+\|v^{\rm
h}_\Phi\|_{\widetilde{L}^2_t(B^{1-\ga,\f12+\ga})}\bigr)\Bigr),
\end{split}
\eeno from which, we deduce that \beq \label{eq:v3-cF1-3}
\|E_\ve(F_1^3)\|_{L_t}\leq C\ve^{1-2\al-\ga}\Psi_4^2(t). \eeq

Similarly due to $\ve^\be K\leq\epsilon,$ it follows from  Lemma
\ref{lem:parabolic-inhome}, Lemma \ref{lem:composition} and
Corollary \ref{col3.1} that
\begin{align*}
\ve^{2\al+2\ga}\|E_\ve(F_2^{\rm h})\|_{L_t} \lesssim
&\ve^{2\al+2\ga}\Bigl(\|[G(\ve^\be a)\Delta_\ve v^{\rm
h}]_\Phi\|_{L^1_t(B^{\ga,\f12-\ga})}
+\|[G(\ve^\be a)\Delta_\ve v^{\rm h}]_\Phi\|_{L^1_t(B^{-\ga,\f12+\ga})}\Bigr)\\
\lesssim &
\ve^{\be+2\al+2\ga}\Bigl(\|a_\Phi\|_{\widetilde{L}^\infty_t(B^{1,\f12})}
\big(\|\Delta_\ve v^{\rm h}_\Phi\|_{L^1_t(B^{\ga,\f12-\ga})}+\|\Delta_\ve v^{\rm h}_\Phi\|_{L^1_t(B^{-\ga,\f12+\ga})}\big)\\
&\qquad\qquad\qquad\qquad\qquad\qquad\qquad\quad+\|a_\Phi\|_{\widetilde{L}^\infty_t(B^{1-\ga,\f12+\ga})}\|\Delta_\ve
v^{\rm h}_\Phi\|_{L^1_t(B^{0,\f12})}\Bigr),
\end{align*}
which gives
 \beq \label{eq:vh-cF2-3} \ve^{2\al+2\ga}\|E_\ve(F_2^{\rm
h})\|_{L_t}\leq C\ve^\be\Psi_1(t)\Psi_3(t). \eeq Along the same
line, we have \beno
\begin{split}
 \|E_\ve(F_2^3)\|_{L_t}\lesssim
& \ve^{\be}\Bigl(\|a_\Phi\|_{\widetilde{L}^\infty_t(B^{1,\f12})}
\big(\|\Delta_\ve v^{3}_\Phi\|_{L^1_t(B^{\ga,\f12-\ga})}+\|\Delta_\ve v^{3}_\Phi\|_{L^1_t(B^{-\ga,\f12+\ga})}\big)\\
&\qquad\qquad\qquad\qquad\qquad\qquad\qquad\quad+\|a_\Phi\|_{\widetilde{L}^\infty_t(B^{1-\ga,\f12+\ga})}\|\Delta_\ve
v^{3}_\Phi\|_{L^1_t(B^{0,\f12})}\Bigr),
\end{split}
\eeno which implies
\begin{align}\label{eq:v3-cF2-3}
\|E_\ve(F_2^3)\|_{L_t}\le C\ve^\be\Psi_1(t)\Psi_3(t).
\end{align}

Finally since $2\al+2\be\leq 1,$  by applying Lemma
\ref{lem:parabolic-inhome} and Proposition \ref{prop6.2}, one has
\begin{align}\label{eq:v-cF3-3}
\ve^{2\al+2\ga}\|E_\ve(F_3^{\rm
h})\|_{L_t}+\|E_\ve(F_3^3)\|_{L_t}\le C\ve^{2\al+2\ga}\|
q\|_{Z_t}\le C\ve^{\ga}\Psi(t)^2.
\end{align}

Summing up (\ref{eq:vhhom-c})--(\ref{eq:v-cF3-3}), we conclude that
\ben\label{eq:Psi-3} \Psi_3(t)\le C\big(\ve^{2\al+2\ga}\|v^{\rm
h}\|_{L_t}+\|v^3\|_{L_t}\big)\le
C\bigl(\|v_0\|_{X_3}+\max\bigl(\ve^\ga,\ve^{1-2\al-\ga}\bigr)\Psi^2(t)\bigr).
\een Here we used Lemma \ref{lem3.1} so that \beno
\|f_\Phi\|_{{L}^1_t(B^{2,\f12})}+\ve^2\|f_\Phi\|_{{L}^1_t(B^{0,\f52})}\le
C\|f\|_{L_t}. \eeno

\no{\bf Step 4.} Estimate of $\Psi_4(t)$

Finally it is easy to observe from \eqref{cheminl} and \eqref{p.1}
that \beno \Psi_4(t)\leq \Psi_2^{\f12}(t)\Psi_3^{\f12}(t)\leq
\f12\bigl(\Psi_2(t)+\Psi_3(t)\bigr), \eeno which together with
\eqref{eq:Psi-1}, \eqref{8.12} and \eqref{eq:Psi-3}  leads to
 Proposition
\ref{prop:Psi}.\ef

\medskip

\noindent {\bf Acknowledgments.} Ping Zhang would like to thank
Professor Jean-Yves Chemin for profitable discussion on this topic.
Part of this work was done when we were visiting Morningside Center
of the Academy of Mathematics and Systems Sciences, CAS. We
appreciate the hospitality and the financial support from MCM. P.
Zhang is partially supported by NSF of China under Grant   11371347,
the fellowship from Chinese Academy of Sciences and innovation grant
from National Center for Mathematics and Interdisciplinary Sciences.
Z. Zhang is partially supported by NSF of China under Grant
11371037, Program for New Century Excellent Talents in University
and Fok Ying Tung Education Foundation.

\end{document}